\documentclass[a4paper,11pt,leqno]{article}
\usepackage{amsmath,amsfonts,amsthm,amssymb}
\usepackage[left=3cm, right=3cm, top=3cm, bottom=3.5cm]{geometry}
\usepackage{graphicx}
\usepackage{mathrsfs}
\usepackage{tikz}
\usepackage{marginnote}
\numberwithin{equation}{section}
\usepackage[final]{pdfpages}

\theoremstyle{plain}
\newtheorem{theorem}{Theorem}[section]

\newtheorem{lemma}[theorem]{Lemma}
\newtheorem{proposition}[theorem]{Proposition}

\theoremstyle{remark}
\newtheorem*{remark}{Remark}

\theoremstyle{definition}

\newcommand{\R}{\mathbb{R}}

\newcommand{\Z}{\mathbb{Z}}

\newcommand{\SF}{\mathbb{S}}
\newcommand{\eps}{\epsilon}

\usepackage{color}

\begin{document}

\title{Periodic motions for multi-wells potentials and \\
layers dynamic for the vector Allen-Cahn equation}
\author{{ Giorgio Fusco\footnote{Dipartimento di Matematica Pura ed Applicata, Universit\`a degli Studi dell'Aquila, Via Vetoio, 67010 Coppito, L'Aquila, Italy; e-mail:{\texttt{fusco@univaq.it}}}}
}

\date{}
\maketitle

\begin{abstract}

We consider a nonnegative potential $W$ that vanishes on a finite set and study the existence of periodic orbits of the equation
\[\ddot{u}=W_u(u),\;\;t\in\R,\]
that have the property of visiting neighborhoods of zeros of $W$ in a given finite sequence. We give conditions for the existence of such orbits. After introducing the new variable $x=\epsilon t$, $\epsilon>0$ small, these orbits correspond to stationary solutions of the parabolic equation
\[u_t=u_{xx}-W_u(u),\;\;x\in(0,1),\;t>0,\]
with periodic boundary conditions.

In the second paper of the paper we study solutions of this equation that, as the stationary solutions, have a layered structure.
We derive a system of ODE that describes the dynamics of the layers and show that their motion is extremely slow.
\end{abstract}
\section{Introduction}\label{1}\noindent

Let $W:\R^m\rightarrow\R$ be a nonnegative $C^4$ potential that satisfies
\[0=W(a)<W(u),\;a\in A,\;u\in\R^m\setminus A,\]
where $A\subset\R^m$ is a finite set with at least two distinct elements. We consider the mechanical system
\begin{equation}
u^{\prime\prime}=W_u(u),\;W_u(u)=(\frac{\partial W}{\partial u_1}(u),\ldots,\frac{\partial W}{\partial u_m}(u))^\top,
\label{newton}
\end{equation} and, given a small number $\delta>0$ and
 $N\geq 2$ points $a_1,\ldots,a_N\in A$ with $a_j\neq a_{j+1}$, $j=1,\ldots,N-1$ and $a_N\neq a_1$ we study the existence of periodic solutions $u^T:\R\rightarrow\R^m$ of \eqref{newton} that satisfies
\begin{equation}\label{nearaj}
 u^T(t_j)\in B_\delta(a_j),\;\;j=1,\ldots,N,
\end{equation}
for some $t_1<\ldots<t_N\in[0,T)$, $T>0$ the period of $u^T$.

We also require that $u^T$ remains away from $A\setminus\{a_1,\ldots,a_N\}$. In Fig. \ref{fig1} we sketch a situation where $A$ has four elements and $N=6$.

\begin{figure}
  \begin{center}
\begin{tikzpicture}
\draw [](1.4,0)--(6.6,0);
\draw [](4,-1.6)--(4,1.6);

\draw [blue, thick](4,0).. controls (4.3,2.5) and(5,.2)..(6,.1);
\draw [blue, thick](4,0).. controls (4.3,-2.5) and(5,-.2)..(6,-.1);

\draw [blue, thick](4,0).. controls (3.7,2.5) and(3,.2)..(2,.1);
\draw [blue, thick](4,0).. controls (3.7,-2.5) and(3,-.2)..(2,-.1);

\draw [blue] [->](5,.79)--(4.9,.89);
\draw [blue] [<-](5,-.79)--(4.9,-.89);

\draw [blue] [->](3,.79)--(3.1,.89);
\draw [blue] [<-](3,-.79)--(3.1,-.89);

\draw [blue, thick](2,.1) arc [radius=.1, start angle=90, end angle=270];
\draw [blue, thick](6,-.1) arc [radius=.1, start angle=-90, end angle=90];





\draw[red, fill] (4.,-1.5) circle [radius=0.05];
\draw[red, fill] (4.,1.5) circle [radius=0.05];
\draw[red, fill] (1.6,0) circle [radius=0.05];
\draw[red, fill] (6.4,0) circle [radius=0.05];

 \node at (6.8,0){$a$};
  \node at (1.2,0){$c$};
   \node at (4,1.9){$b$};
    \node at (4,-1.9){$d$};



\end{tikzpicture}
\end{center}

\caption{ $A=\{a,b,c,d\}$,\;$N=6$ and $a_1=a,a_2=b,a_3=d,a_4=c,a_5=b,a_6=d$.}
\label{fig1}
\end{figure}
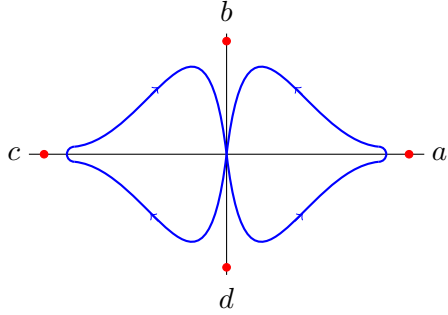

If $A=\{a_-,a_+\}$ with $a_-\neq a_+$ the problem is well understood, see for example \cite{as}, \cite{amz} and under minimal assumption on $W$, there is a $T_0>0$ such that, for each $T\geq T_0$, there exists a $T$-periodic solution $u^T$ that oscillates twice for period on the same trajectory with extreme at $u^T(0)\in B_\delta(a_-)$ and $u^T(T/2)\in B_\delta(a_+)$ and
\[\lim_{T\rightarrow+\infty}\delta=0.\] Moreover the speed  vanishes at $t=0$ and $t=T/2$. For this reason $u^T$ is called a \emph{brake} orbit. The existence of brake orbits that oscillate between neighborhoods of two minima of the potential has been extended also to an infinite
dimensional setting where $W$ is replaced by a functional $J:\mathcal{H}\rightarrow\R$  ($\mathcal{H}$ a suitable function space) with two distinct global minimizers $\bar{u}_-\neq\bar{u}_+\in\mathcal{H}$ that correspond to the minima $a_-,a_+$ of $W$, see \cite{am}, \cite{fgn} and \cite{p}.

The case $\sharp A\geq 3$ is open and its analysis in full generality is probably very difficult. We study the problem under generic assumptions that we now describe. We note that if a periodic solutions $u^T$ with the properties listed above exists for all $\delta>0$ small
we can expect that there is $t_j^*\in(t_j,t_{j+1})$ such that, at least along a subsequence, it results
\begin{equation}
\lim_{T\rightarrow+\infty}u^T(t-t_j^*)=\bar{u}_j(t),
\label{u-baru}
\end{equation}
where $\bar{u}_j$ is a heteroclinic solution of \eqref{newton}
 that connects $a_j$ to $a_{j+1}$. Therefore a natural assumption in the study of the above question is
\begin{description}
\item[H$_1$] For each $j=1,\ldots,N$ there exists a solution $\bar{u}_j$ of (\ref{newton}) which is a minimizer of the action $J(u)=\int_\R\Big(\frac{1}{2}\vert u^\prime\vert^2
    +W(u)\Big)ds$ in the class of maps that connect $a_j$ to $a_{j+1}$ ($a_{N+1}=a_1$).
\end{description}
\begin{remark}
\label{het?} $A=\{a_-,a_+\}$, $a_-\neq a_+$ implies the existence of  a solution $\bar{u}$ of \eqref{newton} that connects $a_-$ to $a_+$ \cite{amz}, \cite{monteil}, \cite{R}, \cite{s}, \cite{ZS}.
If $A$ contains more than two elements, the existence of a heteroclinic solution $\bar{u}$ that connects $a_-\in A$ to $a_+\in A$, $a_-\neq a_+$ is not guaranteed. A sufficient condition for $\bar{u}$ to exist it that
\[\sigma_{a_-,a_+}<\sigma_{a_-,a}+\sigma_{a,a_+},\;\;a\in A\setminus\{a_-,a_+\},\]
where, for $a\neq b\in A$,
\[\sigma_{a,b}=\inf_{u\in\mathscr{A}_{a,b}}J(u),\quad\mathscr{A}_{a,b}=\{u\in W_{\mathrm{loc}}^{1,2}:\,u\,\;\text{connects}\;\,a\,\;\text{to}\;\,b\}.\]
\end{remark}
We also assume the following nondegeneracy conditions:
\begin{description}
\item[H$_2$]  The hessian matrix $W_{uu}(a_j)$, $j=1,\ldots,N$ has distinct positive eigenvalues.
\item[H$_3$] The zero eigenvalue of the operator ${L}_j: W^{2,2}(\R; \R^m) \rightarrow L^2(\R; \R^m)$
 defined by
 \[{L}_j v= -v^{\prime \prime}+W_{uu}(\bar{u}_j) v,\]
 is simple.
\end{description}
\begin{remark}
\label{0-eign}
For each $s_0\in\R$, $\bar{u}_j(\cdot-s_0)$ is a solution of \eqref{newton}. Therefore by differentiating with respect to $s_0$ we get $\bar{u}_j^{\prime\prime\prime}=W_{uu}(\bar{u}_j)\bar{u}_j^\prime$. That is $0$ is an eigenvalue of ${L}_j$ and $\bar{u}_j^\prime$ a corresponding eigenvector.
\end{remark}
The smoothness of $W$ and {\bf{H}$_2$} imply, see Lemma \ref{w}, the existence of the limits
\begin{equation}
\lim_{s\rightarrow-\infty}\frac{\bar{u}_j(s)-a_j}{\vert\bar{u}_j(s)-a_j\vert}=z_j^-,\quad\quad
\lim_{s\rightarrow\infty}\frac{\bar{u}_{j-1}(s)-a_j}{\vert\bar{u}_{j-1}(s)-a_j\vert}=z_j^+,
\label{lim-z+-z-}
\end{equation}
where $z_j^\pm$ are eigenvectors of the matrix $W_{uu}(a_j)$. We define $\bar{u}_j$ also for $j=N+1$ and $s<0$ and for $j=0$ and $s>0$ by setting
\[\begin{split}
&\bar{u}_{N+1}(s)-a_{N+1}=\bar{u}_1(s)-a_1,\;\;s<0,\\
&\bar{u}_0(s)-a_1=\bar{u}_N(s)-a_{N+1},\;\;s>0,
\end{split}\]
and observe that these definitions imply
\[z_1^+=z_{N+1}^+,\;\;z_{N+1}^-=z_1^-.\]
Note that from {\bf{H}$_2$} the eigenvectors $z_j^-$ and $z_j^+$ are either parallel or orthogonal.

We prove the following
\begin{theorem}
\label{CNS}
Assume {\bf{H}$_1$}-{\bf{H}$_3$} and
\begin{description}
\item[H$_4$] $z_j^-=\pm z_j^+,\quad\quad j=1,\ldots,N,$
\end{description}
Then, the condition
\begin{equation}
z_j^-\cdot z_j^+=z_{j+1}^-\cdot z_{j+1}^+,\;\;j=1,\ldots,N,
\label{sigmaproduct0}
\end{equation}
is necessary and sufficient for the existence of a family $u^T:\R\rightarrow\R^m$, $T\geq T_0$, for some $T_0>0$ of $T$-periodic solutions of \eqref{newton} that satisfy:
\begin{enumerate}
\item There exist numbers $0\leq \xi_1<\xi_2<\ldots<\xi_N<1$ such that
\[\begin{split}
&\lim_{T\rightarrow+\infty}\xi_j-\xi_{j-1}=\frac{1/\mu_j}{\sum_h1/\mu_h},\;\; (\xi_0=\xi_N-1)\\
&\quad\text{and}\\
&\lim_{T\rightarrow+\infty}u^T(t+\xi_jT)=\bar{u}_j(t),
\end{split}\]
uniformly in compacts.
\item Given $\delta>0$ there are $t_\delta>0$ and $T_\delta>0$, $t_\delta$ independent of $T$, such that
\[\begin{split}
&u^T(t)\in B_\delta(a_j),\;\;\text{for}\;t\in(\xi_{j-1}T+t_\delta,\xi_jT-t_\delta),\\
& j=1,\ldots,N,\;\;T\geq T_\delta.
\end{split}\]
\item There exist $k_0,K_0$ independent of $T$ such that
\[\begin{split}
&\|u^T-u^{T,\xi}\|_{W^{1,2}((0,T);\R^m)}\leq K_0e^{-k_0 T}e^{-\frac{T}{2\sum_h\frac{1}{\mu_h}}},\\
&\vert J_{(0,T)}(u^T)-J_{(0,T)}(u^{T,\xi})\vert\leq K_0e^{-k_0 T}e^{-\frac{T}{2\sum_h\frac{1}{\mu_h}}}.
\end{split}\]
where $u^{T,\xi}$ is defined in \eqref{uTxi} and $J_{(0,T)}(v)=\int_0^T\Big(\frac{\vert v^\prime\vert^2}{2}+W(v)\Big)dt$.
\end{enumerate}
\end{theorem}
In the scalar case $m=1$ the condition $z_j^-=\pm z_j^+$  is automatically satisfied. In the vector case $m>1$, {\bf{H}$_4$} should be regarded as a generic assumption. Indeed it can be shown that if $\bar{u}$ is a solution of \eqref{newton} that converges to $a\in A$, for $s\rightarrow-\infty$ or for $s\rightarrow\infty$ then, generically, the ratio $\frac{\bar{u}(s)-a}{\vert\bar{u}(s)-a\vert}$ converges to an eigenvector of $W_{uu}(a)$ that corresponds to the smallest eigenvalue of $W_{uu}(a)$. A situation where the condition $z_j^-=\pm z_j^+$ may automatically hold in the vector case is when $W$ is invariant under a group of symmetries. For instance this is the case when $W:\R^m\rightarrow\R$ is
 invariant under the Dihedral group $D_K$ of the symmetries of a regular polygon with $K$ sides and $A=\{W=0\}$ has exactly $K$ distinct elements. The situations corresponding to $z_j^-\cdot z_j^+=\pm 1$ are illustrated in Fig.\ref{fig2}.

In the scalar case $m=1$, in agreement with phase plane analysis, periodic solutions whose existence can the established via Theorem \ref{CNS} are necessarily of brake type oscillating between neighborhoods of consecutive elements of $A$. This follows from the fact that, since $A\subset\R$ is ordered, the necessary condition \eqref{sigmaproduct0} can not be satisfied if the sequence $a_1,\ldots,a_N$ includes three or more distinct elements of $A$. Assume for example that $W(u)=\frac{1}{4}u^2(1-u^2)^2$. Then $A=\{-1,0,1\}$ and $W^{\prime\prime}(\pm1)=2$, $W^{\prime\prime}(0)=\frac{1}{2}$. Phase plane analysis shows that there exist heteroclinic connections
$\bar{u}$ and $\bar{v}$ connecting $-1$ to $0$ and $0$ to $1$ respectively.
For $N=4$ consider the sequence $a_1=-1$, $a_2=0$, $a_3=1$ and $a_4=0$ with the associated sequence of connections
 $\bar{u}_1=\bar{u}$, $\bar{u}_2=\bar{v}$, $\bar{u}_3=\bar{v}(-\cdot)$ and $\bar{u}_4=\bar{u}(-\cdot)$. In this case we have in particular
 \[\begin{split}
 &z_2^-=1,\quad z_2^+=-1,\\
 &z_3^-=z_3^+=-1.
 \end{split}\]
 It follows
 \[z_2^-\cdot z_2^+\neq z_3^-\cdot z_3^+\]

 This shows that the necessary and sufficient condition formulated in Theorem \ref{CNS} is not satisfied and we conclude that the periodic orbit $u^T$ does not exist and indeed inspection of the phase plane confirms that there is no periodic orbit of (\ref{newton}) that visit in sequence $\delta$-neighborhoods of $a_1,a_2,a_3,a_4$. Note also that this example shows that the  product $\varsigma_j:=z_j^-\cdot z_j^+$ can assume both values $1$ and $-1$.

\begin{figure}
  \begin{center}
\begin{tikzpicture}
\draw [blue, thick][->](3.5,0)--(6,0);
\draw [blue, thick][->](10.5,0)--(13,0);

\draw[red, fill] (3.5,0) circle [radius=0.05];
\draw[red, fill] (10.5,0) circle [radius=0.05];

\draw [blue, thick](3.5,0) arc [radius=5, start angle=-90, end angle=-60];
\draw [blue, thick](3.5,0) arc [radius=5, start angle=90, end angle=60];

\draw [blue, thick](10.5,0) arc [radius=5, start angle=-90, end angle=-60];
\draw [blue, thick](10.5,0) arc [radius=5, start angle=-90, end angle=-120];

\draw [blue] [->](5.8,.56)--(5.9,.62);
\draw [blue] [<-](5.8,-.56)--(5.9,-.62);

\draw [blue] [->](12.8,.56)--(12.9,.62);
\draw [blue] [->](8.15,.6)--(8.25,.54);

 \node at (6.8,0){$z_j^-=z_j^+$};
 \node at (13.95,0){$z_j^-=-z_j^+$};
  \node at (5.3,-.7){$\bar{u}_{j-1}$};
   \node at (5.3,.7){$\bar{u}_j$};
    \node at (3.5,-.5){$a_j$};
     \node at (10.5,-.5){$a_j$};

  \node at (12.3,.7){$\bar{u}_j$};
   \node at (8.7,.7){$\bar{u}_{j-1}$};

\node at (3.5,-2){a)};
\node at (10.5,-2){b)};

\end{tikzpicture}
\end{center}

\caption{a) $z_j^-\cdot z_j^+=1$;\quad b) $z_j^-\cdot z_j^+=-1$.}
\label{fig2}
\end{figure}
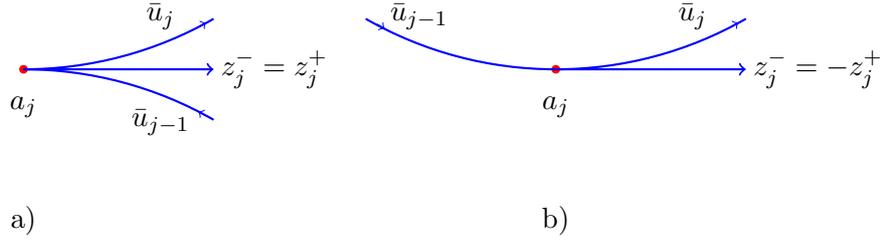

Assume now $m\geq 2$ and that $W:\R^m\rightarrow\R$ satisfies $0=W(a)<W(u)$ for $a\in\{a_1,a_2\}$, $u\in\R^m\setminus\{a_1,a_2\}$ with $a_1\neq a_2$ nondegenerate zeros of $W$ and assume that there is a nondegenerate connection $\bar{u}$ connecting $a_1$ to $a_2$. Consider the sequence $a_1$, $a_2$ ($N=2$) and the associate sequence of connections $\bar{u}_1=\bar{u}$, $\bar{u}_2=\bar{u}(-\cdot)$. In this case  it results
\[z_1^-=z_1^+,\quad z_2^-=z_2^+,\]
and Theorem \ref{CNS} implies the existence of a periodic orbit $u^T$ that, for $T$ sufficiently large  alternates between $\delta$-neighborhoods of $a_1$ and $a_2$.
Beside this simple situation that extends to the case $m\geq 2$ well known results \cite{cp1}, \cite{fh} valid for the scalar case,
in the vector case $m\geq 2$, depending on the sign of the products in \eqref{sigmaproduct0} and on the structure of the set of connections between the elements of $A$, various type of periodic solutions of \eqref{newton} can be deduced from
Theorem \ref{CNS}. Even in the case where $A=\{a_-,a_+\}$ has only two elements, if it happens \cite{albechen} that there are two
distinct orbits connecting $a_-$ to $a_+$, say $\bar{u}\neq\bar{v}$ that satisfy the condition
\begin{equation}
\lim_{s\rightarrow\pm\infty}\frac{\bar{u}(s)-a_\pm}{\vert \bar{u}(s)-a_\pm\vert}
=\lim_{s\rightarrow\pm\infty}\frac{\bar{v}(s)-a_\pm}{\vert \bar{v}(s)-a_\pm\vert}=z^\pm,
\label{cond-2}
\end{equation}
Theorem \ref{CNS} implies the existence of infinitely many periodic solutions of \eqref{newton} with different structure. Indeed \eqref{cond-2} implies that, independently of the path chosen to travel from $a_-$ to $a_+$ and viceversa, the value of the products in \eqref{sigmaproduct0} is always $1$. Therefore if we represent the connections $\bar{u},\bar{u}(-\cdot),\bar{v},\bar{v}(-\cdot)$ with the four symbols $+p,-p,+q,-q$, then Theorem \ref{CNS} implies the existence of a periodic solution for each sequence $\{\omega_h\}_{h=1}^N$ which has $N$ even and
satisfies
\begin{equation}
\omega_h\in\{ +p,-p,+q,-q\},\quad \text{sign}(\omega_h)\text{sign}(\omega_{h+1})=-1,\;\;h=1,\ldots.
\label{con-3}
\end{equation}
In this situation we can even conjecture that, as limits of sequences of periodic solutions corresponding to sequences $\{\omega_h\}_{h=1}^N$ with $N\rightarrow\infty$, there exist chaotic aperiodic solutions of \eqref{newton} associated to infinite sequences $\{\omega_h\}_{h=1}^\infty$ that satisfy \eqref{con-3}.
Similar considerations apply when the graph with vertices the elements of $A$ and edges the connections among them has a connected component that contains three or more elements of $A$. This happens, for instance, if $\sharp A$ is odd.

We denote ${\mu_j\pm}^2$ the eigenvalue of $W_{uu}(a_j)$ associated to eigenvectors $z_j^\pm$ in \eqref{lim-z+-z-}.
From {\bf{H}$_2$} and \eqref{lim-z+-z-} it follows, see Lemma \ref{w}, that there exist maps $\bar{w}_j^\pm$ and constants $\bar{K}_j^\pm>0$, $K_j^\pm>0$ and ${\hat{\mu}_j}^->{\mu_j}^-$, ${\hat{\mu}_j}^+>{\mu_{j+1}^+}$ such that
\begin{equation}\label{generic}
\begin{split}
&\bar{u}_j(s)-a_j=\bar{K}_j^-z_j^-e^{{\mu_j^-}s}+\bar{w}_j^-(s),\;\;s\leq 0,\\
&\bar{u}_j(s)-a_{j+1}=\bar{K}_j^+z_{j+1}^+e^{-{\mu_{j+1}^+}s}+\bar{w}_j^+(s),\;\;s\geq 0,\\\\
&\vert\bar{w}_j^-(s)\vert,\vert D^k\bar{w}_j^-(s)\vert\leq K_j^-e^{{\hat{\mu}_j}^-s},\;\;s\leq 0,\;k=1,2,3,\\
&\vert\bar{w}_j^+(s)\vert,\vert D^k\bar{w}_j^+(s)\vert\leq K_j^+e^{-{\hat{\mu}_j}^+s},\;\;s\geq 0,\;k=1,2,3.
\end{split}
\end{equation}
If also {\bf{H}$_4$} is satisfied \eqref{generic}$_1$ becomes
\[\begin{split}
&\bar{u}_j(s)-a_j=\bar{K}_j^-z_j^-e^{{\mu_j}s}+\bar{w}_j^-(s),\;\;s\leq 0,\\
&\bar{u}_j(s)-a_{j+1}=\bar{K}_j^+z_{j+1}^+e^{-{\mu_{j+1}}s}+\bar{w}_j^+(s),\;\;s\geq 0,\\\\
\end{split}\]
where $\mu_j^2$ is the eigenvalue of $W_{uu}(a_j)$ associated to the eigenvectors $z_j^-$ and $z_j^+$.

The exponential decay of $\bar{u}_j$ in \eqref{generic} suggests that by patching together translates of the maps $\bar{u}_j-a_j$ and $\bar{u}_j-a_{j+1}$, $j=1,\ldots,N$ we can construct an $N$-dimensional manifold of periodic maps that are approximate solutions of (\ref{newton}) in the sense that satisfy (\ref{newton}) up to an exponentially small error that depends on the size of the translations. The idea is then to construct exact periodic solutions of (\ref{newton}) by small perturbations of maps on the approximate manifold. Since the period of the sought periodic solutions is unknown we prefer to rescale equation (\ref{newton}) by the change of variable $t=\frac{x}{\epsilon}$, $\epsilon>0$ small and look for periodic solutions of period $1$ for the rescaled equation
\begin{equation}
\epsilon^2u_{xx}=W_u(u).
\label{rescaled}
\end{equation}
Set $\mu_j=\frac{1}{2}(\mu_j^-+\mu_j^+)$. For $\rho>0$ small, we define
\begin{equation}
\Xi_\rho:=\Big{\{}\xi\in\R^N:\, 0\leq\xi_1,\;\, \xi_j-\xi_{j-1}>\frac{\rho}{\mu_j},\; j=1,\ldots,N,\;\, \xi_0=\xi_N-1<\xi_1\Big{\}}
\label{basic}
\end{equation}
and focus on the manifold $\mathcal{M}=\{u^\xi:\xi\in\Xi_\rho\}$ where $u^\xi$ is defined by
\begin{equation}
u^{\xi}(x)=\displaystyle\sum_{n<0}\sum_{h=1}^N \Big( \bar{u}_h( \frac{x-n-\xi_h}{\eps})-a_{h+1} \Big)+a_1
+\sum_{n\geq 0}\sum_{h=1}^N \Big( \bar{u}_h( \frac{x-n-\xi_h}{\eps})-a_h \Big).
\label{uxi}
\end{equation}
$\mathcal{M}$ is a manifold of $1$-periodic maps which, in each interval $(\hat{\xi}_{j-1}+n,\hat{\xi}_j+n)$, $n\in\Z$ ($\;\hat{\xi}_j=\frac{\xi_{j+1}+\xi_j}{2}$, $j=1,\ldots,N$;\, $\xi_{N+1}=\xi_1+1$), aside from an error of $\mathrm{O}(e^{-\frac{\alpha\rho}{\epsilon}})$, $\alpha>0$ a constant independent of $\epsilon$, coincide with $\bar{u}_j( \frac{x-n-\xi_j}{\eps})$. Since $\bar{u}_j( \frac{x-n-\xi_j}{\eps})$ is a solution of \eqref{rescaled} it follows, see also Lemma \ref{est-Fuxi}
\begin{equation}
\mathscr{F}(u^\xi):=\epsilon^2u_{xx}^\xi-W_u(u^\xi)=\mathrm{O}(e^{-\frac{\alpha\rho}{\epsilon}}).
\label{beta-error}
\end{equation}
In the following $\alpha>0$ and $C>0$ denote generic constants independent of $\epsilon>0$ and $\xi\in\Xi_\rho$.
\begin{remark}
If we set back $x=\frac{t}{T}$ with $T=\frac{1}{\epsilon}$ in the definition \eqref{uxi} of $u^\xi$, the rescaled
version of $u^\xi$
\begin{equation}
u^{T,\xi}(t)=\displaystyle\sum_{n<0}\sum_{h=1}^N \Big( \bar{u}_h(t-nT-\xi_hT)-a_{h+1} \Big)+a_1
+\sum_{n\geq 0}\sum_{h=1}^N \Big( \bar{u}_h(t-nT-\xi_hT)-a_h \Big),
\label{uTxi}
\end{equation}
is a $T$-periodic map and $\{u^{T,\xi}:\xi\in\Xi_\rho\}$ is a manifold of $T$-periodic maps.
\end{remark}
If we set $u=u^\xi+v$ in \eqref{rescaled} the problem of the existence of periodic solutions with the sought properties becomes the problem of finding $\xi\in\Xi_\rho$ and a $1$-periodic map $v:\R\rightarrow\R^m$ that solve the equation $\epsilon^2(u_{xx}^\xi+v_{xx})=W_u(u^\xi+v)$. The estimate \eqref{beta-error} suggests to rewrite this equation as a weakly nonlinear problem
\begin{equation}
\epsilon^2v_{xx}-W_{uu}(u^\xi)v=W_u(u^\xi)-\epsilon^2u_{xx}^\xi+N^\xi(v)=\mathrm{O}(e^{-\frac{\alpha\rho}{\epsilon}})+N^\xi(v),
\label{rescaled1}
\end{equation}
where $N^\xi(v)=W_u(u^\xi+v)-W_u(u^\xi)-W_{uu}(u^\xi)v$ is quadratic in $v$. The difficulty with this equation is that the linear operator $L^\xi v=-\epsilon^2v_{xx}+W_{uu}(u^\xi)v$ is almost singular. Indeed the $0$ eigenvalue of $L_j$, $j=1,\ldots,N$  manifests itself in the existence of $N$ exponentially small eigenvalues of $L^\xi$ with corresponding eigenspace $X^\xi$ approximately spanned by $u_{\xi_1}^\xi,\ldots,u_{\xi_N}^\xi$. More precisely we have  $X^\xi=\mathrm{Span}(\varphi_1,\ldots,\varphi_N)$ where  $\varphi_j=u_{\xi_j}^\xi/{\|u_{\xi_j}^\xi\|}+\mathrm{O}(e^{-\frac{\alpha\rho}{\epsilon}})$, $j=1,\ldots,N$ is an  orthonormal basis for $X^\xi$ and $\|\cdot\|$ is the norm in $L^2((0,1);\R^m)$. To overcome the problem of the singularity of $L^\xi$ on $X^\xi$, a natural approach is, in the spirit of the classical method of Lapunov-Smidth and bifurcation theory, to solve \eqref{rescaled1} in two steps. In the first step we solve the projection of \eqref{rescaled1} on the orthogonal complement $X^{\xi,\perp}$ of $X^\xi$:
\begin{equation}
-{L}^\xi v=\sum_j\varphi_j^\xi c_j(\xi,v)+N^\xi(v)+W_u(u^\xi)-\epsilon^2u_{xx}^\xi,
\label{quasiinv0}
\end{equation}
where $v\in X^{\xi,\perp}$ and $c(\xi,v)\in\R^N$ is a Lagrange multiplier determined by the condition that the right hand side of \eqref{quasiinv0} be in $X^{\xi,\perp}$. We find (Proposition \ref{quasiinv1}) that, for each $\xi\in\Xi_\rho$ there are unique $v^\xi\in X^{\xi,\perp}$, $v^\xi=\mathrm{O}(e^{-\frac{\alpha\rho}{\epsilon}})$ and $c(\xi)=c(\xi,v^\xi)\in\R^N$ that solve \eqref{rescaled1}. The second step in the proof of Theorem \ref{CNS} is the analysis of the bifurcation equation
\[c(\xi)=0.\]

Under the assumptions of Theorem \ref{CNS} and in particular if {\bf{H}$_4$} holds, we prove that (see Theorem \ref{TH-c} and  \eqref{c0-gen})

\begin{equation}
\begin{split}
&c_j(\xi)=\frac{2\epsilon^\frac{1}{2}}{\bar{q}_j}\Big(\varsigma_{j+1}k_{j+1}\mathscr{E}_{j+1}
-\varsigma_jk_j\mathscr{E}_j
+\mathrm{O}(e^{-\frac{\alpha\rho}{\epsilon}}\max_h\mathscr{E}_h\Big),\;\;j=1,\ldots,N,\\
&\mathscr{E}_j=e^{-\frac{{\mu_j}}{\eps}(\xi_j-\xi_{j-1})},
\end{split}
\label{scripC0}
\end{equation}
where $k_j=\mu_j^2\bar{K}_j^-\bar{K}_{j-1}^+$,  $\varsigma_j=z_j^-\cdot z_j^+$ and $\bar{q}_j^2=J(\bar{u}_j)=\int_\R\vert\bar{u}_j^\prime\vert^2ds$ is the energy of the connection $\bar{u}_j$.

On the basis of this estimate the necessity of \eqref{sigmaproduct0} is quite obvious. Indeed from \eqref{scripC0}, if $\mathscr{E}_j=\max_h\mathscr{E}_h$, $c_j(\xi)=0$ is possible only if
\[\varsigma_j=\varsigma_{j+1},\;\;\text{and}\;\;\mathscr{E}_{j+1}\approx\frac{k_j}{k_{j+1}}\mathscr{E}_j.\]
On the other hand if $\bar{\xi}$ is such that $\mathscr{E}_{j+1}=\frac{k_j}{k_{j+1}}\mathscr{E}_j$, the exponential structure of $\mathscr{E}_h$, $h=1,\ldots,N$, and an application of Brouwer fixed point theorem show, see Theorem \ref{per-exists}, that there is a small perturbation $\xi$ of $\bar{\xi}$ that solves the bifurcation equation completing the proof of Theorem \ref{CNS}.
\vskip.3cm
Next we focus on the dynamics of the parabolic equations
\begin{equation}
\begin{split}
&u_t=\mathscr{F}(u),\;\;x\in(0,1),\;u\in W^{1,2},\\
&\mathscr{F}(u)=\epsilon^2u_{xx}-W_u(u),
\end{split}
\label{parabolic0}
\end{equation}
(with periodic boundary conditions)
in a neighborhood of the manifold $\mathcal{M}=\{{u}^\xi:  \xi\in\Xi_\rho\}$.
We observe that with $v=v^\xi$, \eqref{quasiinv0} is actually equivalent to
\[\mathscr{F}(\hat{u}^\xi)=\sum_j\varphi_j^\xi c_j(\xi),\]
where we have set $\hat{u}^\xi=u^\xi+v^\xi$, $\xi\in\Xi_\rho$.
Therefore we can formally write
\[u_t\vert_{u=\hat{u}^\xi}=\sum_{j=1}^Nc_j(\xi)\varphi_j^\xi.\]
This indicates that $c(\xi)$ should play an important role in characterizing the dynamics of \eqref{parabolic0}  near $\mathcal{M}$ and, since maps near $\mathcal{M}$ have a layered structure, $c(\xi)$ should be related to layers dynamics. Our first result is based on the fact that \eqref{parabolic0} is the $L^2$-gradient system associated to the action functional
\begin{equation}
J_\eps(u)=\int_0^1\Big(\epsilon^2\vert u_x\vert^2+W(u)\Big)dx,\;\, u\in W^{1,2}((0,1);\R^N),
\label{eps-action}
\end{equation}
and depends on the geometric structure of the graph of $J_\eps$ near $\mathcal{M}$. We decompose $u$ near $\mathcal{M}$ in the form (see Lemma \ref{Projct})
\[\begin{split}
&u={u}^\xi+w,\\
&\langle w,u_{\xi_j}^\xi\rangle=0,\;\;j=1,\ldots,N.
\end{split}\]
 By applying Theorem 2.1  in \cite{BFK}, we prove that, if the initial condition $u_0=u^{\xi_0}+w_0$ is sufficiently close to $\mathcal{M}$
and if the difference $J_\eps(u_0)-\sup_{\xi\in\Xi_\rho}J_\eps(u^\xi)$ is sufficiently small, then the solution $u(t,u_0)$ of \eqref{parabolic0} creeps along $\mathcal{M}$ for a very long time. We use the weighted $W^{1,2}$ norm \[\|v\|^2_{W_\eps^{1,2}}=\|v\|^2+\epsilon^2\|v_x\|^2.\]
\begin{theorem}
\label{stay-near}
Assume $\bf{H}_1-\bf{H}_3$ and assume that $u_0={u}^{\xi_0}+w_0$, $\xi_0\in\Xi_\rho$,  satisfies
\begin{enumerate}
\item \[\begin{split}
&\langle w_0,u_{\xi_j}^{\xi_0}\rangle=0,\;\;j=1,\ldots,N,\\
&\|w_0\|_{W_\eps^{1,2}}=\mathrm{O}(\epsilon^\frac{3}{2}),\\
&\inf_{\zeta\in\partial\Xi_\rho}\|u^{\xi_0}-u^\zeta\|=d,
\end{split}\]
with $d>0$ independent of $\epsilon$.
\item
\[J_\eps(u_0)-\sup_{\xi\in\Xi_\rho}J_\eps({u}^\xi)=\mathrm{O}(e^{-\frac{\alpha\rho}{\eps}}).\]
\end{enumerate}
Then there is $T>0$ such that the solution $t\rightarrow u(t,u_0)$ of \eqref{parabolic0} through $u_0$ can be decomposed in the form
\[\begin{split}
&u(t,u_0)={u}^{\xi(t)}+w(t),\;\;t\in[0,T),\\
&\langle w(t),u_{\xi_j}^{\xi(t)}\rangle=0,\;\;j=1,\ldots,N,
\end{split}\]
 and
\[\|w(t)\|_{W_\eps^{1,2}}=\mathrm{O}(e^{-\frac{\alpha\rho}{\eps}}),\;\;t\in[0,T).\]
Moreover
either $T=+\infty$ and $\xi(t)\in\Xi_\rho$ for all $t>0$ or $T=\mathrm{O}(e^{\frac{\alpha\rho}{\eps}}d^2)$ and $\xi(T)\in\partial\Xi_\rho$.
\end{theorem}

Based on the exponential estimate for $w(t)$ in Theorem \ref{stay-near} and on some ideas from \cite{cp} we prove
that indeed $c(\xi)$ determines layers dynamics. As before we let $u(t,u_0)=u^{\xi(t)}+w(t)$, $\langle w(t),u_{\xi_j}^{\xi(t)}\rangle=0,\;\;j=1,\ldots,N$, $t\in[0,T)$, $T$ as in Theorem \ref{stay-near}, be the decomposition of the solution $u(t,u_0)$ of \eqref{parabolic0} with initial condition $u_0=u^{\xi_0}+w_0$.

\begin{theorem}
\label{dynamic}
Assume $\bf{H}_1-\bf{H}_3$.  Then there exist $\epsilon_0>0$ and constants $Q>0$, $\beta>1$  independent of $\xi\in\Xi_\rho$ and $\epsilon\in(0,\epsilon_0]$ and such that, if $u_0$ is sufficiently close to $\mathcal{M}$ in the sense that

\begin{equation}
\|w_0\|_{W_\eps^{1,2}}\leq Q\|\mathscr{F}(u^{\xi_0})\|,
\label{c0-nonzero}
\end{equation}

Then, for each $t\in[0,T)$, $T>0$ as in Theorem \ref{stay-near}, we have $u(t,u_0)=u^{\xi(t)}+w(t)$ and

\begin{equation}
\begin{split}
&\dot{\xi}_j=\frac{\epsilon^\frac{1}{2}}{\bar{q}_j}c_j(\xi)
+\mathrm{O}(e^{-\frac{\alpha\rho}{\eps}}\max_{h,\pm}\mathscr{E}_h^\pm),\;\;j=1,\ldots,N,\\
&\|w(t)\|_{W_\eps^{1,2}}\leq \beta Q\|\mathscr{F}(u^{\xi(t)})\|.
\end{split}
\label{dotxi=c}
\end{equation}
\end{theorem}
From Theorem \ref{dynamic} and the explicit expression of $c(\xi)$ derived in Theorem \ref{TH-c} we have that the evolution of the layer positions $\xi_1,\ldots,\xi_N$ is determined by the equations
\begin{equation}
\dot{\xi}_j=\frac{2\epsilon}{\bar{q}_j^2}(\varsigma_{j+1}k_{j+1}^+\mathscr{E}_{j+1}
 -\varsigma_jk_j^-\mathscr{E}_j)+\mathrm{O}(e^{-\frac{\alpha\rho}{\epsilon}}\max_{h,\pm}\mathscr{E}_h^\pm),
 \;\;j=1,\ldots,N.
\label{LD-gen}
\end{equation}
where $\varsigma_j=z_j^-\cdot z_j^+$, $\mu_j=\frac{1}{2}(\mu_j^-+\mu_j^+)$, $k_j^\pm=\mu_j^\pm\mu_j\bar{K}_j^-\bar{K}_{j-1}^+$ and
\[\mathscr{E}_j=e^{-\frac{\mu_j}{\epsilon}(\xi_j-\xi_{j-1})},\quad \mathscr{E}_j^\pm=e^{-\frac{\mu_j^\pm}{\epsilon}(\xi_j-\xi_{j-1})}.\]
If $\bf{H}_4$ holds we have $z_j^-=\pm z_j^+$, $\mu_j=\mu_j^-=\mu_j^+$, $k_j^\pm=k_j=\mu_j^2\bar{K}_j^-\bar{K}_{j-1}^+$, $\mathscr{E}_j=\mathscr{E}_j^\pm$ and \eqref{LD-gen} becomes
\begin{equation}
\dot{\xi}_j=\frac{2\epsilon}{\bar{q}_j^2}(\varsigma_{j+1}k_{j+1}\mathscr{E}_{j+1}
 -\varsigma_jk_j\mathscr{E}_j)+\mathrm{O}(e^{-\frac{\alpha\rho}{\epsilon}}\max_h\mathscr{E}_h),
 \;\;j=1,\ldots,N,
\label{LD-H4}
\end{equation}
with $\varsigma_j=\pm 1$.
Equations \eqref{LD-gen} and \eqref{LD-H4} show that, as was conjectured in \cite{bethsmets1}, beside their relevance in relation to the existence of periodic orbits of \eqref{rescaled} (or equivalently stationary solutions of \eqref{parabolic0}) the scalar products $\varsigma_j=z_j^-\cdot z_j^+$, $j=1,\ldots,N$ have also a central importance in layers dynamics. Indeed \eqref{LD-gen} shows that, depending on the value of $\varsigma_j$ the term  $\varsigma_jk_j^-\mathscr{E}_j$ can give a positive, negative or zero contribution to the speed $\dot{\xi}_j$ of the $j$-th layer. The value of the products $\varsigma_j$ also determines if consecutive layers attract or repel each other. To see this consider \eqref{LD-H4} where $\varsigma_j$ can only assume the extreme values $\pm 1$ and suppose that $\mu_j(\xi_j-\xi_{j-1})>\mu_h(\xi_h-\xi_{h-1})$ for $h\neq j$; then \eqref{LD-H4}, for small $\epsilon>0$, imply
\begin{equation}
\dot{\xi}_j-\dot{\xi}_{j-1}=-2\varsigma_jk_j\mathscr{E}_j+\ldots
\label{attrrepel}
\end{equation}
where $\ldots$ denotes higher order terms in $\epsilon$. From equation \eqref{attrrepel} we have attraction if
$\varsigma_j=1$ and repulsion if
$\varsigma_j=-1$.

\begin{remark} as we have observed, \eqref{parabolic0} is the $L^2$ gradient system associated to the functional \eqref{eps-action}. Therefore it is to be expected that \eqref{LD-gen} or \eqref{LD-H4} is approximately the gradient system of a suitable energy $J^0:\Xi_\rho\rightarrow\R$. We consider only the case of \eqref{LD-H4} and note that, indeed, if we set
$C_j^0(\xi)=\frac{2\epsilon}{\bar{q}_j^2}(\varsigma_{j+1}k_{j+1}\mathscr{E}_{j+1}
-\varsigma_jk_j\mathscr{E}_j)$, we have
\begin{equation}
\begin{split}
&\sum_h C_h^0\bar{q}_h^2\xi_h^\prime=-\sum\frac{\partial}{\partial\xi_h}J^0(\xi)\xi_h^\prime,
\;\;\xi^\prime\in\R^N,\\
&J^0(\xi)=\epsilon\sum_j\bar{q}_j^2-2\epsilon^2\sum_h\frac{\varsigma_hk_h}{\mu_h}\mathscr{E}_h,
\end{split}
\label{J0}
\end{equation}
and therefore $(C_1^0,\ldots,C_N^0)$ is the negative of the gradient of $J^0$  with respect to the inner product $(\xi,\xi^\prime)=\sum_h\bar{q}_h^2\xi_h\xi_h^\prime$ which can be regarded as the restriction to the tangent space of $\mathcal{M}$ of the standard inner product in $L^2((0,1);\R^m)$.
Note also that $J^0(\xi)\approx J_\eps(u^\xi)$, thus $J^0$ is approximately the restriction of $J_\eps(u)$ to $\mathcal{M}$.
\end{remark}
In Fig. \ref{fig3} we consider a potential $W\geq 0$ with three connected zeros. In case a) we have $\varsigma_j=1$ and in case b) $\varsigma_j=-1$. Under assumptions ${\mathbf{H}}_1-{\mathbf{H}}_4$ in case a) Theorem \ref{CNS} gives the existence of a number of stationary solutions of \eqref{parabolic0} that diverges to $\infty$ as $\epsilon\rightarrow 0$ and $\varsigma_j=1$ implies that layers attract each other and is to be expected that all these stationary solutions are unstable for the dynamics of \eqref{parabolic0}. In case b) from Theorem \ref{CNS} we deduce the existence of just one stationary solution which, since $\varsigma_j=-1$ implies that layers repel each other, is conjectured to be stable. In Fig. \ref{fig3} we also sketch the expected behavior of $J^0$ in a neighborhood of the stationary solution.
\begin{figure}
  \begin{center}
\begin{tikzpicture}

{
\draw [blue, thick](2,0).. controls (3,2) and(5,2)..(6,0);
\draw [blue, thick](7.5,1.5).. controls (8.5,-.5) and(10.5,-.5)..(11.5,1.5);
 \node at (3.9,2){$\varsigma_h=1$};
 \node at (8.9,2){$\varsigma_h=-1$};

\node at (7.1,1){$J^0(\xi)$};
\node at (2.1,1){$J^0(\xi)$};
\node at (5,-.1){$\xi\in\Xi_\rho$};
\node at (11,-.2){$\xi\in\Xi_\rho$};
\draw (1.5,-.5)--(6.5,-.5);
\draw (7,-.5)--(12,-.5);}

  \draw[fill]  (3.9,1.5) circle [radius=0.05];
 \draw[fill]  (9.3,0) circle [radius=0.05];


{\draw [blue, thick ] (5.4,-3.75) arc (60:120:2.595);
\draw [blue, thick ] (4.095,-1.5) arc (-180:-120:2.595);
\draw [blue, thick ] (2.8,-3.75) arc (-60:0:2.595);}
{\draw[ fill] (4.09,-1.5) circle [radius=0.05];
\draw[ fill] (5.4,-3.75) circle [radius=0.05];
\draw[ fill] (2.8,-3.75) circle [radius=0.05];}

{\draw[ fill] (8.89,-1.45) circle [radius=0.05];
\draw[ fill] (10.05,-3.55) circle [radius=0.05];
\draw[ fill] (8.05,-3.65) circle [radius=0.05];
\draw [blue, thick ] (9.0,-2.75) circle [radius=1.3];}

{\node at (3,-2){a)};
\node at (7,-2){b)};}

\end{tikzpicture}
\end{center}

\caption{A potential with three zeros, $\varsigma_h=\pm 1$.}
\label{fig3}
\end{figure}
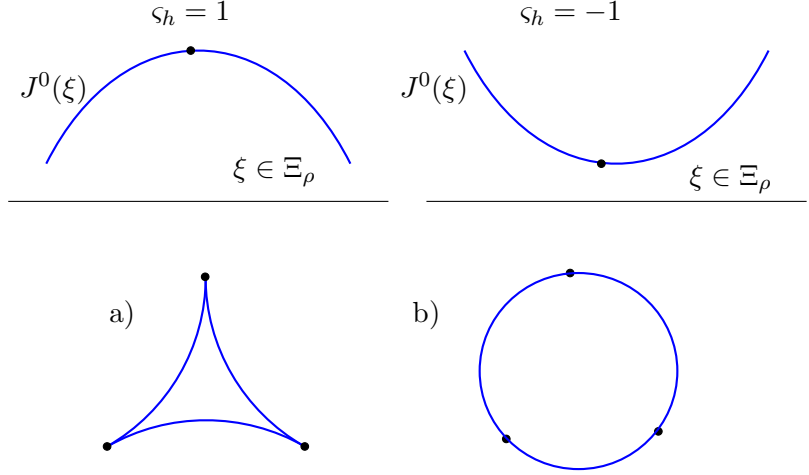

There is an extensive literature on layers dynamics, slow motion and related problems.
Layers dynamics in the scalar case and for potentials with two minima, the typical example being $W(u)=\frac{1}{4}(1-u^2)^2$, was first described by Neu and then analyzed in \cite{cp}, \cite{fh}, \cite{brKohn}, \cite{cp1}, \cite{f}, \cite{ward}, \cite{or} and \cite{bnn} for the Allen-Cahn equation and  \cite{abaf}, \cite{brh}, \cite{bx}, \cite{bx1} for the Cahn-Hilliard equation. Slow motion of layers also appears in the context of dissipative hyperbolic equations \cite{MS}, \cite{FLM}. For the dynamics of the scalar Allen-Cahn equation with general nonlinearity we refer to \cite{MaPo}, \cite{MaPo1} and the references there in. For the vector case $m>1$ we refer to \cite{grant}, \cite{bethsmets1} and \cite{bethsmets2}. The phenomenon of slow motion in higher dimension was studied in \cite{afk}, \cite{abf}, \cite{af3},\cite{ward1},\cite{EY} and \cite{muri}.

 \section{The proof of Theorem \ref{CNS}}

 The analysis of equation \eqref{quasiinv0} is based on the estimate \eqref{beta-error} (see Lemma \ref{est-Fuxi}) and on the properties of the linear operator ${L}^\xi:W^{2,2}((0,1);\R^m)\rightarrow L^2((0,1);\R^m)$, acting on  the subset of $1$-periodic maps,
\begin{equation}\label{operator}
{L}^\xi v=-\epsilon^2v_{xx}+W_{uu}(u^\xi)v,
\end{equation}
that we collect in the following proposition

\begin{proposition}\label{SpectrumL}
 Assume $\mathbf{H}_1-\mathbf{H}_3$.
There exist $\epsilon_0>0$, $\lambda^*>0$ and $\mu^*>0$ such that, for $\epsilon\in(0,\epsilon_0]$, $\xi\in\Xi_\rho$, the operator ${L}^\xi$ has $N$ eigenvalues $\lambda_1^\xi,\ldots,\lambda_N^\xi$ that satisfy
\begin{equation}
\vert\lambda_j^\xi\vert\leq Ce^{-\frac{\alpha\rho}{\epsilon}},\;\;j=1,\ldots,N.
\label{EigenBounds}
\end{equation}
If $\lambda\neq\lambda_1^\xi,\ldots,\lambda_N^\xi$ is an eigenvalue of ${L}^\xi$ then
\begin{equation}
\lambda\geq\lambda^*.
\label{EigenGap}
\end{equation}
Let $X^\xi$ the eigenspace associated to the spectral set $\{\lambda_1^\xi,\ldots,\lambda_N^\xi\}$, then $X^\xi$ has a hortonormal basis $\{\varphi_1^\xi,\ldots,\varphi_N^\xi\}$ of the form
\begin{equation}
\varphi_j^\xi:=\frac{u_{\xi_j}^\xi}{\| u_{\xi_j}^\xi\|}+\eta_j^\xi\in X^\xi,\;\;  \|\eta_j^\xi\|=O(e^{-\frac{\alpha\rho}{\epsilon}}).
\label{approxphi}
\end{equation}
Let $X^{\xi,\perp}$ the orthogonal complement of $X^\xi:=\mathrm{Span}(\varphi_1^\xi,\ldots,\varphi_N^\xi)$. Then
\begin{equation}
\begin{split}
&\vert\langle{L}^\xi v,v\rangle\vert\geq\lambda^*\|v\|^2,\;\;v\in X^{\xi,\perp}\cap W^{1,2}((0,1);\R^m),\\
&\vert\langle{L}^\xi v,v\rangle\vert\geq\mu^*\|v\|_{W_{\eps}^{1,2}}^2,\;\;v\in X^{\xi,\perp}\cap W^{1,2}((0,1);\R^m).
\end{split}
\label{Lbounds}
\end{equation}
\end{proposition}
For the proof of Proposition \ref{SpectrumL} we refer to \cite{BFK} Appendix A where the proof is done for the case of Neumann boundary conditions. The argument applies to the present case of periodic boundary conditions with slight modifications.

 From Proposition \ref{SpectrumL} ${L}^\xi$ is almost singular on $X^\xi$ and therefore equation (\ref{rescaled1}), as it stands, can not be solved for $v$. Therefore we proceed as sketched before and begin by solving for $v\in X^{\xi,\perp}$ the  projection \eqref{quasiinv0} of  equation \eqref{rescaled1} on $X^{\xi,\perp}$.

Given $v\in X^{\xi,\perp}$, by forming the inner product of \eqref{quasiinv0} with $\varphi_j^\xi$, $j=1,\ldots,N$, and observing that $\langle L^\xi v,\varphi_j^\xi\rangle=0$, we obtain that $c(\xi,v)$ is uniquely determined by

\begin{equation}
c_j(\xi,v)=\langle\mathscr{F}(u^\xi)-N^\xi(v),\varphi_j^\xi\rangle,\;\;j=1,\ldots,N,
\label{ortcond}
\end{equation}
where we have set

\begin{equation}
\begin{split}
&N^\xi(v):=W_u(u^\xi+v)-W_u(u^\xi)-W_{uu}(u^\xi)v,\\
&\mathscr{F}(u^\xi):=\epsilon^2 u_{xx}^\xi-W_u(u^\xi),
\end{split}
\label{Nandbeta}
\end{equation}

Let $V^\xi$ be the set of $1$-periodic maps defined by
\begin{equation}
V^\xi:=\{v\in X^{\xi,\perp}\cap W^{1,2}((0,1);\R^m): \|v\|_{W_\epsilon^{1,2}}\leq\frac{2}{\mu^*}\|\mathscr{F}(u^\xi)\|\},
\label{Vxi}
\end{equation}
where $\mu^*$ is the constant in \eqref{Lbounds}.

We show that (\ref{quasiinv0}) with $c(\xi,v)$ given by \eqref{ortcond} has a unique solution $v\in V^\xi$.
\begin{proposition}\label{quasiinv1}
Assume $\mathbf{H}_1-\mathbf{H}_3$.
There is $\epsilon_0>0$ such that, $\epsilon\in(0,\epsilon_0]$ implies the following: for each $\xi\in\Xi_\rho$ there exist uniquely determined $v^\xi\in V^\xi$ and $c(\xi)=c(\xi,v^\xi)\in\R^N$ that solve (\ref{quasiinv0}). The maps  $\Xi_\rho\ni\xi\rightarrow c(\xi)\in\R^N$ and $\Xi_\rho\ni\xi\rightarrow v^\xi\in W^{1,2}$ are smooth and
\begin{equation}
\begin{split}
&\|v^\xi\|_{W_\epsilon^{1,2}}\leq\frac{2}{\mu^*}\|\mathscr{F}(u^\xi)\|,\\
&c(\xi)=(\langle \mathscr{F}(u^\xi),\varphi_1^\xi\rangle,\ldots,\langle\mathscr{F}(u^\xi),\varphi_N^\xi\rangle)^\top
+O(\epsilon^{-\frac{1}{2}}\|\mathscr{F}(u^\xi)\|^2),\\
&\epsilon^2(u^\xi+v^\xi)_{xx}-W_u(u^\xi+v^\xi)=\sum_{j=1}^Nc_j(\xi)\varphi_j^\xi.
\end{split}
\label{cvest}
\end{equation}

\end{proposition}
\begin{proof}
1. As we have observed, given $v\in V^\xi$, (\ref{ortcond}) determines the unique $c(\xi,v)$ ensuring that the right hand side of
\begin{equation}
-{L}^\xi\hat{v}=\sum_j\varphi_j^\xi c_j(\xi,v)+N^\xi(v)-\mathscr{F}(u^\xi),
\label{quasiinvhat}
\end{equation}
belongs to $X^{\xi,\perp}$. Moreover the smoothness of $u^\xi$ and $\varphi_j^\xi$ implies that the right hand side of \eqref{quasiinvhat}
is a $W^{1,2}$ map. Therefore Proposition \ref{SpectrumL} implies that (\ref{quasiinvhat}) has a unique solution $\hat{v}\in X^{\xi,\perp}\cap W^{1,2}((0,1);\R^m)$.

2. For $v\in V^\xi$ we have
\begin{equation}
\|v\|_{L^\infty(\R;\R^m)}\leq C_1\epsilon^{-\frac{1}{2}}\|\mathscr{F}(u^\xi)\|,
\label{v-linfty}
\end{equation}
 and
\begin{equation}
\|N^\xi(v)\|\leq C_2\epsilon^{-\frac{1}{2}}\|\mathscr{F}(u^\xi)\|^2.
\label{Nv}
\end{equation}

Fix $\bar{x}\in[0,1]$ such that $\vert v(\bar{x})\vert=\min_{x\in[0,1]}\vert v(x)\vert$. Then we have
\begin{equation}\vert v(x)\vert^2\leq\vert v(\bar{x})\vert^2+2\int_{\bar{x}}^{\bar{x}+1}\vert v(s)\vert\vert v^\prime(s)\vert ds
\leq\|v\|^2+2\|v\|\|v^\prime\|\leq\frac{2}{\epsilon}\|v\|_{W_\epsilon^{1,2}}^2,
\label{linftyproof}
\end{equation}
and \eqref{v-linfty} follows from \eqref{Vxi}.

To prove \eqref{Nv} we note that, for each $x\in[0,1]$, it results
\[N^\xi(v)=\int_0^1\Big(W_{uu}(u^\xi+s v)-W_{uu}(u^\xi)\Big)v ds=\int_0^1s\int_0^1 W_{uuu}(u^\xi+rs v)v\cdot vdr ds.\]
From this, since \eqref{beta-error} implies
\[\max_{\vert z\vert\leq C_1\epsilon^{-\frac{1}{2}}\|\mathscr{F}(u^\xi)\|} \vert W_{uuu}(u^\xi+z)\vert\leq C,\]
 on the basis of \eqref{v-linfty}  we have,
\[\vert N^\xi(v)\vert\leq C\vert v\vert^2\leq C\epsilon^{-\frac{1}{2}}\|\mathscr{F}(u^\xi)\|\vert v\vert ,\;\;x\in[0,1],\]
and \eqref{Nv} follows.

3. $\hat{v}\in V^\xi$. By taking the inner product of (\ref{quasiinvhat}) with $\hat{v}$ we obtain via (\ref{Lbounds})
\[\begin{split}
&\mu^*\|\hat{v}\|_{W_\epsilon^{1,2}}^2\leq\|N^\xi(v)\|\|\hat{v}\|+\|\mathscr{F}(u^\xi)\|\|\hat{v}\|\\
&\leq (C_2\epsilon^{-\frac{1}{2}}\|\mathscr{F}(u^\xi)\|+1)\|\mathscr{F}(u^\xi)\|\|\hat{v}\|
.
\end{split}\]
This and $\|\mathscr{F}(u^\xi)\|=O(e^{-\frac{\alpha\rho}{\epsilon}})$ imply that, provided $\epsilon>0$ is sufficiently small, we have
\[\|\hat{v}\|_{W_\epsilon^{1,2}}\leq\frac{2}{\mu^*}\|\mathscr{F}(u^\xi)\|,\]
that proves the claim.

4. The map $V^\xi\ni v\rightarrow\hat{v}\in V^\xi$ constructed in 3. is a contraction. By subtracting equation (\ref{quasiinvhat}) written for $v, w\in V^\xi$ and taking the inner product with $\hat{v}-\hat{w}$  we get
\begin{equation} -\langle L^\xi(\hat{v}-\hat{w}),\hat{v}-\hat{w}\rangle
=\langle N^\xi(v)-N^\xi(w),\hat{v}-\hat{w}\rangle.
\label{contr-eq0}
\end{equation}
where we have used $\langle\varphi_j^\xi,\hat{v}-\hat{w}\rangle=0$. From \eqref{Lbounds} and \eqref{contr-eq0} it follows

\begin{equation}
\mu^*\|\hat{v}-\hat{w}\|_{W_\epsilon^{1,2}}
\leq\|N^\xi(v)-N^\xi(w)\|.
\label{contr-eq}
\end{equation}
We have
\[\begin{split}
&N^\xi(v)-N^\xi(w)
=\int_0^1(W_{uu}(u^\xi+w+\tau(v-w))-W_{uu}(u^\xi))(v-w)d\tau\\
&=\int_0^1\int_0^1W_{uuu}(u^\xi+t(w+\tau(v-w)))(w+\tau(v-w))(v-w)dt d\tau,
\end{split}\]
and therefore, since $v,w\in V^\xi$, by \eqref{v-linfty}
\[\begin{split}
&\vert N^\xi(v)-N^\xi(w)\vert\leq C(\vert v\vert+2\vert w\vert)\vert v-w\vert\leq 3\frac{C}{\epsilon^\frac{1}{2}}\|\mathscr{F}(u^\xi)\|\vert v-w\vert,\\
&\Rightarrow\quad\|N^\xi(v)-N^\xi(w)\|\leq 3\frac{C}{\epsilon^\frac{1}{2}}\|\mathscr{F}(u^\xi)\|\|v-w\|,\;\;v,w\in V^\xi.
\end{split}\]
From this and \eqref{contr-eq} we obtain
\[\|\hat{v}-\hat{w}\|_{W_\epsilon^{1,2}}\leq 3\frac{C}{{\mu^*}\epsilon^\frac{1}{2}}\|\mathscr{F}(u^\xi)\|\|v-w\|\]
that implies the claim for $\epsilon>0$ sufficiently small.

5. Let $v^\xi\in V^\xi$ the fixed point of the contraction map $V^\xi\ni v\rightarrow\hat{v}\in V^\xi$. Clearly $v^\xi$ satisfies (\ref{cvest})$_1$. The estimate for $c(\xi)$ follows from
\[c(\xi)=c(\xi,v^\xi)=-
(\langle N^\xi(v^\xi)-\mathscr{F}(u^\xi),\varphi_1^\xi
\rangle,\ldots,\langle N^\xi(v^\xi)-\mathscr{F}(u^\xi),\varphi_N^\xi\rangle)^\top\]
and (\ref{cvest})$_1$ and (\ref{Nv});
\eqref{cvest}$_3$ is just a rewriting of \eqref{quasiinv0} with $v=v^\xi$. The proof is complete.
\end{proof}
\begin{remark}
\label{c-xi-xi}
The definition of $u^\xi$ implies that $u^{\xi+s\nu}=u^\xi(\cdot-s)$, ($\nu=(1,\ldots,1)^\top$) this and \eqref{operator} imply that we also have $\varphi_j^{\xi+s\nu}=\varphi_j^\xi(\cdot-s)$ and in turn $v^{\xi+s\nu}=v^\xi(\cdot-s)$. These observations and \eqref{cvest} imply
\begin{equation}
c(\xi+s\nu)=c(\xi),
\label{c-xi-xi1}
\end{equation}
That is $c(\xi)$ depends on $\xi$ only through the differences $\xi_h-\xi_{h-1}$, $i=1,\ldots,N$ and the identity
\[\sum_{h=1}^N\xi_h-\xi_{h-1}=\xi_N-\xi_0=1,\;\;(\xi_0=\xi_N-1),\]
 implies that $c(\xi)$ is actually a function of $N-1$ variables.
\end{remark}

From Proposition \ref{quasiinv1} there is a periodic solution of (\ref{rescaled}) in a neighborhood of $u^\xi$ of size $\|\mathscr{F}(u^\xi)\|$ if and only if $\xi$ is a solution of the bifurcation equation
\begin{equation}
c(\xi)=0.
\label{bifurcation}
\end{equation}
From the above remark this is a system of $N$ equations in $N-1$ unknowns. The following Proposition shows that, in spite of this, it is  possible to solve \eqref{bifurcation}.

\begin{proposition}
\label{N-1=N}
For $\epsilon\in(0,\epsilon_0]$, for some $\epsilon_0>0$, and for each $\xi\in\Xi_\rho$ there exist numbers $p_j^\xi\in\R$, $j=1,\ldots,N$ that satisfy

\begin{equation}
\begin{split}
&\sum_{j=1}^N(\bar{q}_j+p_j^\xi) c_j(\xi)=0,\\
& p_j^\xi =\mathrm{O}(e^{-\frac{\alpha\rho}{\epsilon}}),
\end{split}
\label{cLC}
\end{equation}
where $\bar{q}_j^2=\int_\R\vert\bar{u}_j^\prime\vert^2 ds$.
\end{proposition}
\begin{proof}

Forming the inner product of \eqref{cvest}$_3$ with $u_x^\xi+v_x^\xi$ yields
\begin{equation}
0=\int_0^1\frac{d}{dx}\Big(\frac{\epsilon^2}{2}\vert u_x^\xi+v_x^\xi\vert^2-W(u^\xi+v^\xi)\Big)dx=\sum_jc_j(\xi)\langle\varphi_j^\xi,u_x^\xi+v_x^\xi\rangle,
\label{=0}
\end{equation}
where we have observed that the left hand side vanishes being the integral extended to one period of the derivative with respect to $x$ of a periodic function. Now we note that \eqref{uxi} implies $u_x^\xi=-\sum_ju_{\xi_j}^\xi$ and rewrite \eqref{=0} in the form
\[\begin{split}
&0=\sum_jc_j(\xi)\langle\varphi_j^\xi,u_x^\xi+v_x^\xi\rangle
=-\sum_jc_j(\xi)(\langle\varphi_j^\xi,\sum_i u_{\xi_i}^\xi\rangle-\langle\varphi_j^\xi,v_x^\xi\rangle)\\
&=-\sum_jc_j(\xi)(\langle\varphi_j^\xi,\sum_i \|u_{\xi_i}^\xi\|(\frac{u_{\xi_i}^\xi}{\|u_{\xi_i}^\xi\|}+\eta_i^\xi)\rangle-
\langle\varphi_j^\xi,v_x^\xi+\sum_i \|u_{\xi_i}^\xi\|\eta_i^\xi\rangle)\\
&=-\sum_jc_j(\xi)(\langle\varphi_j^\xi,\sum_i \|u_{\xi_i}^\xi\|\varphi_i^\xi\rangle+\mathrm{O}(e^{-\frac{\alpha\rho}{\eps}}))\\
&=-\sum_jc_j(\xi)(\|u_{\xi_j}^\xi\|+\mathrm{O}(e^{-\frac{\alpha\rho}{\eps}})),
\end{split}\]
Where we have used \eqref{approxphi} and $\|v_x^\xi\|=\mathrm{O}(e^{-\frac{\alpha\rho}{\eps}})$ that follows from \eqref{beta-error} and Proposition \ref{quasiinv1}. To conclude the proof we note that, see \eqref{uxixi}
\[\|u_{\xi_j}^\xi\|=\frac{1}{\epsilon^\frac{1}{2}}\bar{q}_j+\mathrm{O}(e^{-\frac{\alpha\rho}{\eps}}).\]
\end{proof}

The proof of Theorem \ref{CNS} is based on the detailed analysis of $c(\xi)$ in Theorem \ref{TH-c} and also utilizes assumption $\mathbf{H}_4$. Therefore is worthwhile to observe that, if all the components of $c(\xi)$ were equal or $c(\xi)$ were a scalar quantity,
Proposition \ref{N-1=N} would automatically imply $c(\xi)=0$ without any further analysis and without requiring
$\mathbf{H}_4$. A nontrivial situation where $c(\xi)=0$ can be deduced directly from Proposition \ref{N-1=N} arises when $W:\R^2\rightarrow\R$ is invariant under the rotation group $C_N$ of the regular polygon with $N\geq 3$ sides
\begin{equation}
W(\omega u)=W(u),\;\;u\in\R^2,\;\omega=\left(\begin{array}{l}
\cos{\frac{2\pi}{N}}\;\;-\sin{\frac{2\pi}{N}}\\
\sin{\frac{2\pi}{N}}\;\;\;\cos{\frac{2\pi}{N}}
\end{array}\right).
\label{W-inv}
\end{equation}
We assume that $W\geq 0$ has exactly $N$ zeros $\omega^{j-1}a$, $j=1,\ldots,N$ for some nondegenerate $a\in\R^2\setminus\{0\}$. These assumptions imply the existence of a minimizer $\bar{u}$ of $J(u)$ in the class of maps that connect $a$ to $\omega a$. From Proposition \ref{quasiinv1} and Proposition \ref{N-1=N} we have the following result on the existence of periodic solutions which are equivariant under $C_N$.
\begin{theorem}
\label{Equi-Per}
Assume that $W:\R^2\rightarrow\R$, $W\geq 0$, satisfies \eqref{W-inv} and that $\{W=0\}=\{a,\omega a,\ldots,\omega^{N-1} a\}$ for some nondegenerate $a\in\R^2\setminus\{0\}$. Assume that the zero eigenvalue of the operator $Lv=-v^{\prime\prime}+W_{uu}(\bar{u})v$ is simple. Then there exist $T_0>0$ and a family $u^T$, $T\geq T_0$ of $T$-periodic solutions of \eqref{newton} that satisfy
\[u^T(t+\frac{T}{N})=\omega u^T(t),\;\;t\in\R.\]
Moreover there are positive constants $k_0, K_0$ such that
\begin{equation}\begin{split}
&\|u^T-u^{T,\xi}\|_{W^{1,2}([0,1];\R^2)}\leq K_0e^{-k_0T}e^{-\frac{\mu_m}{2N}T},\\
&\vert J_{(0,T)}(u^T)-J_{(0,T)}(u^{T,\xi}\vert\leq K_0e^{-k_0T}e^{-\frac{\mu_m}{N}T},\\
&J_{(0,T)}(u^{T,\xi})=N\bar{q}^2+\mathrm{O}(e^{-\frac{\mu_m}{2N}T}),
\end{split}
\label{uT-est}
\end{equation}
where $\mu_m=\min\{\mu^-,\mu^+\}$, $\bar{q}^2=\int_\R\vert\bar{u}^\prime\vert^2ds$ and
\[\begin{split}
&u^{T,\xi}(t)=\sum_{n<0}\sum_{h=1}^N\omega^{h-1}\Big(\bar{u}(t-(n+\xi+\frac{h-1}{N})T)-a\Big)+a\\
&+\sum_{n\geq 0}\sum_{h=1}^N\omega^{h-1}\Big(\bar{u}(t-(n+\xi+\frac{h-1}{N})T)-\omega a\Big),\;\;\xi\in[0,\frac{1}{N}].
\end{split}\]
\end{theorem}
\begin{proof}
Set $u^\xi(\cdot)=u^{\frac{1}{\epsilon},\xi}(\frac{\cdot}{\epsilon})$. Since $\xi$ is a real parameter we find that the operator $L^\xi v=-\epsilon^2v_{xx}+W_{uu}(u^\xi)v$, restricted to the class of $1$-periodic equivariant maps
\[v(x+\frac{1}{N})=\omega v(x),\]
has a unique small eigenvalue $\lambda^\xi=\mathrm{O}(e^{-\frac{\alpha\rho}{\epsilon}})$ with corresponding eigenvector $\varphi^\xi=\frac{u_\xi^\xi}{\|u_\xi^\xi\|}+\mathrm{O}(e^{-\frac{\alpha\rho}{\epsilon}})$. Then arguing as in Proposition \ref{quasiinv1} we have that, for each $\xi$ there is $c(\xi)\in\R$ and a $1$-periodic equivariant map $v^\xi$ such that
\[\begin{split}
&\epsilon^2(u^\xi+v^\xi)_{xx}-W_u(u^\xi+v^\xi)=c(\xi)\varphi^\xi,\\
&\langle v^\xi,\varphi^\xi\rangle=0.
\end{split}\]
The reasoning of Proposition \ref{N-1=N} then shows that $c(\xi)=0$ and renders a $1$-periodic equivariant solution  $\hat{u}^\xi=u^\xi+v^\xi$ of \eqref{rescaled}. The change of variable $x=\frac{t}{T}$, $\epsilon=\frac{1}{T}$, yields the sought $T$-periodic equivariant solution $u^T$ of \eqref{newton}. The estimate \eqref{uT-est}$_3$ follows from \eqref{Juxi-est} and $J_{(0,T)}(u^{T,\xi})=\frac{1}{\epsilon}J_\epsilon(u^\xi)$. The other two estimates \eqref{uT-est} follow from \eqref{cvest}.
\end{proof}

\subsection{Discussion of the bifurcation equation}
In this subsection we derive an explicit expression for $c(\xi)$ which is basic for the proof of Theorem \ref{CNS} and for the description of layers dynamics.

We define $\bar{u}_j$ also for $j=N+1$, $s<0$ and for $j=0$ and $s>0$ by setting
\[\begin{split}
&\bar{u}_{N+1}(s)-a_{N+1}=\bar{u}_1(s)-a_1,\;\;s<0\\
&\bar{u}_0(s)-a_1=\bar{u}_N(s)-a_{N+1},\;\;s>0.
\end{split}\]
We also set
\[\xi_0=\xi_N-1,\quad\quad \xi_{N+1}=\xi_1+1.\]
\noindent These definitions imply
\begin{equation}
\begin{split}
&z_{N+1}^-=z_1^-,\quad\quad z_1^+=z_{N+1}^+,\\
&\mu_{N+1}^-=\mu_1^-,\quad\quad\mu_1^+=\mu_{N+1}^+,\\
&\bar{K}_{N+1}^-=\bar{K}_1^-,\quad\quad \bar{K}_0^+=\bar{K}_N^+,
\end{split}
\label{sigma0sigmaN}
\end{equation}

\noindent We introduce the notation:
\begin{equation}
\begin{split}
&\mu_h=\frac{1}{2}(\mu_h^-+\mu_h^+),\;\;h=1,\ldots,N+1,\\
& k_h^\pm=\mu_h^\pm\mu_h\bar{K}_h^-\bar{K}_{h-1}^+,\;\;h=1,\ldots,N+1,\\
&\varsigma_h=z_h^-\cdot z_h^+,\;\;h=1,\ldots,N+1,
\end{split}
\label{varsigma}
\end{equation}
and
\begin{equation}
\begin{split}
&\mathscr{E}_h^\pm=e^{-\frac{\mu_h^\pm}{\epsilon}(\xi_h-\xi_{h-1})},\;\;h=1,\ldots,N+1,\\
&\mathscr{E}_h=e^{-\frac{\mu_h}{\epsilon}(\xi_h-\xi_{h-1})},\;\;h=1,\ldots,N+1.\\
\end{split}
\label{Eh}
\end{equation}

Observe that we have $\xi_1-\xi_0=\xi_{N+1}-\xi_N$ and therefore $\mathscr{E}_{N+1}^\pm=\mathscr{E}_1^\pm$. This and the identity $1=\sum_h(\xi_h-\xi_{h-1})$ imply that $\mathscr{E}_1^\pm,\ldots,\mathscr{E}_N^\pm$ are not independent and one of them is determined by the values of the remaining  $N-1$.

Note that \eqref{sigma0sigmaN} imply
\begin{equation}
\varsigma_{N+1}=\varsigma_1,\quad\quad k_{N+1}^\pm=k_1^\pm.
\label{varsigma1}
\end{equation}
 If {\bf{H}$_4$} holds ($z_h^-=\pm z_h^+$) we have
 \[\begin{split}
 &\varsigma_h=\pm 1,\\
 &\mathscr{E}_h^\pm=\mathscr{E}_h,\\
 &k_h^\pm=k_h:=\mu_h^2\bar{K}_h^-\bar{K}_{h-1}^+.
 \end{split}\]

 We prove
 \begin{theorem}
 \label{TH-c}
 There is $\epsilon_0>0$ such that for $\epsilon\in(0,\epsilon_0]$ and $\xi\in\Xi_\rho$ it results
 \begin{equation}
  c(\xi)=c^0(\xi)+r(\xi),
 \label{c0-def0}
 \end{equation}
  where
  \begin{equation}
 \begin{split}
 &c_j^0(\xi)=\frac{2\epsilon^\frac{1}{2}}{\bar{q}_j}(\varsigma_{j+1}k_{j+1}^+\mathscr{E}_{j+1}
 -\varsigma_jk_j^-\mathscr{E}_j),\;\;j=1,\ldots,N,\\
 &r_j(\xi)=\mathrm{O}(e^{-\frac{\alpha\rho}{\epsilon}}\max_{h,\pm}\mathscr{E}_h^\pm).
 \end{split}
 \label{c0-def}
 \end{equation}
 \end{theorem}
\noindent Note that if {\bf{H}$_4$} holds we have $\mu_h^-=\mu_h^+$, $\mathscr{E}_h=\mathscr{E}_h^\pm$,
$k_h^\pm=k_h=\mu_h^2\bar{K}_h^-\bar{K}_{h-1}^+$ and therefore
\begin{equation}
\begin{split}
&c_j^0(\xi)=\frac{2\epsilon^\frac{1}{2}}{\bar{q}_j}(\varsigma_{j+1}k_{j+1}\mathscr{E}_{j+1}
 -\varsigma_jk_j\mathscr{E}_j),\\
&r_j(\xi)=\mathrm{O}(e^{-\frac{\alpha\rho}{\epsilon}}\max_h\mathscr{E}_h)
\end{split}
\label{c0-gen}
 \end{equation}
with $\varsigma_j=\pm 1$.

\begin{remark}
From Theorem \ref{TH-c} $c_j^0(\xi)$ depends only on the differences $\xi_j-\xi_{j-1}$ and $\xi_{j+1}-\xi_j$ and we expect that, at least approximately, the same is true for $c_j(\xi)$. Actually this is not always the case. Indeed, depending on the values of the $\mu_j^\pm$ and on the spacing of $\xi_h$ the contribution of the term $\mathrm{O}(e^{-\frac{\alpha\rho}{\epsilon}}\max_{h,\pm}\mathscr{E}_h^\pm)$ may exceed $c_j^0(\xi)$ and therefore $c_j(\xi)$ may depend on $\xi_h$ for $h\not\in\{j-1,j,j+1\}$. In spite of this, the estimate provided by Theorem \ref{TH-c} is sufficient for the proof of Theorem \ref{CNS} and to derive information on layers dynamics. We also note that we don't expect that a better analysis of the difference $c_j(\xi)-c_j^0(\xi)$ is possible under our general assumptions. Improving the estimate of  $c_j(\xi)-c_j^0(\xi)$ given by Theorem \ref{TH-c} requires  specific assumptions on the values of $\mu_h^\pm$ and on the position of the $\xi_h$.
\end{remark}

\begin{proof} (Proof of Theorem \ref{TH-c}) The proof of Theorem \ref{TH-c} is quite elaborate and requires several steps and lemmas.

We begin with a detailed analysis of the expression \eqref{uxi} of $u^\xi$, $\xi\in\Xi_\rho$.

\noindent Given $1\leq j\leq N$, using the identity
\[ a_1+\sum_{h\leq j-1} \Big( \bar{u}_h( \frac{x-\xi_h}{\eps})-a_h \Big)=\sum_{h\leq j-1} \Big( \bar{u}_h( \frac{x-\xi_h}{\eps})-a_{h+1} \Big)+a_j,\]
we rewrite (\ref{uxi}) in the form
\begin{equation}
\begin{split}
u^{\xi}(x)=\displaystyle\bar{u}_{j-1}( \frac{x-\xi_{j-1}}{\eps})-a_j+\bar{u}_j( \frac{x-\xi_j}{\eps})
+\bar{u}_{j+1}( \frac{x-\xi_{j+1}}{\eps})-a_{j+1}+ \tau_j,
\end{split}
\label{uxi1}
\end{equation}


where
\begin{equation}
\begin{split}
&\tau_j=\sum_{h<j-1} \Big( \bar{u}_h( \frac{x-\xi_h}{\eps})-a_{h+1}\Big)+\sum_{h>j+1} \Big( \bar{u}_h( \frac{x-\xi_h}{\eps})-a_h\Big)\\
&+\sum_{h=1}^N\sum_{n<0}\Big( \bar{u}_h( \frac{x-n-\xi_h}{\eps})-a_{h+1} \Big)
+\sum_{h=1}^N\sum_{n>0} \Big( \bar{u}_h( \frac{x-n-\xi_h}{\eps})-a_h \Big).
\end{split}
\label{tau}
\end{equation}
We also consider
\begin{equation}
\begin{split}
&\epsilon u_{\xi_j}^\xi(x)=-\bar{u}_j^\prime(\frac{x-\xi_j}{\epsilon})-\sum_{n\neq 0}\bar{u}_j^\prime(\frac{x-n-\xi_j}{\epsilon})=-\bar{u}_j^\prime(\frac{x-\xi_j}{\epsilon})+\varkappa_j,\\
&\varkappa_j=-\sum_{n\neq 0}\bar{u}_j^\prime(\frac{x-n-\xi_j}{\epsilon}).
\end{split}
\label{uxixij}
\end{equation}

\begin{lemma}
\label{est-uxi}
For small $\epsilon>0$ we have the estimates
\begin{equation}
\begin{split}
&\vert\tau_j\vert
\leq Ce^{-\frac{\alpha\rho}{\eps}}\max_{h,\pm}\mathscr{E}_h^\pm,\;x\in[\hat{\xi}_{j-1},\hat{\xi}_j],\;\;
\hat{\xi}_h=\frac{1}{2}(\xi_{h+1}+\xi_h),\\
&\vert\varkappa_j\vert\leq C(\mathscr{E}_j^-+\mathscr{E}_{j+1}^+),\;\;x\in[0,1].
\end{split}
\label{tau-est}
\end{equation}
Moreover $\vert\eps\tau_{jx}\vert$,  $\vert\eps\tau_{j\xi_j}\vert$, $\vert\eps^2\tau_{jxx}\vert$ and $\vert\eps^3\tau_{jxx\xi_j}\vert$ are bounded in $[\hat{\xi}_{j-1},\hat{\xi}_j]$ by $Ce^{-\frac{\alpha\rho}{\eps}}\max_{h,\pm}\mathscr{E}_h^\pm$. A similar statement applies to $\varkappa_j$ in $[0,1]$. We also have

\begin{equation}
\begin{split}
&\|u_{\xi_j}^\xi\|=\frac{1}{\epsilon^\frac{1}{2}}\bar{q}_j+\mathrm{O}(\mathscr{E}_j^-+\mathscr{E}_{j+1}^+),
\;\;(\bar{q}_j^2=\int_\R\vert\bar{u}_j^\prime\vert^2 ds),\\
&\|u_{\xi_i\xi_j}^\xi\|\leq\frac{C}{\epsilon^\frac{3}{2}}\delta_{ij},\\
&\| u_x^\xi\|\leq \frac{C}{\epsilon^\frac{1}{2}},\;\;
\|u_{xx}^\xi\|\leq \frac{C}{\epsilon^\frac{3}{2}},
\end{split}
\label{uxixi}
\end{equation}
\begin{equation}
J_\epsilon(u^\xi)=\epsilon\sum_j\bar{q}_j^2+\mathrm{O}(\epsilon\max_{h,\pm}{\mathscr{E}_h^\pm}^\frac{1}{2}),
\label{Juxi-est}
\end{equation}
and
\begin{equation}
\frac{\vert u_{\xi_i}^\xi(x)\cdot u_{\xi_j}^\xi(x)\vert}{\|u_{\xi_i}^\xi\|\|u_{\xi_j}^\xi\|}
\leq\frac{C}{\epsilon}\max_{h,\pm}\mathscr{E}_h^\pm,\;\;i\neq j
.\quad\quad\quad
\label{vij}
\end{equation}
\end{lemma}
\begin{proof}
See Section \ref{Pf}.
\end{proof}

Next we derive accurate estimates for $\mathscr{F}(u^\xi)=\epsilon^2u_{xx}^\xi-W_u(u^\xi)$:
\begin{lemma}
\label{est-Fuxi}
We have
\begin{equation}
\begin{split}
&\vert\mathscr{F}(u^\xi)\vert\leq Ce^{-\frac{\alpha\rho}{\eps}}\max_{h,\pm}{\mathscr{E}_h^\pm}^\frac{1}{2},
\;\;x\in[0,1],\\
&\vert\mathscr{F}_{\xi_j}(u^\xi)\vert\leq \frac{C}{\eps}\max_{h,\pm}\mathscr{E}_h^\pm,
\;\;x\in[0,1].\\
\end{split}
\label{scrF}
\end{equation}

\end{lemma}
\begin{proof}

1. Set $U(x)=\bar{u}_{j-1}( \frac{x-\xi_{j-1}}{\eps})-a_j
+\bar{u}_{j+1}( \frac{x-\xi_{j+1}}{\eps})-a_{j+1}$.
From \eqref{generic} we have
\begin{equation}
\begin{split}
&x\geq\xi_{j-1}\;\;\Rightarrow\\
&\vert\bar{u}_{j-1}(\frac{x-\xi_{j-1}}{\epsilon})-a_j\vert\leq Ce^{-\frac{\mu_j^+}{\epsilon}(x-\xi_{j-1})},\\
&x\leq\xi_{j+1}\;\;\Rightarrow\\
&\vert\bar{u}_{j+1}(\frac{x-\xi_{j+1}}{\epsilon})-a_{j+1}\vert\leq Ce^{-\frac{\mu_{j+1}^-}{\epsilon}(\xi_{j+1}-x)}.\\
\end{split}
\label{uj-aj}
\end{equation}
It follows
\begin{equation}
\begin{split}
&\vert U\vert\leq C({\mathscr{E}_j^+}^\frac{1}{2}+{\mathscr{E}_{j+1}^-}^\frac{1}{2}),\;\;x\in[\hat{\xi}_{j-1},\hat{\xi}_j],\\
&\vert U\vert^2\leq C(\mathscr{E}_j^++\mathscr{E}_{j+1}^-),\;\;x\in[\hat{\xi}_{j-1},\hat{\xi}_j]
\end{split}
\label{Uquadro}
\end{equation}

From \eqref{uxi1} we have
\begin{equation}
\begin{split}
&-\mathscr{F}(u^\xi)=W_{u}(\bar{u}_j(\frac{x-\xi_j}{\eps})+U+\tau_j)
-\bar{u}_j^{\prime\prime}(\frac{x-\xi_j}{\eps})-\epsilon^2U_{xx}-\epsilon^2\tau_{jxx}\\
&=W_{u}(\bar{u}_j(\frac{x-\xi_j}{\eps})+U)-W_{u}(\bar{u}_j(\frac{x-\xi_j}{\eps}))-\epsilon^2U_{xx}+\tau_j^1
\;\;x\in[\hat{\xi}_{j-1},\hat{\xi}_j].
\end{split}
\label{beta0}
\end{equation}
where we have used $\bar{u}_j^{\prime\prime}=W_{u}(\bar{u}_j)$ and denoted $\tau_j^1$ a quantity that satisfies \eqref{tau-est}. We have
\begin{equation}
\begin{split}
&-\mathscr{F}(u^\xi)=\int_0^1W_{uu}(\bar{u}_j(\frac{x-\xi_j}{\eps})+sU)Uds-\epsilon^2U_{xx}+\tau_j^1\\
&=W_{uu}(\bar{u}_j(\frac{x-\xi_j}{\eps}))U-\epsilon^2U_{xx}+\mathrm{O}(\vert U\vert^2)+\tau_j^1,
\;\;x\in[\hat{\xi}_{j-1},\hat{\xi}_j],
\end{split}
\label{beta.}
\end{equation}
and using \eqref{Uquadro}
\begin{equation}
-\mathscr{F}(u^\xi)=W_{uu}(\bar{u}_j(\frac{x-\xi_j}{\eps}))U-\epsilon^2U_{xx}
+\mathrm{O}(\mathscr{E}_j^++\mathscr{E}_{j+1}^-)+\tau_j^1,
\;\;x\in[\hat{\xi}_{j-1},\hat{\xi}_j].
\label{beta}
\end{equation}

We now continue the analysis of $\mathscr{F}(u^\xi)$ separately in the intervals $[\hat{\xi}_{j-1},\xi_j]$ and $[\xi_j,\hat{\xi}_j]$. We observe that, since ${\mu_h^\pm}^2$ is the eigenvalue of $W_{uu}(a_h)$ associated to the eigenvector $z_h^\pm$, \eqref{generic} implies
\begin{equation}
\begin{split}
&\bar{u}_h^{\prime\prime}(s)=W_{uu}(a_h)(\bar{u}_h(s)-a_h)+\mathrm{O}(e^{\hat{\mu}_h^-s}),\;\;s\leq 0,\\
&\bar{u}_{h-1}^{\prime\prime}(s)=W_{uu}(a_h)(\bar{u}_{h-1}(s)-a_h)+\mathrm{O}(e^{-\hat{\mu}_{h-1}^+s}),\;\;s\geq 0.
\end{split}
\label{2-der}
\end{equation}

Writing simply $\bar{u}_h$ instead of $\bar{u}_h(\frac{x-\xi_h}{\eps})$ we have
\begin{equation}
\begin{split}
&W_{uu}(\bar{u}_j)U-\epsilon^2U_{xx}
=W_{uu}(a_j)U-\epsilon^2U_{xx}+\mathrm{Int}\\
&=W_{uu}(a_j)(\bar{u}_{j-1}-a_j)-\bar{u}_{j-1}^{\prime\prime}
+W_{uu}(a_j)(\bar{u}_{j+1}-a_{j+1})-\bar{u}_{j+1}^{\prime\prime}+\mathrm{Int},
\end{split}
\label{beta-sin}
\end{equation}
where
\begin{equation}
\begin{split}
&\vert\mathrm{Int}\vert=\vert \int_0^1W_{uu}(a_j+s(\bar{u}_j-a_j))(\bar{u}_j-a_j)dsU\vert\\
&\leq C\vert\bar{u}_j-a_j\vert\vert U\vert\leq
 Ce^{\frac{\mu_j^-}{\eps}(x-\xi_j)}(e^{-\frac{\mu_j^+}{\eps}(x-\xi_{j-1})}
+e^{\frac{\mu_{j+1}^-}{\eps}(x-\xi_{j+1})})\\
&\leq C(\mathscr{E}_j+\mathscr{E}_j^+
+\mathscr{E}_{j+1}^-),\;\;x\in[\hat{\xi}_{j-1},\xi_j].
\end{split}
\label{Int}
\end{equation}
where we have assumed $x\in[\hat{\xi}_{j-1},\xi_j]$ and also used \eqref{uj-aj}.
From \eqref{2-der} and \eqref{generic}, we also have
\[
\begin{split}
&\vert W_{uu}(a_j)(\bar{u}_{j-1}-a_j)-\bar{u}_{j-1}^{\prime\prime}\vert\leq Ce^{-\frac{\hat{\mu}_{j-1}^+}{\eps}(x-\xi_{j-1})}\\
&\leq Ce^{-\frac{\hat{\mu}_{j-1}^+}{2\eps}(\xi_j-\xi_{j-1})}
\leq Ce^{-\frac{\hat{\mu}_{j-1}^+-\mu_j^+}{2\eps}(\xi_j-\xi_{j-1})}{\mathscr{E}_j^+}^\frac{1}{2},\;\;x\in[\hat{\xi}_{j-1},\xi_j],
\end{split}
\]
 and
\[
\begin{split}
&\vert W_{uu}(a_j)(\bar{u}_{j+1}-a_{j+1})\vert+\vert\bar{u}_{j+1}^{\prime\prime}\vert\\
&\leq Ce^{\frac{\mu_{j+1}^-}{\eps}(x-\xi_{j+1})}\leq C\mathscr{E}_{j+1}^-,
\;\; x\in[\hat{\xi}_{j-1},\xi_j].
\end{split}
\]

These estimates and \eqref{beta-sin} yield
\begin{equation}
\begin{split}
&\vert W_{uu}(\bar{u}_j)U-\epsilon^2U_{xx}\vert\leq C(e^{-\frac{\hat{\mu}_{j-1}^+-\mu_j^+}{2\eps}(\xi_j-\xi_{j-1})}{\mathscr{E}_j^+}^\frac{1}{2}+\mathscr{E}_j
+\mathscr{E}_j^++\mathscr{E}_{j+1}^-),\;\;x\in[\hat{\xi}_{j-1},\xi_j],
\end{split}
\label{beta-sin1}
\end{equation}

A similar computation gives
\begin{equation}
\begin{split}
&\vert W_{uu}(\bar{u}_j)U-\epsilon^2U_{xx}\vert\leq C(e^{-\frac{\hat{\mu}_{j+1}^--\mu_{j+1}^-}{2\eps}(\xi_{j+1}-\xi_j)}{\mathscr{E}_{j+1}^-}^\frac{1}{2}+\mathscr{E}_{j+1}
+\mathscr{E}_j^++\mathscr{E}_{j+1}^-),\;\;x\in[\xi_j,\hat{\xi}_{j+1}].
\end{split}
\label{beta-des1}
\end{equation}

From \eqref{beta}, \eqref{beta-sin1} and \eqref{beta-des1} we finally conclude
\[\vert\mathscr{F}(u^\xi)\vert\leq Ce^{-\frac{\alpha\rho}{\eps}}\max_{h,\pm}{\mathscr{E}_h^\pm}^\frac{1}{2}
+\vert\tau_j^1\vert\leq Ce^{-\frac{\alpha\rho}{\eps}}\max_{h,\pm}{\mathscr{E}_h^\pm}^\frac{1}{2}
,\;\;x\in[\hat{\xi}_{j-1},\hat{\xi}_j].\]
where we have used \eqref{tau-est}. Since this is valid for all $j$ \eqref{scrF}$_1$ follows.

To estimate $\mathscr{F}_{\xi_j}(u^\xi)$ we differentiate the expression of $\mathscr{F}(u^\xi)$ given in \eqref{beta0} and observe that $U_{\xi_j}=U_{xx\xi_j}=0$ to obtain
\[
-\mathscr{F}_{\xi_j}(u^\xi)=-\Big(W_{uu}(\bar{u}_j+U+\tau_j)-W_{uu}(\bar{u}_j)\Big)
\frac{1}{\eps}\bar{u}_j^\prime
+W_{uu}(\bar{u}_j+U+\tau_j)\tau_{j\xi_j}+\epsilon^2\tau_{jxx\xi_j},
\]
where we have used $\bar{u}_j^{\prime\prime\prime}=W_{uu}(\bar{u}_j)\bar{u}_j^\prime$.
It follows
\[\vert\mathscr{F}_{\xi_j}(u^\xi)\vert\leq\frac{C}{\epsilon}\Big((\vert U\vert+\vert\tau_j\vert)\vert\bar{u}_j^\prime\vert+\epsilon\vert\tau_{j\xi_j}\vert+\epsilon^3\vert\tau_{jxx\xi_j}\vert\Big)\]
As in the estimate \eqref{Int} for $\mathrm{Int}$, from \eqref{generic} and \eqref{uj-aj} it follows
\[\begin{split}
&\vert U\vert\vert\bar{u}_j^\prime\vert\leq Ce^{\frac{\mu_j^-}{\epsilon}(x-\xi_j)}(e^{-\frac{\mu_j^+}{\epsilon}(x-\xi_{j-1})}
+e^{\frac{\mu_{j+1}^-}{\epsilon}(x-\xi_{j+1})})\\
&\leq C\max_{h,\pm}\mathscr{E}_h^\pm
,\;\;x\in[\hat{\xi}_{j-1},\xi_j],\\
&\vert U\vert\vert\bar{u}_j^\prime\vert\leq
Ce^{-\frac{\mu_j^+}{\epsilon}(x-\xi_j)}(e^{-\frac{\mu_j^+}{\epsilon}(x-\xi_{j-1})}
+e^{\frac{\mu_{j+1}^-}{\epsilon}(x-\xi_{j+1})})\\
&\leq C\max_{h,\pm}\mathscr{E}_h^\pm ,\;\;x\in[\xi_j,\hat{\xi}_j].
\end{split}\]
From this and the estimates for $\tau_j$, $\epsilon\tau_{j\xi_j}$ and $\epsilon^3\tau_{jxx\xi_j}$ the bound for $\mathscr{F}_{\xi_j}(u^\xi)$ follows.
The proof is complete
\end{proof}

\vskip.2cm
A first step in the analysis of $c(\xi)$ is an asymptotic formula for $\bar{c}(\xi)$ defined by
\begin{equation}
\bar{c}_j(\xi)=\langle\mathscr{F}(u^\xi),\frac{u_{\xi_j}^\xi}{\| u_{\xi_j}^\xi\|}\rangle,\;\;j=1,\ldots,N.
\label{cbar-def}
\end{equation}

\begin{lemma}
\label{barc-lem}
 There is $\epsilon_0>$ such that, for $\epsilon\in(0,\epsilon_0]$, it results
 \[\begin{split}
&\bar{c}_j(\xi)=\frac{2\epsilon^\frac{1}{2}}{\bar{q}_j}\Big(\varsigma_{j+1}k_{j+1}^+\mathscr{E}_{j+1}
-\varsigma_jk_j^-\mathscr{E}_j
\,+\mathrm{O}(e^{-\frac{\alpha\rho}{\eps}}\max_{h,\pm}\mathscr{E}_h^\pm)\Big),\;\;j=1,\ldots,N.
\end{split}\]

where $\varsigma_j$ and $k_j^\pm$ are defined in \eqref{varsigma} and  $\bar{q}_j^2=\int_\R\vert\bar{u}_j^\prime\vert^2ds$.
\end{lemma}
\begin{proof}In this proof, if there is no risk of confusion, we simply write $\bar{u}_h$ instead of $\bar{u}_h(\frac{x-\xi_h}{\eps})$.
 From \eqref{beta.} and the expression of $U$ we obtain
 \begin{equation}\begin{split}
 &\bar{c}_j(\xi)=\frac{1}{\| \epsilon u_{\xi_j}^\xi\|}(\int_0^1\mathscr{F}(u^\xi)\cdot\varkappa_j dx
 -\int_0^{\hat{\xi}_{j-1}}\mathscr{F}(u^\xi)\cdot\bar{u}_j^\prime dx-\int_{\hat{\xi}_j}^1\mathscr{F}(u^\xi)\cdot \bar{u}_j^\prime dx)\\
 &+\frac{1}{\|\epsilon u_{\xi_j}^\xi\|}
 \int_{\hat{\xi}_{j-1}}^{\hat{\xi}_j}\Big(W_{uu}(\bar{u}_j)(\bar{u}_{j-1}-a_j+\bar{u}_{j+1}-a_{j+1})\\
&-\bar{u}_{j-1}^{\prime\prime}-\bar{u}_{j+1}^{\prime\prime}
+\mathrm{O}(\vert\tau_j^1\vert+\vert U\vert^2)\Big)\cdot\bar{u}_j^\prime dx
\end{split}
\label{barcj}
\end{equation}
From Lemma \ref{est-uxi} and Lemma \ref{est-Fuxi} we have
\[\begin{split}
&\vert\mathscr{F}(u^\xi)\cdot\varkappa_j\vert\leq Ce^{-\frac{\alpha\rho}{\eps}}\max_{h,\pm}\mathscr{E}_h^\pm,\;\;x\in[0,1],\\
&\vert\mathscr{F}(u^\xi)\cdot\bar{u}_j^\prime\vert\leq C\vert\mathscr{F}(u^\xi)\vert e^{-\frac{\mu_j^-}{2\eps}(\xi_j-\xi_{j-1})}
\leq Ce^{-\frac{\alpha\rho}{\eps}}\max_{h,\pm}\mathscr{E}_h^\pm,\;\;x\leq\hat{\xi}_{j-1},\\
&\vert\mathscr{F}(u^\xi)\cdot\bar{u}_j^\prime\vert\leq C\vert\mathscr{F}(u^\xi)\vert e^{-\frac{\mu_{j+1}^+}{2\eps}(\xi_{j+1}-\xi_j)}
\leq Ce^{-\frac{\alpha\rho}{\eps}}\max_{h,\pm}\mathscr{E}_h^\pm,\;\;x\geq\hat{\xi}_j,\\
&\vert\tau_j^1\vert\vert\bar{u}_j^\prime\vert\leq Ce^{-\frac{\alpha\rho}{\eps}}\max_{h,\pm}\mathscr{E}_h^\pm,\;\;x\in[\hat{\xi}_{j-1},\hat{\xi}_j].
\end{split}\]
and by mean of \eqref{uj-aj} we find
\[\begin{split}
&\vert U\vert^2\vert\bar{u}_j^\prime\vert
\leq 2(\vert\bar{u}_{j-1}-a_j\vert^2+\vert\bar{u}_{j+1}-a_{j+1}\vert^2)\vert\bar{u}_j^\prime\vert\\
&\leq C(e^{-\frac{2\mu_j^+}{\eps}(x-\xi_{j-1})}+e^{\frac{2\mu_{j+1}^-}{\eps}(x-\xi_{j+1})})
e^{\frac{\mu_j^-}{\eps}(x-\xi_j)}\\
&\leq Ce^{-\frac{\alpha\rho}{\eps}}\max_{h,\pm}\mathscr{E}_h^\pm,\;\;x\in[\hat{\xi}_{j-1},\xi_j],
\end{split}\]
and in a similar way
\[\vert U\vert^2\vert\bar{u}_j^\prime\vert
\leq Ce^{-\frac{\alpha\rho}{\eps}}\max_{h,\pm}\mathscr{E}_h^\pm,\;\;x\in[\xi_j,\hat{\xi}_j].
\]
These estimates imply that we can rewrite \eqref{barcj} in the form
\begin{equation}
 \begin{split}
 &\bar{c}_j(\xi)=
 \frac{1}{\|\epsilon u_{\xi_j}^\xi\|}\Big(
 \int_{\hat{\xi}_{j-1}}^{\hat{\xi}_j}\Big(W_{uu}(\bar{u}_j)(\bar{u}_{j-1}-a_j+\bar{u}_{j+1}-a_{j+1})
-\bar{u}_{j-1}^{\prime\prime}-\bar{u}_{j+1}^{\prime\prime}
\Big)\cdot\bar{u}_j^\prime dx\\
&+\mathrm{O}(e^{-\frac{\alpha\rho}{\eps}}\max_{h,\pm}\mathscr{E}_h^\pm\Big)=\frac{1}{\|\epsilon u_{\xi_j}^\xi\|}(I+\mathrm{O}(e^{-\frac{\alpha\rho}{\eps}}\max_{h,\pm}\mathscr{E}_h^\pm)),
\end{split}
\label{barcj1}
\end{equation}
with obvious definition of $I$.
Since $W_{uu}(u)$ is a symmetric matrix and
 $\bar{u}_j^{\prime\prime\prime}= W_{uu}(\bar{u}_j)\bar{u}_j^\prime$, we can rewrite $I$ as
\[\begin{split}
&I= \int_{\hat{\xi}_{j-1}}^{\hat{\xi}_j}\Big(W_{uu}(\bar{u}_j)\bar{u}_j^\prime\cdot(\bar{u}_{j-1}-a_j+\bar{u}_{j+1}-a_{j+1})
-(\bar{u}_{j-1}^{\prime\prime}+\bar{u}_{j+1}^{\prime\prime})
\cdot\bar{u}_j^\prime\Big) dx\\
&=\int_{\hat{\xi}_{j-1}}^{\hat{\xi}_j}\Big(\bar{u}_j^{\prime\prime\prime}\cdot(\bar{u}_{j-1}-a_j+\bar{u}_{j+1}-a_{j+1})
-(\bar{u}_{j-1}^{\prime\prime}+\bar{u}_{j+1}^{\prime\prime})
\cdot\bar{u}_j^\prime\Big) dx.
\end{split}\]
Set
\[g=\bar{u}_j^{\prime\prime}\cdot
(\bar{u}_{j-1}-a_j+\bar{u}_{j+1}-a_{j+1})
-\bar{u}_j^{\prime}\cdot
(\bar{u}_{j-1}^{\prime}+\bar{u}_{j+1}^{\prime}).\]
We regard $g$ as a function of $s=\frac{x}{\epsilon}$ and we have $\epsilon\frac{d}{dx}g=g^\prime$.
Note that
\[g^\prime=\bar{u}_j^{\prime\prime\prime}\cdot(\bar{u}_{j-1}-a_j+\bar{u}_{j+1}-a_{j+1})
-(\bar{u}_{j-1}^{\prime\prime}+\bar{u}_{j+1}^{\prime\prime})
\cdot\bar{u}_j^\prime,\]
that is $g^\prime$ coincides with the integrand of $I$. It follows
\begin{equation}
I=\int_{\hat{\xi}_{j-1}}^{\hat{\xi}_j}\epsilon\frac{d}{dx}gdx=\epsilon g\vert_{\hat{\xi}_{j-1}}^{\hat{\xi}_j}.
\label{I-g}
\end{equation}

From \eqref{generic} we obtain 

\[\begin{split}
&\bar{u}_j^{\prime\prime}\vert_{x={\hat{\xi}_j}}=z_{j+1}^+{\mu_{j+1}^+}^2\bar{K}_j^+
e^{-\frac{\mu_{j+1}^+}{2\epsilon}(\xi_{j+1}-\xi_j)}
+\mathrm{O}(e^{-\frac{\hat{\mu}_j^+-\mu_{j+1}^+}{2\epsilon}(\xi_{j+1}-\xi_j)}
e^{-\frac{\mu_{j+1}^+}{2\epsilon}(\xi_{j+1}-\xi_j)})\\
&=z_{j+1}^+{\mu_{j+1}^+}^2\bar{K}_j^+{\mathscr{E}_{j+1}^+}^\frac{1}{2}
+\mathrm{O}(e^{-\frac{\alpha\rho}{\eps}}{\mathscr{E}_{j+1}^+}^\frac{1}{2})\\
&(\bar{u}_{j+1}-a_{j+1})\vert_{x={\hat{\xi}_j}}=z_{j+1}^-\bar{K}_{j+1}^-
e^{-\frac{\mu_{j+1}^-}{2\epsilon}(\xi_{j+1}-\xi_j)}
+\mathrm{O}(e^{-\frac{\hat{\mu}_{j+1}^--\mu_{j+1}^-}{2\epsilon}(\xi_{j+1}-\xi_j)}e^{-\frac{\mu_{j+1}^-}{2\epsilon}(\xi_{j+1}-\xi_j)})
\\
&=z_{j+1}^-\bar{K}_{j+1}^-
{\mathscr{E}_{j+1}^-}^\frac{1}{2}
+\mathrm{O}(e^{-\frac{\alpha\rho}{\eps}}{\mathscr{E}_{j+1}^-}^\frac{1}{2})
\\\\
&-\bar{u}_j^{\prime}\vert_{x={\hat{\xi}_j}}=z_{j+1}^+\mu_{j+1}^+\bar{K}_j^+
{\mathscr{E}_{j+1}^+}^\frac{1}{2}
+\mathrm{O}(e^{-\frac{\alpha\rho}{\eps}}{\mathscr{E}_{j+1}^+}^\frac{1}{2})\\
&\bar{u}_{j+1}^{\prime}\vert_{x={\hat{\xi}_j}}=z_{j+1}^-\mu_{j+1}^-\bar{K}_{j+1}^-
{\mathscr{E}_{j+1}^-}^\frac{1}{2}
+\mathrm{O}(e^{-\frac{\alpha\rho}{\eps}}{\mathscr{E}_{j+1}^-}^\frac{1}{2}).
\end{split}\]
Since from \eqref{Eh} ${\mathscr{E}_h^-}^\frac{1}{2}{\mathscr{E}_h^+}^\frac{1}{2}=\mathscr{E}_h$, it follows
\[\begin{split}
&\bar{u}_j^{\prime\prime}\cdot(\bar{u}_{j+1}-a_{j+1})\vert_{x={\hat{\xi}_j}}=
\varsigma_{j+1}{\mu_{j+1}^+}^2\bar{K}_{j+1}^-\bar{K}_j^+\mathscr{E}_{j+1}
+\mathrm{O}(e^{-\frac{\alpha\rho}{\eps}}\mathscr{E}_{j+1})\\\\
&-\bar{u}_j^{\prime}\cdot\bar{u}_{j+1}^{\prime}\vert_{x={\hat{\xi}_j}}=
\varsigma_{j+1}\mu_{j+1}^+\mu_{j+1}^-\bar{K}_{j+1}^-\bar{K}_j^+\mathscr{E}_{j+1}
+\mathrm{O}(e^{-\frac{\alpha\rho}{\eps}}\mathscr{E}_{j+1}),
\end{split}\]
and recalling the definition \eqref{varsigma} of $k_j^\pm$ we finally obtain
\[\begin{split}
&(\bar{u}_j^{\prime\prime}\cdot
(\bar{u}_{j+1}-a_{j+1})
-\bar{u}_j^{\prime}\cdot
\bar{u}_{j+1}^{\prime})\vert_{x={\hat{\xi}_{j}}}=\varsigma_{j+1}2k_{j+1}^+\mathscr{E}_{j+1}
+\mathrm{O}(e^{-\frac{\alpha\rho}{\eps}}\mathscr{E}_{j+1}).
\end{split}\]
A similar computation yields

\[\begin{split}
&(\bar{u}_j^{\prime\prime}\cdot
(\bar{u}_{j-1}-a_j)
-\bar{u}_j^{\prime}\cdot
\bar{u}_{j-1}^{\prime})
\vert_{x=\hat{\xi}_{j-1}}=\varsigma_j2k_j^-\mathscr{E}_j
+\mathrm{O}(e^{-\frac{\alpha\rho}{\eps}}\mathscr{E}_{j}),
\end{split}\]
and we also have
\[\begin{split}
&(\bar{u}_j^{\prime\prime}\cdot
(\bar{u}_{j+1}-a_{j+1})
-\bar{u}_j^{\prime}\cdot
\bar{u}_{j+1}^{\prime})\vert_{x={\hat{\xi}_{j-1}}}
=\mathrm{O}(e^{-\frac{\mu_j^-+\mu_{j+1}^-}{2\epsilon}(\xi_j-\xi_{j-1})}
e^{-\frac{\mu_{j+1}^-}{\epsilon}(\xi_{j+1}-\xi_j)})\\
&=\mathrm{O}(e^{-\frac{\alpha\rho}{\eps}}\mathscr{E}_{j+1}^-)\\
&(\bar{u}_j^{\prime\prime}\cdot
(\bar{u}_{j-1}-a_j)
-\bar{u}_j^{\prime}\cdot
\bar{u}_{j-1}^{\prime})
\vert_{x=\hat{\xi}_j}
=\mathrm{O}(e^{-\frac{\mu_j^++\mu_{j+1}^+}{2\epsilon}(\xi_{j+1}-\xi_{j})}
e^{-\frac{\mu_{j}^+}{\epsilon}(\xi_{j}-\xi_{j-1})})\\
&=\mathrm{O}(e^{-\frac{\alpha\rho}{\eps}}\mathscr{E}_j^+)
\end{split}\]

From these estimates \eqref{I-g} and \eqref{barcj} we get
\begin{equation}
\begin{split}
&\bar{c}_j(\xi)=\frac{2}{\| u_{\xi_j}^\xi\|}\Big(\varsigma_{j+1}k_{j+1}\mathscr{E}_{j+1}
-\varsigma_jk_j\mathscr{E}_j
\,+\mathrm{O}(e^{-\frac{\alpha\rho}{\eps}}\max_{h,\pm}\mathscr{E}_h^\pm)\Big),\;\;j=1,\ldots,N.
\end{split}
\label{scripC}
\end{equation}
This and \eqref{uxixi} complete the  proof.
\end{proof}



To obtain a good approximation to $c(\xi)$ valid for $\eps>0$ small, the estimate for $\eta_j^\xi$ in (\ref{approxphi}) is not sufficient. We need the following refinement,

\begin{lemma}\label{gammasquare0}
Let $X^\xi$ the eigen-space introduced in Proposition \ref{SpectrumL}. There exist vectors $\eta_1^\xi,\ldots,\eta_N^\xi\in W^{1,2}((0,1);\R^m)$ such that
\begin{enumerate}
\item
$\|\eta_j^\xi\|\leq\frac{C}{\eps}\max_{h,\pm}\mathscr{E}_h^\pm.$
\item
$\varphi_j^\xi=\frac{u_{\xi_j}^\xi}{\|u_{\xi_j}^\xi\|}+\eta_j^\xi,\;\;j=1,\ldots,N
$ is an orthonormal basis for $X^\xi$.
\end{enumerate}
\end{lemma}
\begin{proof}
See Section \ref{Pf}
\end{proof}




From \eqref{scrF} and Lemma \ref{gammasquare0} it follows
\[\begin{split}
&\|\mathscr{F}(u^\xi)\|^2\leq Ce^{-\frac{\alpha\rho}{\epsilon}}\max_{h,\pm}\mathscr{E}_h^\pm,\\
&\vert\langle\mathscr{F}(u^\xi),\eta_j^\xi\rangle\vert
\leq Ce^{-\frac{\alpha\rho}{\epsilon}}\max_{h,\pm}\mathscr{E}_h^\pm).
\end{split}\]
 This \eqref{cvest} and Lemma \ref{barc-lem} yield $c_j(\xi)-\bar{c}_j(\xi)=\mathrm{O}(e^{-\frac{\alpha\rho}{\epsilon}}\max_{h,\pm}\mathscr{E}_h^\pm)$ and we can conclude that

\begin{equation}
\begin{split}
& c_j(\xi)
=c_j^0(\xi)
+r_j(\xi)\\
& r_j(\xi)=\mathrm{O}(e^{-\frac{\alpha\rho}{\eps}}
\max_{h,\pm}\mathscr{E}_h^\pm),\;\;j=1,\ldots,N.
\end{split}
\label{cest}
\end{equation}

where $c_j^0(\xi)$ is defined in \eqref{c0-gen}. This conclude the proof of Theorem \ref{TH-c}.
\end{proof}

Note that, since, as we have observed,
$c(\xi)$ depends on $\xi$ through the differences $\xi_j-\xi_{j-1}$, $j=1,\ldots,N$ and therefore is a function of only $N-1$ variables, equation \eqref{cest} implies that also the error $r_j(\xi)$ in the estimate \eqref{cest} depends only on $\xi_j-\xi_{j-1}$, $j=1,\ldots,N$.

We can now state
\begin{theorem}
\label{per-exists}
Assume $\bf{H}_1-\bf{H}_4$, then there exists $\epsilon_0>0$ such that, provided $\epsilon\in(0,\epsilon_0]$, the condition
\begin{equation}
\varsigma_j=\varsigma_{j+1},\;\;j=1,\ldots,N
\label{sigmaproduct}
\end{equation}
is necessary and sufficient in order that there is $\xi\in\Xi_\rho$ with the property that $\hat{u}^\xi=u^\xi+v^\xi$, $v^\xi$ as in Proposition \ref{quasiinv1}, is a periodic solution of period $1$ of \eqref{rescaled}. The vector $\xi\in\Xi_\rho$ that determines the periodic orbit satisfies
\begin{equation}
\xi_j-\xi_{j-1}=\frac{1/\mu_j}{\sum_{h=1}^N1/\mu_h}+\epsilon\delta_j(\xi),
\;\;j=1,\ldots,N,\;\;\sum_{j=1}^N\delta_j=0.
\label{xi-det}
\end{equation}
where $\delta_j=\bar{\delta}_j+\mathrm{O}(e^{-\frac{\alpha\rho}{\eps}})$ and $\bar{\delta}_j$ are constants that depend only on $\mu_1,\ldots,\mu_N$  and  $k_1,\ldots,k_N$.
\end{theorem}
\begin{proof}
From Theorem \ref{TH-c} we have $c(\xi)=0$ if and only if
\begin{equation}
\varsigma_{j+1}k_{j+1}\mathscr{E}_{j+1}=\varsigma_jk_j\mathscr{E}_j+\frac{\bar{q}_jr_j(\xi)}{2\epsilon^\frac{1}{2}},
\;\;j=1,\ldots,N.
\label{c=0}
\end{equation}
Since we are assuming $\bf{H}_4$, \eqref{c0-gen} implies that at least one of the $\mathscr{E}_j$ coincides with $\max_h\mathscr{E}_h$ and that, for small $\epsilon>0$, $r_j(\xi)$ is much smaller than $\max_h\mathscr{E}_h$. From this and $\varsigma_j=\pm 1$ it follows that \eqref{sigmaproduct} is a necessary condition for the existence of a solution $\xi\in\Xi$ of \eqref{c=0} and it results
\begin{equation}
C^-\max_h\mathscr{E}_h\leq\mathscr{E}_j\leq C^+\max_h\mathscr{E}_h,
\label{E=Emax}
\end{equation}
for some constants $C^\pm>0$ independent of $\epsilon$. This proves the necessity of condition \eqref{sigmaproduct}. The prove the sufficiency we note that, with \eqref{sigmaproduct}, equation \eqref{c=0} becomes
\begin{equation}
k_{j+1}\mathscr{E}_{j+1}=k_j\mathscr{E}_j+\frac{\bar{q}_j\varsigma_jr_j(\xi)}{2\epsilon^\frac{1}{2}},
\;\;j=1,\ldots,N.
\label{c=01}
\end{equation}
From Proposition \ref{N-1=N} to show that there exists $\xi\in\Xi$ that solves the equation $c(\xi)=0$ it suffices to show that we can solve the first $N-1$ of the equations \eqref{c=01}. These equations constitute a linear system in the unknowns $k_j\mathscr{E}_j$, $j=2,\ldots,N$ that can be solved yielding
\begin{equation}
k_j\mathscr{E}_j=k_1\mathscr{E}_1+\frac{\varsigma_1}{2\epsilon^\frac{1}{2}}
\sum_{h=1}^{j-1}\bar{q}_hr_h(\xi),\;\;j=2,\ldots,N.
\label{c=02}
\end{equation}
If we take the logarithm of these equations and multiply by $-\eps$ we get
\begin{equation}
\begin{split}
&\mu_j(\xi_j-\xi_{j-1})=\mu_1(\xi_j-\xi_0)+\epsilon\ln{\frac{k_j}{k_1}}+\epsilon\tilde{r}_j(\xi),\;\;j=2,\ldots,N.\\
&\tilde{r}_j(\xi)=-\ln{(1+
\frac{\varsigma_1}{2\epsilon^\frac{1}{2}k_1\mathscr{E}_1}
\sum_{h=1}^{j-1}\bar{q}_hr_h(\xi))}.
\end{split}
\label{c=03}
\end{equation}
We introduce the new variable $\delta\in\R^N$ by setting
\begin{equation}
\begin{split}
&\xi_j-\xi_{j-1}=\frac{\frac{1}{\mu_j}}{\sum_h\frac{1}{\mu_h}}+\epsilon\delta_j,\;\;j=1,\ldots,N,\\
&\sum_{j=1}^N\delta_j=0,
\end{split}
\label{new-var}
\end{equation}
where \eqref{new-var}$_2$ follows from $\sum_{j=1}^N(\xi_j-\xi_{j-1})=1$.
Since, as $r(\xi)$, $\tilde{r}(\xi)$ depends on $\xi$ through the differences $\xi_j-\xi_{j-1}$, $j=1,\ldots,N$ the function $\tilde{r}(\xi(\delta))$ is well defined via the change of variables \eqref{new-var}. We also observe that, since $\rho>0$ in the definition \eqref{basic} of $\Xi$ is a small fixed number, given $\delta_0>0$, for small $\epsilon>0$, we have
\begin{equation}
\frac{\frac{1}{\mu_j}}{\sum_h\frac{1}{\mu_h}}+\epsilon\delta_j>\rho,
\;\;\vert\delta\vert\leq\delta_0,\,\;j=1,\ldots,N,
\label{new-var1}
\end{equation}
and therefore \eqref{c0-gen}, \eqref{c=03}$_2$ and \eqref{E=Emax} imply
\begin{equation}
\tilde{r}(\xi(\delta))=\mathrm{O}(e^{-\frac{\alpha\rho}{\epsilon}}),\;\;\vert\delta\vert\leq\delta_0.
\label{rdelta-bound}
\end{equation}
 With the new variable, after dividing by $\eps\mu_j$,  \eqref{c=03} becomes
\begin{equation}
\begin{split}
&\delta_j=\frac{\mu_1}{\mu_j}\delta_1+\frac{1}{\mu_j}\ln{\frac{k_j}{k_1}}+\frac{1}{\mu_j}\tilde{r}_j(\xi(\delta)),
\;\;j=2,\ldots,N,\\
&\sum_{j=1}^N\delta_j=0.
\end{split}
\label{c=04}
\end{equation}
We make a further change of variable by setting $\delta=\bar{\delta}+\delta^*$ where $\bar{\delta}$ is the solution of \eqref{c=04} corresponding to $\epsilon=0$:
\[\begin{split}
&\bar{\delta}_1=-\frac{1}{\mu_1}\frac{\sum_{h=2}^N\frac{1}{\mu_h}\ln{\frac{k_h}{k_1}}}{\sum_{h=1}^N\frac{1}{\mu_h}},\\
&\bar{\delta}_j=\frac{\mu_1}{\mu_j}\bar{\delta}_1+\frac{1}{\mu_j}\ln{\frac{k_j}{k_1}}
,\;\;j=2,\ldots,N.
\end{split}\]
If we insert $\delta=\bar{\delta}+\delta^*$ in \eqref{c=04} we have
\begin{equation}
\begin{split}
&\delta_j^*=\frac{1}{\mu_j}\tilde{r}_j(\xi(\bar{\delta}+\delta^*)),
\;\;j=2,\ldots,N,\\
&\delta_1^*=-\sum_{j=2}^N\delta_j^*.
\end{split}
\label{c=05}
\end{equation}
\vskip.3cm
Set $B_1=\{\delta^*\in\R^N:\vert\delta^*\vert\leq 1,\;\sum_h\delta_h^*=0\}$ and let $F:B_1\rightarrow\R^N$ be defined by
setting $F_j(\delta^*)=\tilde{r}_j(\xi(\bar{\delta}+\delta^*))$, $j=2,\ldots,N$,  $F_1(\delta^*)=-\sum_{j=2}^N\tilde{r}_j(\xi(\bar{\delta}+\delta^*))$. $F$ is a continuous function that, by \eqref{rdelta-bound} satisfies $F(\delta^*)=\mathrm{O}(e^{-\frac{\alpha\rho}{\eps}})$. It follows that $F$ is a continuous map that maps $B_1$ into itself and Brouwer fixed point theorem implies that $F$ has a fixed in $B_1$ that is a solution of \eqref{c=05} that we still denote $\delta^*$. Moreover from \eqref{rdelta-bound} it follows
 that $\delta^*=\mathrm{O}(e^{-\frac{\alpha\rho}{\eps}})$. This concludes the proof.
\end{proof}

Theorem \ref{CNS} is essentially equivalent to Theorem \ref{per-exists}. Theorem \ref{per-exists} shows that under the assumptions of Theorem \ref{CNS} there exists a $\xi\in\Xi_\rho$ such that, for small $\epsilon>0$, $\hat{u}=u^\xi+v^\xi$, $v^\xi$ as in Proposition \ref{quasiinv1}, is a $1$-periodic solution of \eqref{rescaled}. If we set $\epsilon=\frac{1}{T}$ and $x=\frac{t}{T}$, then $u^T(t)=\hat{u}(\frac{t}{T})$ is a $T$- periodic solution of \eqref{newton}. From \eqref{uxi1}, \eqref{tau-est}, $\|v^\xi\|_{W_\epsilon^{1,2}}\leq C\|\mathscr{F}(u^\xi)\|$ and \eqref{scrF} it follows

\begin{equation}
\begin{split}
&\vert\hat{u}(x)-\bar{u}_j(\frac{x-\xi_j}{\epsilon})\vert\leq C(e^{-\frac{\mu_j}{\eps}(x-\xi_{j-1})}
+e^{\frac{\mu_{j+1}}{\eps}(x-\xi_{j+1})}
+e^{-\frac{\alpha\rho}{\eps}}\max_h\mathscr{E}_h^\frac{1}{2})\\
&\leq Ce^{-\frac{1}{2\epsilon\sum_h\frac{1}{\mu_h}}}),\;\;x\in[\hat{\xi}_{j-1},\hat{\xi}_j],\;j=1,\ldots,N,
\end{split}
\label{hatu-aj}
\end{equation}
where we have also used the characterization \eqref{xi-det} of $\xi$. In terms of $u^T$ \eqref{hatu-aj}  becomes
\begin{equation}
u^T(t)=\bar{u}_j(t-\xi_jT)+\mathrm{O}(e^{-\frac{T}{2\sum_h\frac{1}{\mu_h}}}),
\;\;t\in[\hat{\xi}_{j-1}T,\hat{\xi}_jT],\;j=1,\ldots,N,
\label{hatu-aj1}
\end{equation}
Statement (i) of Theorem \ref{CNS} follows from \eqref{xi-det} with $\epsilon=\frac{1}{T}$ and from \eqref{hatu-aj1} with the change of variable $t\rightarrow t+\xi_jT$. To prove (ii) we note that \eqref{generic} implies the existence of a constant $C>0$ such that for small $\delta>0$ we have
\[\begin{array}{l}
\vert\bar{u}_{j-1}-a_j\vert\leq\frac{\delta}{2},\;\;s\geq -C\ln{\delta},\\
\vert\bar{u}_j-a_j\vert\leq\frac{\delta}{2},\;\;s\leq C\ln{\delta}.
\end{array},\;\;j=1,\ldots,N.\]
This and \eqref{hatu-aj1} imply (ii) with $t_\delta=-C\ln{\delta}$.
The last statement of Theorem \ref{CNS} follows from $\|v^\xi\|_{W_\epsilon^{1,2}}=\mathrm{O}(e^{-\frac{\alpha\rho}{\eps}}e^{-\frac{1}{2\eps\sum_h\frac{1}{\mu_h}}})$.
The proof of Theorem \ref{CNS} is complete.



\subsection{Proofs of Lemma \ref{est-uxi} and Lemma \ref{gammasquare0}.}
\label{Pf}
\begin{proof} (Lemma \ref{est-uxi}) We assume $[\xi_{j-1},\xi_{j+1}]\subset[0,1]$. The estimates that we derive are valid in any case.

 We rearrange $\tau_j$ in the form
\begin{equation}
\begin{split}
&\tau_j=\sum_{h<j-1}\sum_{n\leq 0}\Big( \bar{u}_h( \frac{x-n-\xi_h}{\eps})-a_{h+1} \Big)
+\sum_{h\geq j-1}\sum_{n<0}\Big( \bar{u}_h( \frac{x-n-\xi_h}{\eps})-a_{h+1} \Big)\\
&+\sum_{h>j+1}\sum_{n\geq 0}\Big( \bar{u}_h( \frac{x-n-\xi_h}{\eps})-a_h \Big)
+\sum_{h\leq j+1}\sum_{n>0}\Big( \bar{u}_h( \frac{x-n-\xi_h}{\eps})-a_h \Big)\\
&=\sum_{h<j-1}\tilde{\tau}_h^++\sum_{h\geq j-1}\hat{\tau}_h^+
+\sum_{h>j+1}\tilde{\tau}_h^-+\sum_{h\leq j+1}\hat{\tau}_h^-,
 \end{split}
\label{tau1}
\end{equation}
with obvious definition of $\tilde{\tau}_h^\pm$ and $\hat{\tau}_h^\pm$. We only analyze the last two summations. The other two can be estimate in the same way.

1. Estimating $\sum_{h>j+1}\tilde{\tau}_h^-$.

Note that
\[\begin{split}
&h>j+1\;\;\text{and}\;\;x\leq\hat{\xi}_j\;\Rightarrow\;\;
 x-n-\xi_h\leq\frac{\xi_{j+1}+\xi_j}{2}-n-\xi_h\\
&=-(\frac{\xi_{j+1}-\xi_j}{2}+\xi_h-\xi_{j+1}+n)\leq-(\frac{\xi_{j+1}-\xi_j}{2}+\xi_h-\xi_{h-1}+n) .
\end{split}\]
This and \eqref{generic}, observing also that the contribution of $\bar{w}_h^-$ can be absorbed in the constant $C$, imply
\[\begin{split}
&\vert\bar{u}_h( \frac{x-n-\xi_h}{\eps})-a_h\vert\\
&\leq Ce^{-\frac{\mu_h^-}{\eps}\frac{\xi_{j+1}-\xi_j}{2}}e^{-\frac{\mu_h^-}{\eps}(\xi_h-\xi_{h-1})}e^{-\frac{\mu_h^-}{\eps}n},
\;\;n=0,1,\ldots,\;x\leq\hat{\xi}_j.\end{split}\]
 It follows

\[\vert\tilde{\tau}_h^-\vert\leq Ce^{-\frac{\mu_h^-}{\eps}\frac{\xi_{j+1}-\xi_j}{2}}\mathscr{E}_h^-,\;x\leq\hat{\xi}_j,\;h>j+1,\;
,\]
where we have also used \eqref{Eh}.

2. Estimating $\sum_{h\leq j+1}\hat{\tau}_h^-$.
We have
\[\begin{split}
&h\leq j+1\;\;\text{and}\;\;x\leq\hat{\xi}_j\;\Rightarrow\;\; x-n-\xi_h\leq\frac{\xi_{j+1}+\xi_j}{2}-\xi_h-n\\
&=-(\frac{\xi_{j+1}-\xi_j}{2}+\xi_h-\xi_{j+1}+n)
\leq-(\frac{\xi_{j+1}-\xi_j}{2}+\xi_h-\xi_{h-1}+(n-1)),
\end{split}\]
where we have used that $h\leq j+1$ implies $\xi_h-\xi_{j+1}+1\geq\xi_h-\xi_{h-1}$.
This and \eqref{generic} yield
\[\vert \bar{u}_h( \frac{x-n-\xi_h}{\eps})-a_h\vert\leq e^{-\frac{\mu_h^-}{\eps}\frac{\xi_{j+1}-\xi_j}{2}}e^{-\frac{\mu_h^-}{\eps}(\xi_h-\xi_{h-1})}e^{-\frac{\mu_h^-}{\eps}(n-1)},\;n>0.\]
It follows
\[\vert\hat{\tau}_h^-\vert\leq Ce^{-\frac{\mu_h^-}{\eps}\frac{\xi_{j+1}-\xi_j}{2}}\mathscr{E}_h^-,\;x\leq\hat{\xi}_j,\;h\leq j+1.\]

In the same way, since
\[\begin{split}
&h<j-1,\;\;n\leq 0,\;\text{ and }\;x\geq\hat{\xi}_{j-1},\;\Rightarrow\\
&x-n-\xi_h\geq\frac{1}{2}(\xi_j-\xi_{j-1})+\xi_{j-1}-\xi_h+\vert n\vert\geq\frac{1}{2}(\xi_j-\xi_{j-1})+\xi_{h+1}-\xi_h+\vert n\vert,\\
\\
&h\geq j-1,\;\;n<0,\;\text{ and }\;x\geq\hat{\xi}_{j-1},\;\Rightarrow\\
&x-n-\xi_h\geq\frac{1}{2}(\xi_j-\xi_{j-1})+\xi_{j-1}-\xi_h+\vert n\vert\geq\frac{1}{2}(\xi_j-\xi_{j-1})+1+\xi_{j-1}-\xi_h+(\vert n\vert-1)\\
&\geq\frac{1}{2}(\xi_j-\xi_{j-1})+\xi_{h+1}-\xi_h+(\vert n\vert-1),
\end{split}\]
we
 find
 \[\begin{split}
&\vert \tilde{\tau}_h^+\vert\leq Ce^{-\frac{\mu_{h+1}^+}{\eps}\frac{\xi_j-\xi_{j-1}}{2}}\mathscr{E}_{h+1}^+,\;x\leq\hat{\xi}_j,\;h<j-1\\
&\vert \hat{\tau}_h^+\vert\leq Ce^{-\frac{\mu_{h+1}^+}{\eps}\frac{\xi_j-\xi_{j-1}}{2}}\mathscr{E}_{h+1}^+,\;x\leq\hat{\xi}_j,\;h\geq j-1.
 \end{split}\]

and we conclude that
\begin{equation}
\begin{split}
&\vert\tau_j\vert\leq C\sum_h(e^{-\frac{\mu_{h}^-}{\eps}\frac{\xi_{j+1}-\xi_{j}}{2}}\mathscr{E}_{h}^-
+e^{-\frac{\mu_{h+1}^+}{\eps}\frac{\xi_j-\xi_{j-1}}{2}}\mathscr{E}_{h+1}^+)\\
&\leq Ce^{-\frac{\alpha\rho}{\epsilon}}\max_{h,\pm}\mathscr{E}_h^\pm.
\label{tau-est1}
\end{split}
\end{equation}
This concludes the estimate of $\tau_j$.

3.
 Note that \eqref{generic} and
\[\begin{split}
&n>0\;\;\text{and}\;\;x\in[0,1]\;\;\Rightarrow \\
&x-\xi_j-n\leq -\xi_j-(n-1)\leq-(\xi_j-\xi_{j-1})-(n-1),\\
&n<0\;\;\text{and}\;\;x\in[0,1]\;\;\Rightarrow\\
& x-\xi_j-n\geq -\xi_j+1+(\vert n\vert-1)\geq (\xi_{j+1}-\xi_j)+(\vert n\vert-1)
\end{split}\]
imply
\begin{equation}
\begin{split}
&\sum_{n>0}\vert\bar{u}_j^\prime(\frac{x-\xi_j-n}{\eps})\vert\leq Ce^{-\frac{\mu_j^-}{\eps}(\xi_j-\xi_{j-1})},\;\;x\in[0,1],\\
&\sum_{n<0}\vert\bar{u}_j^\prime(\frac{x-\xi_j-n}{\eps})\vert\leq Ce^{-\frac{\mu_{j+1}^+}{\eps}(\xi_{j+1}-\xi_j)}
,\;\;x\in[0,1].
\end{split}
\label{sigma-}
\end{equation}
It follows
\begin{equation}
\vert\varkappa_j\vert\leq C(\mathscr{E}_j^-+\mathscr{E}_{j+1}^+)
,\;\;x\in[0,1].
\label{sigma}
\end{equation}

From the expression \eqref{tau1} of $\tau_j$ we see that $\epsilon\tau_{jx}$, $\epsilon\tau_{j\xi_j}$, $\epsilon^2\tau_{jxx}$ and $\epsilon^3\tau_{jxx\xi_j}$ have a structure similar to $\tau_j$ and therefore,  \eqref{generic} implies that can be estimated by the same arguments developed for estimating $\tau_j$. The same remark applies to $\varkappa_j$.

4. From \eqref{uxixij} we have
\[
\|\epsilon u_{\xi_j}^\xi\|=\|\bar{u}_j^\prime(\frac{\cdot-\xi_j}{\eps})\|+\mathrm{O}(\|\varkappa_j\|),\]
and recalling that $\bar{q}_j^2=\int_\R\vert\bar{u}_j^\prime\vert^2ds$
\begin{equation}
\begin{split}
&\vert\|\bar{u}_j^\prime(\frac{\cdot-\xi_j}{\eps})\|^2-\epsilon\bar{q}_j^2\vert
\leq\epsilon\int_{-\infty}^{-\frac{\xi_j}{\epsilon}}\vert\bar{u}_j^\prime(s)\vert^2ds
+\epsilon\int_{\frac{1-\xi_j}{\epsilon}}^{+\infty}\vert\bar{u}_j^\prime(s)\vert^2ds\\
&\leq C\epsilon({\mathscr{E}_j^-}^2+{\mathscr{E}_{j+1}^+}^2),
\end{split}
\label{uxij2-es}
\end{equation}
where we have also used \eqref{generic} and the assumption $[\xi_{j-1},\xi_{j+1}]\subset[0,1]$.
This and the estimate for $\varkappa_j$ imply \eqref{uxixi}$_1$. Similarly we have
\[\begin{split}
&\|\epsilon^2 u_{\xi_j\xi_j}^\xi\|\leq\|\bar{u}_j^{\prime\prime}(\frac{\cdot-\xi_j}{\eps})\|+\|\epsilon^2\varkappa_{j\xi_j\xi_j}\|,\\
&\|\bar{u}_j^{\prime\prime}(\frac{\cdot-\xi_j}{\eps})\|^2\leq\epsilon\int_\R\vert\bar{u}_j^{\prime\prime}(s)\vert^2ds,
\end{split}\]
that together with $u_{\xi_i\xi_j}^\xi=0$ for $i\neq j$ yield \eqref{uxixi}$_2$. The proof of \eqref{uxixi}$_3$ follows from the previous estimates for $u_{\xi_j}^\xi, u_{\xi_j\xi_j}^\xi$ and $u_x^\xi=-\sum_ju_{\xi_j}^\xi$ and $u_{xx}^\xi=\sum_ju_{\xi_j\xi_j}^\xi$.

5. Set
\begin{equation}
\label{wij}
p_{ij}(x)=\frac{\vert u_{\xi_i}^\xi(x)\cdot u_{\xi_j}^\xi(x)\vert}{\|u_{\xi_i}^\xi\|\|u_{\xi_j}^\xi\|},\;\;i\neq j.
\end{equation}
To estimate $p_{ij}$ we can assume $\xi_i<\xi_j$.
Observe that $u_{\xi_i}^\xi$ and $u_{\xi_j}^\xi$ are periodic maps which, for small $\epsilon>0$, concentrate at $\xi_i+\Z$ and $\xi_j+\Z$ respectively. Therefore we consider the two intervals $[\xi_i,\xi_j]$ and
$[\xi_j,\xi_i+1]$. Set
\[\bar{p}_{ij}(x)=\frac{\vert\bar{u}_i^\prime(\frac{x-\xi_i}{\epsilon})\cdot\bar{u}_j^\prime(\frac{x-\xi_j}{\epsilon})\vert}
{\|\epsilon u_{\xi_i}^\xi\|\|\epsilon u_{\xi_j}^\xi\|},\]
then we have
\begin{equation}
p_{ij}(x)\leq\bar{p}_{ij}(x)+C\frac{\vert\varkappa_i\vert+\vert\varkappa_j\vert}{\|\epsilon u_{\xi_i}^\xi\|\|\epsilon u_{\xi_j}^\xi\|}.
\label{w-ij}
\end{equation}
From \eqref{generic} we have
\begin{equation}
\begin{split}
&\bar{p}_{ij}(x)\leq\frac{C}{\epsilon}e^{-\frac{\mu_{i+1}^+}{\eps}(x-\xi_i)}e^{-\frac{\mu_{j}^-}{\eps}(\xi_j-x)}
\leq\frac{C}{\epsilon}(e^{-\frac{\mu_{j}^-}{\eps}(\xi_j-\xi_i)}+e^{-\frac{\mu_{i+1}^+}{\eps}(\xi_j-\xi_i)})\\
&\leq\frac{C}{\epsilon}(\mathscr{E}_j^-+\mathscr{E}_{i+1}^+),\;\;x\in[\xi_i,\xi_j]
\end{split}
\label{h<k}
\end{equation}
where we have used the fact that a linear function assumes the maximum at the boundary of the interval of definition and also used $\xi_j\geq\xi_{i+1}$. In the same way one finds that in the interval $[\xi_j,\xi_i+1]$ it results
\[\bar{p}_{ij}\leq\frac{C}{\epsilon}(\mathscr{E}_i^-+\mathscr{E}_{j+1}^+).\] This, \eqref{h<k}, \eqref{w-ij}, \eqref{uxixi}$_1$ and the bound for $\varkappa_j$ conclude the proof of \eqref{vij}.

6. From $u_x^\xi=-\sum_ju_{\xi_j}^\xi$ and \eqref{vij} we have
\[u_x^\xi\cdot u_x^\xi=\sum_ju_{\xi_j}^\xi\cdot u_{\xi_j}^\xi
+\mathrm{O}(\frac{1}{\epsilon}\max_{j,\pm}\mathscr{E}_j^\pm).\]
From this \eqref{uxixij} and \eqref{uxij2-es} it follows
\[\begin{split}
&\epsilon^2\|u_x^\xi\|^2=\epsilon^2\sum_j\|u_{\xi_j}^\xi\|^2
+\mathrm{O}(\epsilon\max_{j,\pm}\mathscr{E}_j^\pm)\\
&=\sum_j\|\bar{u}_j^\prime(\frac{\cdot-\xi_j}{\epsilon})\|^2
+\mathrm{O}(\epsilon\max_{j,\pm}\mathscr{E}_j^\pm)
=\epsilon\sum_j\bar{q}_j^2+\mathrm{O}(\epsilon\max_{j,\pm}\mathscr{E}_j^\pm).
\end{split}\]
We also observe that from \eqref{uxi1} and Lemma \ref{est-uxi} it follows
\[\begin{split}
&x\in[\hat{\xi}_{j-1},\hat{\xi}_j]\;\Rightarrow\\
&W(u^\xi)=W(\bar{u}_j(\frac{x-\xi_j}{\epsilon}))
+\mathrm{O}(e^{-\frac{\mu_j^+}{\epsilon}(x-\xi_{j-1})}+e^{\frac{\mu_{j+1}^-}{\epsilon}(x-\xi_{j+1})}
+e^{-\frac{\alpha\rho}{\epsilon}}\max_{h,\pm}\mathscr{E}_h^\pm),
\end{split}\]
which yields
\[\begin{split}
&\int_{\hat{\xi}_{j-1}}^{\hat{\xi}_j}W(u^\xi)dx=
\epsilon\int_{-\frac{(\xi_j-\xi_{j-1})}{2\epsilon}}^{\frac{(\xi_{j+1}-\xi_j)}{2\epsilon}}W(\bar{u}_j(s))ds
+\mathrm{O}(\epsilon\max_{h,\pm}{\mathscr{E}_h^\pm}^\frac{1}{2})\\
&=\epsilon\int_\R W(\bar{u}_j(s))ds+\mathrm{O}(\epsilon\max_{h,\pm}{\mathscr{E}_h^\pm}^\frac{1}{2})
=\frac{\epsilon}{2}\bar{q}_j^2+\mathrm{O}(\epsilon\max_{h,\pm}{\mathscr{E}_h^\pm}^\frac{1}{2}),
\end{split}\]
where we have also used the fact that the heteroclinic solution $\bar{u}_j$ satisfies the equipartition of energy.
This completes the proof of \eqref{Juxi-est} and concludes the proof of the lemma.
\end{proof}
\begin{proof} (Lemma \ref{gammasquare0})
We begin by deriving an $L^\infty$ bound for the projection $\gamma_j^\xi$ of $\epsilon u_{\xi_j}^\xi$ on $X^{\xi,\top}$.

\begin{lemma}\label{eta}
Let $\gamma_j^\xi\in X^{\xi,\perp}$ be defined by the condition that $\epsilon u_{\xi_j}^\xi+\gamma_j^\xi\in X^\xi$. Then, provided $\epsilon>0$ is sufficiently small, it results
\[\begin{split}
&\|\gamma_j^\xi\|_{W_\epsilon^{1,2}}\leq C\max_{h,\pm}\mathscr{E}_h^\pm,\\
&\|\gamma_j^\xi\|_{L^\infty}\leq\frac{C}{\epsilon^\frac{1}{2}}\max_{h,\pm}\mathscr{E}_h^\pm.
\end{split}\]

\end{lemma}
\begin{proof}
In this proof we simply write $\gamma$ instead of $\gamma_j^\xi$.
The condition $\epsilon u_{\xi_j}^\xi+\gamma\in X^\xi$
can be expressed in the form
\[\epsilon u_{\xi_j}^\xi+\gamma=\epsilon\sum_{h=1}^N\langle u_{\xi_j}^\xi,\varphi_h^\xi\rangle\varphi_h^\xi,\]
where as before $\varphi_h^\xi$, $h=1,\ldots,N$ is an orthonormal basis for $X^\xi$.
Since $X^\xi$ is invariant under ${L}^\xi$ and $\gamma\in X^{\xi,\perp}$ we have

\[{L}^\xi(\epsilon u_{\xi_j}^\xi+\gamma)=\epsilon\sum_{h=1}^N\langle u_{\xi_j}^\xi,\varphi_h^\xi\rangle{L}^\xi\varphi_h^\xi.\]

Forming the inner product of this equation with $\gamma$ and using
\[\begin{split}
&\langle{L}^\xi\varphi_h^\xi,\gamma\rangle=0,\\
&{L}^\xi u_{\xi_j}^\xi=-\mathscr{F}_{\xi_j}(u^\xi),
\end{split}\]
on the basis of \eqref{Lbounds} in Proposition \ref{SpectrumL}, we get

\begin{equation}
\begin{split}
&\mu^*\|\gamma\|_{W_\epsilon^{1,2}}^2\leq\langle{L}^\xi\gamma,\gamma\rangle=\epsilon\langle \mathscr{F}_{\xi_j}(u^\xi),\gamma\rangle\\
&\leq\epsilon\|\mathscr{F}_{\xi_j}(u^\xi)\|\|\gamma\|
\leq\epsilon\|\mathscr{F}_{\xi_j}(u^\xi)\|\|\gamma\|_{W_\epsilon^{1,2}},
\end{split}
\label{gammagamma-eq}
\end{equation}
and therefore
\[\begin{split}
&\|\gamma\|_{W_\epsilon^{1,2}}\leq\frac{\epsilon}{\mu^*}\|\mathscr{F}_{\xi_j}(u^\xi)\|,\\
&\|\gamma\|_{L^\infty}\leq\frac{\epsilon^\frac{1}{2}}{\mu^*}\|\mathscr{F}_{\xi_j}(u^\xi)\|.
\end{split}\]
This and Lemma \ref{est-Fuxi} concludes the proof.
\end{proof}

We continue with the proof of Lemma \ref{gammasquare0}. In the remaining part of the proof we drop
 the superscript $\xi$. Define
\[v_j=\frac{u_{\xi_j}^\xi}{\|u_{\xi_j}^\xi\|},\quad v_{ij}=\vert\langle v_i,v_j\rangle\vert,\;i\neq j,\quad w_j=\frac{\gamma_j^\xi}{\|\epsilon u_{\xi_j}^\xi\|},\;\;j=1,\ldots,N,\]
where $\gamma_j^\xi$ is defined in Lemma \ref{eta}.
Note that Lemma \ref{eta}, \eqref{uxixi}$_1$ and \eqref{vij} imply
\begin{equation}
\begin{split}
&\|w_j\|\leq\frac{C}{\epsilon^\frac{1}{2}}\max_{h,\pm}\mathscr{E}_h^\pm,\\
&v_{ij}\leq\frac{C}{\epsilon}\max_{h,\pm}\mathscr{E}_h^\pm.
\end{split}
\label{w-v-est}
\end{equation}
The proof is based on Gram-Schmidt orthonormalization process applied to the vectors $v_j+w_j$ that on the basis of Lemma \ref{eta} belong to $X^\xi$. We have
\[\begin{split}
&\varphi_1=\frac{v_1+w_1}{\|v_1+w_1\|}=v_1+\hat{w}_1,\\
&\hat{w}_1=v_1\frac{1-\|v_1+w_1\|}{\|v_1+w_1\|}+\frac{w_1}{\|v_1+w_1\|},
\end{split}\]
and therefore

\[\begin{split}
&\|\hat{w}_1\|\leq\frac{2\|w_1\|}{\|v_1+w_1\|}\leq 4\|w_1\|,\\
&
\end{split}\]
where we have used \eqref{w-v-est} and $\|v_1\|=1$ that imply $\|v_1+w_1\|\geq\frac{1}{2}$.
Let $\psi_n$, $\tilde{w}_n$ and $\hat{w}_n$ be defined by
\[\begin{split}
&\psi_n=v_n +w_n-\sum_{k=1}^{n-1}\langle v_n+w_n,\varphi_k\rangle\varphi_k=v_n+\tilde{w}_n,\\
&\varphi_n=\frac{v_n+\tilde{w}_n}{\|v_n+\tilde{w}_n\|}=v_n+\hat{w}_n.
\end{split}\]
Then there is $C_n>0$ such that, provided $\epsilon>0$ is sufficiently small
\begin{equation}
\begin{split}
&\|\hat{w}_n\|\leq 4\|\tilde{w}_n\|,\\
&\|\tilde{w}_n\|\leq C_n\Big(\sum_{k=1}^n\|w_k\|+\sum_{k=2}^n\sum_{h=1}^{k-1}v_{hk}\Big).
\end{split}
\label{iter}
\end{equation}

 It suffices to prove the second inequality. Indeed, if \eqref{iter}$_2$ holds and $\epsilon>0$ is small, \eqref{iter}$_1$ follows from \eqref{w-v-est} and the same argument that shows $\|\hat{w}_1\|\leq 4\|{w}_1\|$. We prove the second inequality by induction. For $n=2$, observing also that $\tilde{w}_1=w_1$, we have
\[\psi_2=v_2 +w_2-\langle v_2+w_2,\stackrel{\varphi_1}{v_1+\hat{w}_1}\rangle\varphi_1.\]
It follows, using also that $\langle w_2,\varphi_1\rangle=0$,
\[\begin{split}
&\tilde{w}_2=w_2-\langle v_2,v_1\rangle\varphi_1-\langle v_2,\hat{w}_1\rangle\varphi_1
,\\
&\|\tilde{w}_2\|\leq\|{w}_2\|+\|\hat{w}_1\| +v_{12}\leq \|{w}_2\|+C\|{w}_1\|+v_{12},
\end{split}\]
which shows that \eqref{iter}$_2$ holds for $n=2$ with a suitable choice of $C_2$. Now we show that if \eqref{iter}$_2$ holds for $2,\ldots,n$ then it also holds for $n+1$. From
\[\begin{split}
&\psi_{n+1}=v_{n+1} +w_{n+1}-\sum_{k=1}^{n}\langle v_{n+1}+w_{n+1},\stackrel{\varphi_k}{v_k+\hat{w}_k}\rangle\varphi_k\\
&=v_{n+1}+w_{n+1}-\sum_{k=1}^{n}(\langle w_{n+1},\varphi_k\rangle+\langle v_{n+1},v_k\rangle
+\langle v_{n+1},\hat{w}_k\rangle)\varphi_k
\end{split}\]
we obtain, recalling also that $\langle w_{n+1},\varphi_k\rangle=0$, $k=1,\ldots,n$,
\begin{equation}
\begin{split}
&\|\tilde{w}_{n+1}\|\leq\|{w}_{n+1}\|+\sum_{k=1}^nv_{k,n+1}+\sum_{k=1}^n\|\hat{w}_k\|\\
&\leq\|{w}_{n+1}\|+\sum_{k=1}^nv_{k,n+1}+4\sum_{k=1}^n\|\tilde{w}_k\|\\
&\leq\|{w}_{n+1}\|+\sum_{k=1}^nv_{k,n+1}+4\sum_{k=1}^n
C_k\Big(\sum_{j=1}^k\|w_j\|+\sum_{i=2}^k\sum_{h=1}^{i-1}v_{hi}\Big),
\end{split}
\label{n+1}
\end{equation}

where, for the last inequality, we have used \eqref{iter}$_2$ for $2,\ldots,n$.
The inequality \eqref{n+1} establishes \eqref{iter}$_2$ for $n+1$ for a suitable choice of $C_{n+1}$.
With the identification $\eta_j^\xi=\hat{w}_j$, the definition of $v_j$ and $\hat{w}_j$ imply that
$\varphi_j^\xi=\frac{u_{\xi_j}^\xi}{\|u_{\xi_j}^\xi\|}+\eta_j^\xi$, $j=1,\ldots,N$ is an orthonornal
basis for the subspace $X^\xi$. The estimate for $\eta_j^\xi$ follows from \eqref{iter} and \eqref{w-v-est}.
The proof is complete.
\end{proof}


\section{Layers Dynamics}
 We now focus on the dynamics of the parabolic equations
\begin{equation}
u_t=\epsilon^2u_{xx}-W_u(u),\;\;x\in(0,1),\;u\in W^{1,2},
\label{parabolic}
\end{equation}
on the set of $1$-periodic maps. As we have indicated in the introduction we expect that $c(\xi)$ determines the  dynamics of \eqref{parabolic}
in a neighborhood of the manifold $\mathcal{M}=\{{u}^\xi:  \xi\in\Xi_\rho\}$ and therefore layers dynamics
since maps in a $W^{1,2}$ neighborhood of $\mathcal{M}$ have a layered structure.
 The scope of this section is to prove that this is indeed the case. We begin by defining a smooth change of coordinate in a neighborhood of $\mathcal{M}$.  As before we denote $L^2$ and $W^{1,2}$ the corresponding subsets of  $1$-periodic maps.
For $\eta>0$, $\rho>0$ small set
\begin{equation}
\mathscr{U}_\rho^{\eta}=\{u\in W^{1,2}: \inf_{\zeta\in\Xi_\rho}\|u-u^\zeta\|_{W_\epsilon^{1,2}}<\eta\}.
\label{Neta}
\end{equation}
\begin{lemma}
\label{Projct}
Set $\eta=a\epsilon^\frac{1}{2}$. Then, if $a>0$ is sufficiently small, there is a constant $c>0$ such that, for each $u\in\mathscr{U}_\rho^{a\epsilon^\frac{1}{2}}$ there is a unique $\xi=\xi(u)\in\Xi_{\rho-c\epsilon}$ that satisfies
\begin{equation}
\begin{split}
&\langle u-{u}^\xi,{u}_{\xi_i}^\xi\rangle=0,\;\;i=1,\ldots,N,\\
&\|u-u^\xi\|\leq\inf_{\zeta\in\Xi_\rho}\|u-u^\zeta\|_{W_\epsilon^{1,2}}.
\end{split}
\label{ortho}
\end{equation}
Moreover
\begin{equation}
\|u-u^\xi\|_{W_\epsilon^{1,2}}\leq\bar{C}\inf_{\zeta\in\Xi_\rho}\|u-u^\zeta\|_{W_\epsilon^{1,2}},
\label{the-projec}
\end{equation}
for some constant $\bar{C}>0$, independent of $\epsilon>0$.
\end{lemma}
\begin{proof}
We need the following result that we prove in Section \ref{Proflemma},
\begin{lemma}
\label{lemmaL2}
There are positive constants $\omega_0, \omega_1$, $C_\pm^0, C_\pm^1$ and positive maps  $m_0:(0,\omega_0]\rightarrow\R$, $m_1:(0,\omega_1]\rightarrow\R$  such that, for $\xi,\zeta\in\Xi_\rho$, it results
\begin{equation}
\begin{split}
&\vert\zeta-\xi\vert\geq\omega\epsilon\Rightarrow \|u^\zeta-u^{\xi}\|_{L^\infty}\geq m_1(\omega),\;\;\omega\in(0,\omega_1],\\
&\vert\zeta-\xi\vert\leq\omega_1\epsilon\Rightarrow \frac{C_-^1}{\epsilon}\vert\zeta-\xi\vert\leq\|u^\zeta-u^{\xi}\|_{L^\infty}
\leq\frac{C_+^1}{\epsilon}\vert\zeta-\xi\vert,
\end{split}
\label{Linftycase}
\end{equation}
and
\begin{equation}
\begin{split}
&\vert\zeta-\xi\vert\geq\omega\epsilon\Rightarrow \|u^\zeta-u^{\xi}\|\geq m_0(\omega)\epsilon^\frac{1}{2}
,\;\;\omega\in(0,\omega_0],\\
&\vert\zeta-\xi\vert\leq\omega_0\epsilon\Rightarrow \frac{C_-^0}{\epsilon^\frac{1}{2}}\vert\zeta-\xi\vert\leq\|u^\zeta-u^{\xi}\|
\leq\frac{C_+^0}{\epsilon^\frac{1}{2}}\vert\zeta-\xi\vert.
\end{split}
\label{L2case}
\end{equation}
Moreover
\begin{equation}
\vert\zeta-\xi\vert\leq\omega_0\epsilon\Rightarrow \|u_x^\zeta-u_x^{\xi}\|
\leq\frac{C_+^0}{\epsilon^\frac{3}{2}}\vert\zeta-\xi\vert.\quad\quad
\label{uxix-est}
\end{equation}

\end{lemma}
1. From the definition of $\mathscr{U}_\rho^{a\epsilon^\frac{1}{2}}$ we can associate to each $u\in\mathscr{U}_\rho^{a\epsilon^\frac{1}{2}}$ a $\bar{\xi}\in\Xi_\rho$ such that  $\|u-u^{\bar{\xi}}\|<a\epsilon^\frac{1}{2}$. Let $\omega\in(0,\omega_0]$ a number to be chosen later. For $a\in(0,\frac{1}{2}m_0(\frac{\omega}{2})]$, $m_0$ as in \eqref{L2case}, the problem
\begin{equation}
\|u-u^\xi\|=\min_{\zeta\in\bar{B}_{\omega\epsilon}(\bar{\xi})}\|u-u^\zeta\|<a\epsilon^\frac{1}{2},
\label{local-min}
\end{equation}
has a solution $\xi\in B_{\frac{\omega}{2}\epsilon}(\bar{\xi})$.

Existence follows from the compactness of $\bar{B}_{\omega\epsilon}(\bar{\xi})$ and the continuity of the map $\zeta\rightarrow\|u-u^\zeta\|$. Lemma \ref{lemmaL2}, $a\in(0,\frac{1}{2}m_0(\frac{\omega}{2})]$, and
 $\vert\zeta-\bar{\xi}\vert\geq\frac{\omega}{2}\epsilon$ imply
\begin{equation}
 \|u-u^\zeta\|\geq\|u^{\bar{\xi}}-u^\zeta\|-\|u-u^{\bar{\xi}}\|\geq(m_0(\frac{\omega}{2})-a)\epsilon^\frac{1}{2}\geq a\epsilon^\frac{1}{2}.
\label{out}
\end{equation}
This implies $\xi\in B_{\frac{\omega}{2}\epsilon}(\bar{\xi})$ and concludes the proof of 1.

Note that while $\bar{\xi}\in\Xi_\rho$ it may occur that $B_{\frac{\omega}{2}\epsilon}(\bar{\xi})\setminus\Xi_\rho\neq\emptyset$. Therefore it may happen that  $\xi\not\in\Xi_\rho$ but, since $\vert\xi-\bar{\xi}\vert<\frac{\omega}{2}\epsilon$, there is a constant $c>0$ such that $\xi\in\Xi_{\rho-c\epsilon}$.

2. The minimizer $\xi$ determined in 1. satisfies \eqref{ortho}. 

The fact that $\xi$ satisfies \eqref{ortho}$_2$ follows from \eqref{local-min} and \eqref{out} that imply
\[\|u-u^\xi\|=\min_{\zeta\in\Xi_{\rho}}\|u-u^\zeta\|.\]
 On the other hand, since $\xi$ belong to the interior of $B_{\omega\epsilon}(\bar{\xi})$,
\eqref{ortho}$_1$ follows from the minimality of $\xi$ and a standard argument.

3. The solution $\xi$ of \eqref{ortho} is unique, independent from the auxiliary point $\bar{\xi}$ and is Lipschitz  continuous on $u\in L^2$. 

Given $u,v\in\mathscr{U}_\rho^{a\epsilon^\frac{1}{2}}$ let $\xi$ and $\eta$ be the corresponding solutions of \eqref{ortho} claimed in 2. We have
\begin{equation}
\begin{split}
&0=(\langle u-u^\xi,u_{\xi_j}^\xi\rangle)-(\langle v-u^\eta,u_{\xi_j}^\eta\rangle)\\
&=(\langle u^\eta-u^\xi,u_{\xi_j}^\xi\rangle)+(\langle u-v,u_{\xi_j}^\xi\rangle)+(\langle
 v-u^\eta,u_{\xi_j}^\xi-u_{\xi_j}^\eta\rangle),
 \end{split}
 \label{0+0}
 \end{equation}
where $(w_i)\subset\R^N$ denotes the vector with components $w_1,\ldots,w_N$.

Assume that $\vert\xi-\eta\vert\leq\omega\epsilon$. Then from \eqref{uzeta-uxi}, \eqref{ui.ujprime}, \eqref{D2uh} (cfr. the proof of Lemma \ref{lemmaL2} and \eqref{uxixij} we obtain
\[\vert\langle u^\zeta-u^\xi,u_{\xi_j}^\xi\rangle-\frac{\bar{q}_j^2}{\epsilon}(\eta_j-\xi_j)\vert\leq \frac{C}{\epsilon^2}\vert\zeta-\xi\vert^2+\mathrm{O}(e^{-\frac{\alpha\rho}{\epsilon}})\vert\zeta-\xi\vert,\]
and therefore, if we take $\omega=\frac{\bar{q}_m^2}{2C}$, $\bar{q}_m=\min_h\bar{q}_h$, we have

\[\begin{split}
&\vert(\langle u^\eta-u^\xi,u_{\xi_j}^\xi\rangle)\vert\geq\frac{\bar{q}_m^2}{\epsilon}\vert\eta-\xi\vert
-\frac{C}{\epsilon^2}\vert\eta-\xi\vert^2\\
&\geq
\frac{\bar{q}_m^2-C\omega}{\epsilon}\vert\eta-\xi\vert\geq \frac{\bar{q}_m^2}{2\epsilon}\vert\eta-\xi\vert,
\end{split}\]

A similar computation yields
\[
\|u_{\xi_j}^\xi-u_{\xi_j}^\eta\|\leq\frac{C}{\epsilon^\frac{3}{2}}\vert\eta-\xi\vert,
\]
and we obtain
\[\vert(\langle
 v-u^\eta,u_{\xi_j}^\xi-u_{\xi_j}^\eta\rangle)\vert\leq \frac{Ca}{\epsilon}\vert\eta-\xi\vert.\]
 We also have
\[\vert(\langle
u- v,u_{\xi_j}^\xi\rangle)\vert\leq\frac{C}{\epsilon^\frac{1}{2}}\|u-v\| \]
\noindent
From these estimates and \eqref{0+0} it follows
\begin{equation}
0\geq\frac{\frac{q_m^2}{2}-Ca}{\epsilon}\vert\eta-\xi\vert-\frac{C}{\epsilon^\frac{1}{2}}\|u-v\|
\geq\frac{q_m^2}{4\epsilon}\vert\eta-\xi\vert-\frac{C}{\epsilon^\frac{1}{2}}\|u-v\|,
\label{no=0}
\end{equation}
where $a\in(0,\frac{q_m^2}{4C}]$.

From \eqref{out} if a minimizer of \eqref{ortho}$_2$ exists it belongs to $B_{\frac{\omega}{2}\epsilon}(\bar{\xi})$. Therefore if $\xi$ and $\eta$ are solutions of \eqref{ortho}  corresponding to the same $u\in\mathscr{U}_\rho^{a\epsilon^\frac{1}{2}}$ we have $\vert\eta-\xi\vert<\omega\epsilon$ from \eqref{out} and we can apply \eqref{no=0} with $u=v$ that yields $\eta=\xi$. This establishes uniqueness. The same argument implies $\xi=\eta$ if $\xi$ and $\eta$ are minimizers of \eqref{local-min} corresponding to two different auxiliary points $\bar{\xi}$ and $\bar{\eta}$. Clearly \eqref{no=0} also yields the Lipschitz continuity of the solution of \eqref{ortho}.

To complete the proof set $\gamma=\inf_{\zeta\in\Xi_\rho}\|u-u^\zeta\|_{W_\epsilon^{1,2}}$. Since we have $\gamma<a\epsilon^\frac{1}{2}$ we can fix a $\bar{\xi}\in\Xi_\rho$ that satisfies
 \[\|u-u^{\bar{\xi}}\|_{W_\epsilon^{1,2}}\leq k\gamma<a\epsilon^\frac{1}{2},\]
 for some $k\in(1,2]$.
  This and $\|u-u^\xi\|\leq\gamma$ that follows from \eqref{ortho}$_2$ imply $\|u^{\bar{\xi}}-u^\xi\|\leq (1+k)\gamma$. From the first part of the proof we also have $\vert\xi-\bar{\xi}\vert<\omega_0\epsilon$ and we can apply \eqref{L2case}$_2$ in Lemma \ref{lemmaL2} that yields $\frac{C_-^0}{\epsilon^\frac{1}{2}}\vert\bar{\xi}-\xi\vert\leq(1+k)\gamma$ and in turn
\[\begin{split}
&\epsilon\|u_x^{\bar{\xi}}-u_x^\xi\| \leq\frac{C_+^0}{\epsilon^\frac{1}{2}}\vert\bar{\xi}-\xi\vert\leq\frac{C_+^0}{C_-^0}(1+k)\gamma,\\
&\Rightarrow\quad\|u^{\bar{\xi}}-u^\xi\|_{W_\epsilon^{1,2}}\leq(1+\frac{C_+^0}{C_-^0})(1+k)\gamma,
\end{split}\]
and, using also $k\leq 2$, we finally obtain
\[\|u-u^\xi\|_{W_\epsilon^{1,2}}\leq\|u-u^{\bar{\xi}}\|_{W_\epsilon^{1,2}}+\|u^{\bar{\xi}}-u^\xi\|_{W_\epsilon^{1,2}}
\leq (2+3(1+\frac{C_+^0}{C_-^0}))\gamma,\]
that prove the last statement of the Lemma. The proof is complete.
\end{proof}
Set
\[\begin{split}
&\mathscr{W}_\rho^{a\epsilon^\frac{1}{2}}=\{(\xi,w):\xi\in\Xi_\rho, w\in W^{1,2},\;\|w\|_{W_\epsilon^{1,2}}<
a\epsilon^\frac{1}{2},\;(\langle w,u_{\xi_j}^\xi\rangle)=0\},\\
&\mathscr{U}(\mathscr{W}_\rho^{a\epsilon^\frac{1}{2}})=\{u=u^\xi+w:(\xi,w)\in\mathscr{W}_\rho^{a\epsilon^\frac{1}{2}}\}.
\end{split}\]

Lemma \ref{Projct} implies that $\mathscr{U}(\mathscr{W}_\rho^{a\epsilon^\frac{1}{2}})\subset\mathscr{U}_\rho^{a\epsilon^\frac{1}{2}}$ and
  $\mathscr{U}_\rho^{a\epsilon^\frac{1}{2}}\subset\mathscr{U}(\mathscr{W}_{\rho-c\epsilon}^{\bar{C}a\epsilon^\frac{1}{2}})$  and that the map
\[\mathscr{W}_\rho^{a\epsilon^\frac{1}{2}}\ni(\xi,w)\rightarrow u=u^\xi+w\in\mathscr{U}(\mathscr{W}_\rho^{a\epsilon^\frac{1}{2}})\]
is a surjection. Note also that $\mathscr{U}(\mathscr{W}_\rho^{a\epsilon^\frac{1}{2}})$ is open in $W_\epsilon^{1,2}$. This follows from the fact that $(\xi_0,w_0)\in\mathscr{W}_\rho^{a\epsilon^\frac{1}{2}}$ implies $\xi_0\in\Xi_{\rho+\delta_0}$ and $\|w\|_{W_\epsilon^{1,2}}<a\epsilon^\frac{1}{2}-\delta_0$ for some $\delta_0>0$ that depends on $(\xi_0,w_0)$.
In conclusion $\mathscr{U}(\mathscr{W}_\rho^{a\epsilon^\frac{1}{2}})$ is a tubular neighborhood of the basic manifold $\mathcal{M}$ with a well defined system of coordinate and for each $u\in\mathscr{U}(\mathscr{W}_\rho^{a\epsilon^\frac{1}{2}})$
there is a  unique pair $\xi\in\Xi_\rho$ and $w\in W^{1,2}$ such that
\begin{equation}
\begin{split}
&u={u}^\xi+w,\\
&\langle w,u_{\xi_j}^\xi\rangle=0,\;\;j=1,\ldots,N.
\end{split}
\label{decomposition}
\end{equation}

\subsection{The proof of Theorems \ref{stay-near} and \ref{dynamic}}

The proof of Theorem \ref{stay-near} is based on a geometric characterization of the
energy functional $W^{1,2}\ni u\rightarrow J_\epsilon({u})=\int_0^1(\frac{\epsilon^2}{2}\vert{u}_x\vert^2+W({u}))dx$ for $u$ near $\mathcal{M}$. By mean of Theorem 2.1 in \cite{BFK} we show that a map $u$ laying in a neighborhood of $\mathcal{M}$ of size $\eta^*=\mathrm{O}(\epsilon^\frac{3}{2})$, if has energy $J_\epsilon(u)$ slightly larger than the energy of the maps $u^\xi\in\mathcal{M}$, ( $J_\epsilon(u)-J_\epsilon(u^\xi)=\mathrm{O}(e^{-\frac{\alpha\rho}{\epsilon}})$), actually lays in a much smaller neighborhood of size
$\eta_*=\mathrm{O}(e^{-\frac{\alpha\rho}{\epsilon}})$.

From Lemma \ref{Projct} for $u\in\mathscr{U}(\mathscr{W}_\rho^{a\epsilon^\frac{1}{2}})$ the change of variable $u\leftrightarrow(\xi,w)$ is well defined and the decomposition $u=u^\xi+w$ entails a decomposition of $J_\epsilon(u^\xi+w)-J_\epsilon(u^\xi)$:

\begin{equation}
J_\epsilon({u}^\xi+w)-J_\epsilon({u}^\xi)=\mathscr{L}^\xi(w)+\mathscr{Q}^\xi(w)+\mathscr{N}^\xi(w),
\label{Ew-dec}
\end{equation}
where
\begin{equation}
\begin{split}
&\mathscr{L}^\xi(w)=\int_0^1(\epsilon^2{u}_x^\xi\cdot w_x+W_u({u}^\xi)w)dx,\\
&\mathscr{Q}^\xi(w)=\frac{1}{2}\int_0^1(\epsilon^2\vert w_x\vert^2+W_{uu}({u}^\xi)w\cdot w)dx,\\
&\mathscr{N}^\xi(w)=\int_0^1(W({u}^\xi+w)-W({u}^\xi)-W_u({u}^\xi)-\frac{1}{2}W_{uu}({u}^\xi)w\cdot w)dx.
\end{split}
\label{Ew-dec1}
\end{equation}

Next we derive estimates that describe the geometry of the functional $u\rightarrow J_\epsilon(u)$ and allow for the application of Theorem 2.1 in \cite{BFK}.

We have
\begin{equation}
\vert\mathscr{L}^\xi(w)\vert
\leq\delta_L\|w\|_{W_\epsilon^{1,2}},\;\;\delta_L=\mathrm{O}( e^{-\frac{\alpha\rho}{\epsilon}}).
\label{deltaL}
\end{equation}
Indeed integrating by parts and using \eqref{scrF} yields
\[\begin{split}
&\vert\mathscr{L}^\xi(w)\vert=\vert\int_0^1\mathscr{F}(u^\xi)\cdot wdx\vert\\
&\leq e^{-\frac{\alpha\rho}{\epsilon}}\max_{h,\pm}{\mathscr{E}_h^\pm}^\frac{1}{2}\|w\|
\leq e^{-\frac{\alpha\rho}{\epsilon}}\|w\|_{W_\epsilon^{1,2}}.
\end{split}\]
Since $w\perp\mathrm{span}(\frac{u_{\xi_1}^\xi}{\|u_{\xi_1}^\xi\|},\ldots,\frac{u_{\xi_N}^\xi}{\|u_{\xi_N}^\xi\|})$  Lemma \ref{w-Lw-LwLw} below implies
\begin{equation}
\mathscr{Q}^\xi(w)=\frac{1}{2}\langle{L}^\xi w,w\rangle\geq\frac{1}{2}{\mu}^*\|w\|_{W_\epsilon^{1,2}}^2,
\label{Qscript}
\end{equation}

To estimate $\mathscr{N}^\xi(w)$ we observe that 
 $u\in\mathscr{U}(\mathscr{W}_\rho^{a\epsilon^\frac{1}{2}})$ implies $\|w\|_{W_\epsilon^{1,2}}\leq a\epsilon^\frac{1}{2}$ and, arguing as in \eqref{linftyproof}, we have
 \begin{equation}
 \|w\|_{L^\infty}\leq C\epsilon^{-\frac{1}{2}}\|w\|_{W_\epsilon^{1,2}}.
 \label{w-Linfty}
 \end{equation}
 It follows $\|w\|_{L^\infty}\leq C$ and therefore $\vert W_{uuu}(u^\xi+w)\vert\leq C$ and we conclude
\begin{equation}
\vert\mathscr{N}^\xi(w)\vert\leq C\vert w\vert^3\leq\frac{C}{\epsilon^{\frac{3}{2}}}\|w\|_{W_\epsilon^{1,2}}^3.
\label{N}
\end{equation}
From \eqref{Juxi-est} we also have
\[J_\epsilon({u}^\xi)-J_\epsilon({u}^{\bar{\xi}})\leq\delta_0,\;\;\xi,\bar{\xi}\in\Xi_\rho,\]
with
\begin{equation}
\delta_0\leq Ce^{-\frac{\alpha\rho}{\eps}}.
\label{delta0}
\end{equation}

We restrict to the class of maps that satisfy
\begin{equation}
\delta_1\geq J_\epsilon(u)-\sup_{\xi\in\Xi_\rho} J_\epsilon({u}^\xi),\;\;\delta_1=\mathrm{O}(e^{-\frac{\alpha\rho}{\eps}}).
\label{int-E-bound}
\end{equation}
Since the energy is nonincreasing along the flow generated by the parabolic equation \eqref{parabolic}, if this condition is satisfies at a certain time $\bar{t}$ then is satisfied for $t\geq\bar{t}$.
\vskip.2cm

The estimates \eqref{deltaL},\eqref{Qscript},\eqref{N},\eqref{delta0}  and the assumption \eqref{int-E-bound} allow us to compute the numbers $\eta^*$ and $\eta_*$ defined in
 Theorem 2.1 in \cite{BFK}. In the case at hand it results
\[\eta^*=\frac{\mu^*}{C\bar{C}}\epsilon^\frac{3}{2}=C\epsilon^\frac{3}{2},\quad \eta_*=\frac{\delta_L}{\mu^*}+\sqrt{\frac{\delta_L}{\mu^*}^2+2\frac{\delta_0+\delta_1}{\mu^*}}
=\mathrm{O}(e^{-\frac{\alpha\rho}{\eps}}),\]
($\bar{C}$ is the constant in \eqref{the-projec}).
Therefore, on the basis of Theorem 2.1 in \cite{BFK}, we can state that, for $\epsilon>0$ sufficiently small,
\begin{equation}
u=u^\xi+w\in \mathscr{U}(\mathscr{W}_\rho^{C\epsilon^\frac{3}{2}}),\;\delta_1=\mathrm{O}(e^{-\frac{\alpha\rho}{\eps}})
\;\;\Rightarrow\;\;\|w\|_{W_\epsilon^{1,2}}< Ce^{-\frac{\alpha\rho}{\eps}}.
\label{w-bound}
\end{equation}
Since \eqref{parabolic} is the $L^2$-gradient system associated to the functional $J_\epsilon(u)$, if the initial datum $u_0=u^{\xi_0}+w_0\in\mathscr{U}(\mathscr{W}_\rho^{C\epsilon^\frac{3}{2}})$ satisfies \eqref{int-E-bound}, the same is true for the the solution $u(t,u_0)=u^{\xi(t)}+w(t)$ of \eqref{parabolic} through $u_0$. It follows that $u(t,u_0)$ creeps along $\mathcal{M}$ and $w(t)$ satisfies \eqref{w-bound} for $t\in[0,T)$ where either $T=+\infty$ or $\xi(T)\in\partial\Xi_\rho$
and a standard computation, see Theorem 2.2 in \cite{BFK}, yields
\[T\geq e^{\frac{\alpha\rho}{\eps}}(d_{L^2}(u_0,\partial\mathcal{M})-C\epsilon^\frac{1}{2})^2\geq e^{\frac{\alpha\rho}{\eps}}d^2,\;\;d=\inf_{\zeta\in\partial\Xi_\rho}\|u^{\xi_0}-u^\zeta\|.\]
This concludes the proof of Theorem \ref{stay-near}.
 In Fig. \ref{fig4} we explain the meaning of the numbers $\eta^*$ and $\eta_*$ and sketch the qualitative structure of a cross section of the graph of $u\rightarrow J_\epsilon(u)$ for $u$ near ${\mathcal{M}}$.
 \begin{figure}
  \begin{center}
\begin{tikzpicture}
\draw [blue, very thick](-3,0).. controls (-2.8,-.1) and(-2.5,3)..(-1.8,5);
\draw [blue, very thick](-3,0).. controls (-3.2,-.1) and(-3.5,3)..(-4.2,5);

\draw [blue, very thick](-1.8,5).. controls (-1,5.5) and(-.8,3)..(-.5,2.5);
\draw [blue, very thick](-4.2,5).. controls (-5,5.5) and(-5.2,3)..(-5.5,2.5);

\draw [red,dotted](-2.95,-1)--(-2.95,2);

\draw [red,dotted](-.95,-1)--(-.95,5);

\node at (-.95,-1.2){$\eta^*$};

\draw [red,dotted](-2.64,-1)--(-2.64,2);

\node at (-2.64,-1.2){$\eta_*$};

\draw (-3.5,-.12)--(-1.5,.35);
\draw (-5,1.4)--(-2.67,1.4);
\draw (-5,.02)--(-2.95,.02);

\draw (-5,.02)--(-2.95,.02);


\draw (-7,-1)--(1,-1);
\node at (1.2,-.8){$w$};

\draw [<-] (-4.5,1.4)--(-4.5,1.1);
\draw [<-](-4.5,.02)--(-4.5,.3);
\node at (-4.5,.7){$\delta_1$};

\draw[ fill] (-2.95,.02) circle [radius=0.05];
\node at (-3,-.2){$\mathcal{M}$};

\draw [](-3,0).. controls (-2.8,-.1) and(-2.52,3)..(-2.4,4);
\draw [](-3,0).. controls (-3.2,-.1) and(-3.48,3)..(-3.6,4);

\node at (-.9,.4){$\mathscr{L}^\xi(w)$};
\node at (-1.2,2.){$\mathscr{Q}^\xi(w)$};

\draw [] (-2.6,2.3)--(-1.7,2.);

\end{tikzpicture}
\end{center}
\caption{The structure of $J_\epsilon(u)$ near $\mathcal{M}$.}
\label{fig4}
\end{figure}
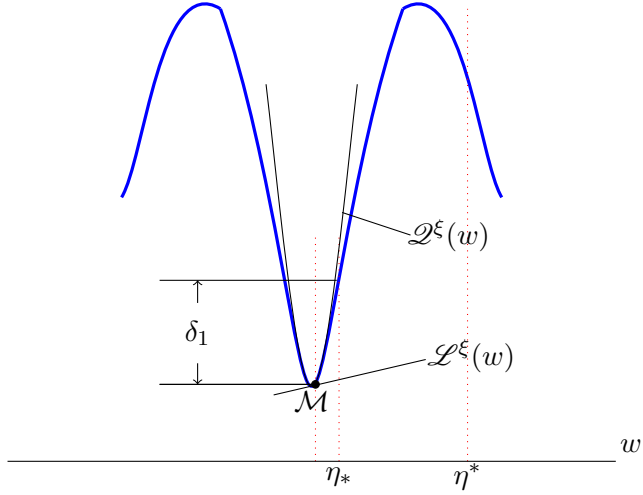

\vskip.3cm

We now continue with the proof of Theorem \ref{dynamic}.
On the basis of Theorem \ref{stay-near} we can assume that $u_0=u^{\xi_0}+w_0$ is chosen so that $u(t,u_0)$ remains in $\mathscr{U}(\mathscr{W}_\rho^{C\epsilon^\frac{3}{2}})$
 and \eqref{w-bound} holds for an exponentially long time.
In particular we have that the change of variables $u\rightarrow (\xi,w)$ is well defined along the motion and \eqref{parabolic} is equivalent to \begin{equation}
 \begin{split}
&w_t+\sum_{h=1}^N{u}_{\xi_h}^\xi\dot{\xi}_h=\mathscr{F}({u}^\xi+w)\\
&=\mathscr{F}({u}^\xi)+\mathscr{F}({u}^\xi+w)-\mathscr{F}({u}^\xi)\\
&=\mathscr{F}({u}^\xi)-{L}^\xi w-N^\xi(w),
\end{split}
\label{par-equiv}
\end{equation}

where
$N^\xi(w)=W_u({u}^\xi+w)-W_u({u}^\xi)-W_{uu}({u}^\xi)w$. Now, using \eqref{par-equiv}, we derive an equation for the evolution of $\langle w,{L}^\xi w\rangle=\int_0^1(\epsilon^2\vert w_x\vert^2+W_{uu}({u}^\xi)w\cdot w)dx$. We have
\begin{equation}
\begin{split}
&\frac{1}{2}\frac{d}{dt}\langle w,{L}^\xi w\rangle
=\langle{L}^\xi w,w_t\rangle+\frac{1}{2}\int_0^1W_{uuu}({u}^\xi)\sum_{h=1}^N{u}_{\xi_h}^\xi\dot{\xi}_hw\cdot wdx\\
&=-\|{L}^\xi w\|^2-\langle{L}^\xi w,N^\xi(w)\rangle+\langle{L}^\xi w,\mathscr{F}(u^\xi)-\sum_{h=1}^N{u}_{\xi_h}^\xi\dot{\xi}_h
\rangle\\
&+\frac{1}{2}\int_0^1W_{uuu}({u}^\xi)\sum_{h=1}^N{u}_{\xi_h}^\xi\dot{\xi}_hw\cdot wdx.
\end{split}
\label{wLw-eq}
\end{equation}

 Differentiating \eqref{decomposition}$_2$ we get the identity
\begin{equation}
\langle w_t,{u}_{\xi_j}^\xi\rangle+ \sum_{h=1}^N\langle w,{u}_{\xi_j\xi_h}^\xi\rangle\dot{\xi}_h=0,\;\;j=1,\ldots,N.
\label{iden-w}
\end{equation}

Inserting the expression of $w_t$ given by \eqref{par-equiv} into \eqref{iden-w} yields
\begin{equation}
\sum_{h=1}^N(\langle{u}_{\xi_h}^\xi,{u}_{\xi_j}^\xi\rangle-\langle w,{u}_{\xi_j\xi_h}^\xi\rangle)\dot{\xi}_h=
\langle\mathscr{F}({u}^\xi)-L^\xi w-N^\xi(w),{u}_{\xi_j}^\xi\rangle,\;\;j=1,\ldots,N,
\label{xidot-impl-1}
\end{equation}

 We rewrite \eqref{xidot-impl-1} in the form
\begin{equation}
\dot{\xi}_j=\frac{1}{\|u_{\xi_j}^\xi\|^2}\langle\mathscr{F}({u}^\xi)-L^\xi w-N^\xi(w),u_{\xi_j}^\xi \rangle
+\frac{\langle w,{u}_{\xi_j\xi_j}^\xi\rangle}{\|u_{\xi_j}^\xi\|^2}\dot{\xi}_j
-\sum_{h\neq j}\frac{\langle u_{\xi_j}^\xi,u_{\xi_h}^\xi\rangle}{\|u_{\xi_j}^\xi\|^2}
\dot{\xi}_h,
\label{diaddotxi}
\end{equation}
where we have used \eqref{uxixij} that implies ${u}_{\xi_j\xi_h}^\xi=0$ for $h\neq j$. From \eqref{w-bound} and \eqref{uxixi} it follows that, for small $\epsilon>0$, this equation can be solved for $\dot{\xi}$. Note that by \eqref{cbar-def} $\langle\mathscr{F}({u}^\xi),\frac{u_{\xi_j}^\xi}{\|u_{\xi_j}^\xi\|}\rangle=\bar{c}_j(\xi)$ and that
$L^\xi u_{\xi_j}^\xi=-\mathscr{F}_{\xi_j}({u}^\xi)$. From these observations and
\begin{equation}
\vert N^\xi(w)\vert=\vert\int_0^1\int_0^1sW_{uuu}({u}^\xi+\tau sw)w\cdot wd\tau ds\vert\leq C\vert w\vert^2
\leq\frac{C}{\epsilon^\frac{1}{2}}\|w\|_{W_\eps^{1,2}}\vert w\vert,
\label{N-est}
\end{equation}
which implies
\[\vert\langle N^\xi(w),\frac{u_{\xi_j}^\xi}{\|u_{\xi_j}^\xi\|}\rangle\vert
\leq\frac{C}{\epsilon^\frac{1}{2}}\|w\|_{W_\eps^{1,2}}\|w\|,\]
it follows
\begin{equation}
\begin{split}
&\frac{1}{\|u_{\xi_j}^\xi\|^2}\langle\mathscr{F}({u}^\xi)-L^\xi w-N^\xi(w),u_{\xi_j}^\xi \rangle\\
&=\frac{\bar{c}_j(\xi)}{\|u_{\xi_j}^\xi\|}+\mathrm{O}(\|w\|_{W_\eps^{1,2}}
(\epsilon\|\mathscr{F}_{\xi_j}({u}^\xi)\|+\|w\|_{W_\eps^{1,2}})),
\end{split}
\label{maindotxi}
\end{equation}
where we have also used \eqref{uxixi}.
Therefore by solving \eqref{diaddotxi} we obtain
\begin{equation}
\begin{split}
&\dot{\xi}_j=\frac{\bar{c}_j(\xi)}{\|u_{\xi_j}^\xi\|}+\mathrm{O}(\|w\|_{W_\eps^{1,2}}
(\epsilon\|\mathscr{F}_{\xi_j}({u}^\xi)\|+\|w\|_{W_\eps^{1,2}}))\\
&+\mathrm{O}(e^{-\frac{\alpha\rho}{\epsilon}}
\sum_h(\vert\bar{c}_h(\xi)\vert+\|w\|_{W_\eps^{1,2}}\|\mathscr{F}_{\xi_h}({u}^\xi)\|)).
\label{xidot}
\end{split}
\end{equation}

From \eqref{xidot}, $\vert\bar{c}_h(\xi)\vert\leq\|\mathscr{F}(u^\xi)\|$, Lemma \ref{est-Fuxi}, and \eqref{w-bound} it follows
 \begin{equation}
 \vert\dot{\xi}\vert\leq C(\epsilon^\frac{1}{2}\|\mathscr{F}(u^\xi)\|+e^{-\frac{\alpha\rho}{\eps}}\|w\|)\leq C e^{-\frac{\alpha\rho}{\eps}}.
 \label{xdot-bound}
 \end{equation}

We now estimate the terms on the right hand side of \eqref{wLw-eq}. Using \eqref{xdot-bound} we have
\begin{equation}
\begin{split}
&\vert\int_0^1W_{uuu}({u}^\xi)\sum_{h=1}^N{u}_{\xi_h}^\xi\dot{\xi}_hw\cdot w)dx\vert\\
&\leq C_\eps e^{-\frac{\alpha\rho}{\eps}}\int_0^1\vert w\vert^2dx=  C_\eps e^{-\frac{\alpha\rho}{\eps}}\| w\|^2.
\end{split}
\label{first}
\end{equation}
From $\|w\|_{L^\infty}\leq\frac{C}{\epsilon^\frac{1}{2}}\|w\|_{W_\epsilon^{1,2}}$, \eqref{N-est} and \eqref{w-bound} we have
$\|N^\xi(w)\|\leq C\|w\|_{W_\epsilon^{1,2}}\|w\|\leq e^{-\frac{\alpha\rho}{\eps}}\|w\|$
which implies

\begin{equation}
\begin{split}
&\vert\langle{L}^\xi w,N^\xi(w)\rangle\vert\leq e^{-\frac{\alpha\rho}{2\eps}}\|{L}^\xi w\|\|w\|
\leq e^{-\frac{\alpha\rho}{2\eps}}(\|{L}^\xi w\|^2+\|w\|^2),\\
&\vert\langle{L}^\xi w,\mathscr{F}(u^\xi)\rangle\vert\leq\|{L}^\xi w\|\|\mathscr{F}(u^\xi)\|\leq\frac{1}{4}\|{L}^\xi w\|^2+\|\mathscr{F}(u^\xi)\|^2.
\end{split}
\label{second}
\end{equation}
 Finally we observe that Lemma \ref{est-Fuxi} and $L^\xi u_{\xi_h}=-\mathscr{F}_{\xi_h}(u^\xi)$ together with \eqref{xdot-bound} imply

\begin{equation}
\begin{split}
&\vert\langle{L}^\xi w,\sum_{h=1}^N{u}_{\xi_h}^\xi\dot{\xi}_h\rangle\vert
\leq\|w\|\sum_h\|\mathscr{F}_{\xi_h}(u^\xi)\|\vert\dot{\xi}_h\vert\\
&\leq C e^{-\frac{\alpha\rho}{\eps}}\|w\|(C\|\mathscr{F}(u^\xi)\|+e^{-\frac{\alpha\rho}{\eps}}\|w\|)
\leq C e^{-\frac{\alpha\rho}{\eps}}(\|\mathscr{F}(u^\xi)\|^2+\|w\|^2).
\end{split}
\label{third}
\end{equation}
 From these estimates and \eqref{wLw-eq} it follows

\begin{equation}
\frac{1}{2}\frac{d}{dt}\langle w,{L}^\xi w\rangle
\leq-\frac{1}{2}\|{L}^\xi w\|^2+ 2\|\mathscr{F}(u^\xi)\|^2+
C e^{-\frac{\alpha\rho}{\eps}}\|w\|^2.
\label{w-eq0}
\end{equation}

\begin{lemma}
\label{w-Lw-LwLw}
There exist positive constants $\nu$ and $D$ independent of $\xi\in\Xi_\rho$ and of $\epsilon>0$ such that
$w\perp\mathrm{span}(\frac{u_{\xi_1}^\xi}{\|u_{\xi_1}^\xi\|},\ldots,\frac{u_{\xi_N}^\xi}{\|u_{\xi_N}^\xi\|})$ implies
\begin{equation}
\begin{split}
&{\mu}^*\|w\|_{W_\eps^{1,2}}^2\leq\langle w,{L}^\xi w\rangle\leq D\|w\|_{W_\eps^{1,2}}^2,\\
&\langle w,{L}^\xi w\rangle\leq\nu\|{L}^\xi w\|^2,\\
&\|w\|\leq\nu\|{L}^\xi w\|.
\end{split}
\label{w-Lw-LwLw1}
\end{equation}
\end{lemma}
\begin{proof}
We have $w=w^{\shortparallel}+w^\perp$ with $w^{\shortparallel}\in X^\xi$ and $w^\perp\in X^{\xi,\perp}$. From \eqref{decomposition} and Lemma \ref{gammasquare0} it follows
\begin{equation}
\begin{split}
&w^{\shortparallel}=\sum_j\langle w,\varphi_j^\xi\rangle\varphi_j^\xi=\sum_j\langle w,\eta_j^\xi\rangle\varphi_j^\xi,\\
&w_x^{\shortparallel}=\sum_j\langle w,\eta_j^\xi\rangle\varphi_{j,x}^\xi.
\end{split}
\label{wshort0}
\end{equation}
This, Lemma \ref{eta} and Lemma \ref{gammasquare0} that imply $\|\varphi_{j,x}^\xi\|=\mathrm{O}(\frac{C}{\epsilon^\frac{1}{2}})$, yields
\begin{equation}
\begin{split}
&\|w^{\shortparallel}\|\leq e^{-\frac{\alpha\rho}{\epsilon}}\|w\|\leq e^{-\frac{\alpha\rho}{\epsilon}}\|w\|_{W_\epsilon^{1,2}},\\
&\|w_x^{\shortparallel}\|\leq e^{-\frac{\alpha\rho}{\epsilon}}\|w\|\leq e^{-\frac{\alpha\rho}{\epsilon}}\|w\|_{W_\epsilon^{1,2}}.
\end{split}
\label{wshort}
\end{equation}
and in turn
\begin{equation}
\begin{split}
& (1-e^{-\frac{\alpha\rho}{\epsilon}})\|w\|\leq\|w^{\perp}\|\leq \|w\|,\\
& (1-e^{-\frac{\alpha\rho}{\epsilon}})\|w\|_{W_\epsilon^{1,2}}\leq\|w^{\perp}\|_{W_\epsilon^{1,2}}\leq
(1+e^{-\frac{\alpha\rho}{\epsilon}})\|w\|_{W_\epsilon^{1,2}}.
\end{split}
\label{wperp}
\end{equation}
Let  $\varphi_j^\xi$ and $\lambda_j^\xi$, $\lambda_j^\xi\leq\lambda_{j+1}^\xi$,  $j=1,\ldots$ the eigenvectors and the corresponding eigenvalues of ${L}^\xi$. Then Proposition \ref{SpectrumL} yields
\[\begin{split}
&(\lambda_{N+1}^\xi)^2\|w^\perp\|^2\leq\|L^\xi w^\perp\|^2\leq\|L^\xi w\|^2,\\
&\|L^\xi w^{\shortparallel}\|^2\leq e^{-\frac{\alpha\rho}{\epsilon}}\|w^{\shortparallel}\|^2
\end{split}\]
This and \eqref{wperp} imply \eqref{w-Lw-LwLw1}$_3$ with $\nu=\sup_{\xi\in\Xi_\rho}\frac{1}{2\lambda_{N+1}^\xi}$. The inequality \eqref{w-Lw-LwLw1}$_2$ follows from \eqref{w-Lw-LwLw1}$_3$. To complete the proof we note that Proposition \ref{SpectrumL} implies
\[\mu\|w^{\perp}\|_{W_\epsilon^{1,2}}^2+\langle w^{\shortparallel},L^\xi w^{\shortparallel}\rangle\leq
\langle w,L^\xi w\rangle.\]

The left inequality in \eqref{w-Lw-LwLw1}$_1$ follows from this \eqref{wperp} and $\vert\langle w^{\shortparallel},L^\xi w^{\shortparallel}\rangle\vert\leq e^{-\frac{\alpha\rho}{\epsilon}}\|w\|_{W_\epsilon^{1,2}}^2$. The right inequality from $\vert W_{uu}({u}^\xi)\vert\leq C$.
\end{proof}
From Lemma \ref{w-Lw-LwLw} and \eqref{w-eq0} it follows
\begin{equation}
\begin{split}
&\frac{1}{2}\frac{d}{dt}\langle w,{L}^\xi w\rangle
\leq-\frac{1}{2}\|{L}^\xi w\|^2+ 2\|\mathscr{F}(u^\xi)\|^2+
C_\eps e^{-\frac{\alpha\rho}{\eps}}\|w\|^2\\
&\leq-\frac{1}{4}\|{L}^\xi w\|^2+ 2\|\mathscr{F}(u^\xi)\|^2
\leq-\frac{1}{4\nu}\langle w,{L}^\xi w\rangle+ 2\|\mathscr{F}(u^\xi)\|^2.
\end{split}
\label{w-Eq}
\end{equation}
Which should be complemented by \eqref{xidot}.

\begin{lemma}
\label{w<c} There exist constants $Q>0$ and $\beta> 1$, independent of $\xi\in\Xi_\rho$ and $\epsilon\in(0,\epsilon_0]$, such that the condition
\[\|w_0\|_{W_\epsilon^{1,2}}\leq Q\|\mathscr{F}(u^{\xi_0})\|,\]
as long as $\xi(t)\in\Xi_\rho$, implies
\begin{equation}
\|w(t)\|_{W_\epsilon^{1,2}}\leq\beta Q\|\mathscr{F}(u^{\xi(t)})\|.
\label{w<Cc}
\end{equation}
\end{lemma}
\begin{proof} Set $Q=4\sqrt{\frac{2\nu}{\mu^*}}$, $\beta=\sqrt{\frac{2D}{\mu^*}+\frac{1}{2}}$ and
 $\bar{t}=2\nu\ln{\frac{4D}{\mu^*}}$ and observe that \eqref{w-Lw-LwLw1} implies $\beta>1$. From this and the condition $\|w_0\|_{W_\eps^{1,2}}\leq Q\|\mathscr{F}(u^{\xi_0})\|$ it follows
 \begin{equation}
 \|w(t)\|_{W_\eps^{1,2}}\leq\beta Q\|\mathscr{F}(u^{\xi(t)})\|,\;\;t\in[0,\min\{t^*,\bar{t}\}],
 \label{w<Cc1}
 \end{equation}
 for some $t^*>0$.This and \eqref{xdot-bound} imply
 \[\vert\dot{\xi}_h\vert\leq C\epsilon^\frac{1}{2}\|\mathscr{F}(u^{\xi(t)})\|,\;\;t\in[0,\min\{t^*,\bar{t}\}],\]
 and therefore, from  \eqref{scrF}$_2$ we obtain
 \[\begin{split}
 &\vert\frac{d}{dt}\|\mathscr{F}(u^{\xi(t)})\|\vert\leq\sum_{h=1}^N\|\mathscr{F}_{\xi_h}(u^{\xi(t)})\|\vert \dot{\xi}_h(t)\vert\\
 &\leq e^{-\frac{\alpha\rho}{\eps}}\|\mathscr{F}(u^{\xi(t)})\|,\;\;t\in[0,\min\{t^*,\bar{t}\}].
 \end{split}\]
 It follows
 \begin{equation}
 (1-e^{-\frac{\alpha\rho}{\eps}}\bar{t})\|\mathscr{F}(u^{\xi_0})\|\leq\|\mathscr{F}(u^{\xi(t)})\|
 \leq(1+2e^{-\frac{\alpha\rho}{\eps}}\bar{t})\|\mathscr{F}(u^{\xi_0})\|,\;\;t\in[0,\min\{t^*,\bar{t}\}].
 \label{c-and-c0}
 \end{equation}
This and \eqref{w-Eq} yield
\[\begin{split}
&\frac{d}{dt}\langle w,{L}^\xi w\rangle+\frac{1}{2\nu}\langle w,{L}^\xi w\rangle\leq 4(1+2e^{-\frac{\alpha\rho}{\eps}}\bar{t})^2\|\mathscr{F}(u^{\xi_0})\|^2
,\;\;t\in[0,\min\{t^*,\bar{t}\}],
\end{split}\]
which implies
\[\langle w,{L}^\xi w\rangle\leq\langle w_0,{L}^\xi w_0\rangle e^{-\frac{t}{2\nu}}
+8\nu(1-e^{-\frac{t}{2\nu}})(1+2e^{-\frac{\alpha\rho}{\eps}}\bar{t})^2\|\mathscr{F}(u^{\xi_0})\|^2,\;\;t\in[0,\min\{t^*,\bar{t}\}].\]
This and Lemma \ref{w-Lw-LwLw} imply

\begin{equation}
\begin{split}
&{\mu^*}\|w(t)\|_{W_\eps^{1,2}}^2\leq\langle w,{L}^\xi w\rangle\leq e^{-\frac{t}{2\nu}}D\|w_0\|_{W_\eps^{1,2}}^2+8\nu(1+2e^{-\frac{\alpha\rho}{\eps}}\bar{t})^2\|\mathscr{F}(u^{\xi_0})\|^2\\
&\Rightarrow\quad\|w(t)\|_{W_\eps^{1,2}}^2\leq\frac{1}{\mu^*} \Big(e^{-\frac{t}{2\nu}}DQ^2+
8\nu(1+2e^{-\frac{\alpha\rho}{\eps}}\bar{t})^2\Big)\|\mathscr{F}(u^{\xi_0})\|^2\\
&\leq\frac{e^{-\frac{t}{2\nu}}DQ^2+
8\nu(1+2e^{-\frac{\alpha\rho}{\eps}}\bar{t})^2}
{\mu^*(1-e^{-\frac{\alpha\rho}{\eps}}\bar{t})^2}\|\mathscr{F}(u^{\xi(t)})\|^2,
\;\;t\in[0,\min\{t^*,\bar{t}\}],
\end{split}
\label{w-Eq2}
\end{equation}
where we have also used \eqref{c-and-c0}.
Assume that $\epsilon>0$ is so small that $\frac{(1+2e^{-\frac{\alpha\rho}{\eps}}\bar{t})^2}{(1-e^{-\frac{\alpha\rho}{\eps}}\bar{t})^2}\leq 2$, then \eqref{w-Eq2} implies
\begin{equation}\|w(t)\|_{W_\eps^{1,2}}^2
\leq\frac{2DQ^2+16\nu}{\mu^*}
\|\mathscr{F}(u^{\xi(t)})\|^2
=\beta^2Q^2\|\mathscr{F}(u^{\xi(t)})\|^2,\;\;t\in[0,\min\{t^*,\bar{t}\}]
\label{betaQ}
\end{equation}
and
\begin{equation}\|w(\bar{t})\|_{W_\eps^{1,2}}^2
\leq(2\frac{e^{-\frac{\bar{t}}{2\nu}}}{\mu^*}DQ^2+\frac{16\nu}{\mu^*})
\|\mathscr{F}(u^{\xi(\bar{t})})\|^2=Q^2\|\mathscr{F}(u^{\xi(\bar{t})})\|^2.
\label{Q}
\end{equation}
From \eqref{betaQ}
It follows that $\min\{t^*,\bar{t}\}=\bar{t}$ and therefore that \eqref{w-Eq2} is valid for $t\in[0,\bar{t}]$.
This and \eqref{Q} imply that we can repeat what we have said above for $t\in[0,\bar{t}]$ for the interval $[\bar{t},2\bar{t}]$ and so on as long as $\xi(t)\in\Xi_\rho$. This concludes the proof.
\end{proof}

We are now in the position of completing the proof of Theorem \ref{dynamic}. Indeed \eqref{dotxi=c} follows by inserting the estimate \eqref{w<Cc} for $w$ into \eqref{xidot}, by observing that Lemma \ref{est-Fuxi} implies
\[\begin{split}
&\|\mathscr{F}(u^{\xi})\|\leq Ce^{-\frac{\alpha\rho}{\eps}}\max_{h,\pm}{\mathscr{E}_h^\pm}\frac{1}{2},\\
&\|\mathscr{F}_{\xi_j}(u^{\xi})\|\leq\frac{C}{\epsilon}\max_{h,\pm}{\mathscr{E}_h^\pm},\;\;j=1,\ldots,N.
\end{split},\]
and by recalling \eqref{uxixi}.
This conclude the proof of the theorem.


\subsection{The proof of Lemma \ref{lemmaL2}.}
\label{Proflemma}

\begin{proof}
1. For $\xi\in\Xi_\rho$, we have
\begin{equation}
\begin{split}
&u^\xi(\xi_j)=\bar{u}_j(0)+\mathrm{O}(e^{-\frac{\alpha\rho}{\epsilon}}),\\
&u^\xi(\hat{\xi}_{j-1})=a_j+\mathrm{O}(e^{-\frac{\alpha\rho}{\epsilon}}),\;\;(\hat{\xi}_h=\frac{\xi_{h+1}+\xi_h}{2}),\\
&u^\xi(\hat{\xi}_j)=a_{j+1}+\mathrm{O}(e^{-\frac{\alpha\rho}{\epsilon}}).
\end{split}
\label{uxi0}
\end{equation}
This follows from \eqref{uxi1},\eqref{generic} and Lemma \ref{est-uxi}.

2. Assume that there is $1\leq j\leq N$ such that $\xi,\zeta\in\Xi_\rho$ satisfy
\begin{equation}
\zeta_h\notin(\hat{\xi}_{j-1},\hat{\xi}_j),\;\;h\neq j,\;h=1,\ldots,N,
\label{not-in}
\end{equation}
then there is $\bar{m}_1>0$ independent of $\epsilon>0$ such that
\begin{equation}
\|u^\zeta-u^{\xi}\|_{L^\infty}\geq \bar{m}_1.
\label{Linfty-diff}
\end{equation}
To prove this we note that \eqref{not-in} implies $u^\zeta(\xi_j)=a_h+\mathrm{O}(e^{-\frac{\alpha\rho}{\epsilon}})$ for some $1\leq h\leq N$ that together with \eqref{uxi0} and
\begin{equation}
\{\bar{u}_i(s): s\in\R\}\cap A=\emptyset,\;\;\;(A=\{W=0\}).
\label{us-aj}
\end{equation}
that follows from the minimality of $\bar{u}_i$, $i=1,\ldots,N$,
imply $u^\zeta(\xi_j)-u^\xi(\xi_j)
=a_h-\bar{u}_j(0)+\mathrm{O}(e^{-\frac{\alpha\rho}{\epsilon}})\neq 0$ where we have assumed $\epsilon>0$ sufficiently small. This implies \eqref{Linfty-diff} with $\bar{m}_1=\frac{1}{2}\vert a_h-\bar{u}_j(0)\vert$.

3. Assume now that $\xi,\zeta$ are such that there is no $j$ that satisfies \eqref{not-in}. This implies that, for each $j$ there is one and exactly one $\zeta_{h_j}\in(\hat{\xi}_{j-1},\hat{\xi}_j)$. There are three different possibilities
\begin{description}
\item[a)]there is $j$ such that $a_{h_j}\neq a_j$ and $a_{{h_j}+1}\neq a_{j+1}$.
\item[b)]there is $j$ such that either $a_{h_j}=a_j$ and $a_{{h_j}+1}\neq a_{j+1}$
 or $a_{h_j}\neq a_j$ and $a_{{h_j}+1}= a_{j+1}$.
\item[c)] $a_{h_j}=a_j$ and $a_{{h_j}+1}= a_{j+1}$ for all $j$.
\end{description}
In case $\mathbf{a)}$, for small $\epsilon>0$, we have
\[\begin{split}
&\zeta_{h_j}\leq\xi_j\;\;\Rightarrow u^\zeta(\xi_j+{d})-u^\xi(\xi_j+{d})
=a_{{h_j}+1}-a_{j+1}+\mathrm{O}(e^{-\frac{\alpha\rho}{\epsilon}})\neq 0,\\
&\zeta_{h_j}\geq\xi_j\;\;\Rightarrow u^\zeta(\xi_j-{d})-u^\xi(\xi_j-{d})
=a_{h_j}-a_{j}+\mathrm{O}(e^{-\frac{\alpha\rho}{\epsilon}})\neq 0,
 \end{split}\]
 where ${d}=\frac{\rho}{4\max_i \mu_i}$.
 It follows that \eqref{Linfty-diff} holds also in this case with $\bar{m}_1=\frac{1}{2}(\vert a_{h_j+1}-a{j+1}\vert+\vert a_{h_j}-a{j}\vert)$.

Case $\mathbf{b)}$. We only discuss the case $a_{h_j}=a_j$, the argument for the case $a_{{h_j}+1}=a_{j+1}$ is analogous. If $\zeta_{h_j}\leq\xi_j$
arguing as in $\mathbf{a)}$ we again obtain \eqref{Linfty-diff}. Therefore we assume $\zeta_{h_j}>\xi_j$. We observe that the minimality of $\bar{u}_{h_j}$ and $\bar{u}_j$ together with $a_{h_j}=a_j$ and $a_{{h_j}+1}\neq a_{j+1}$ imply the existence of $\delta>0$ such that
\[\vert\bar{u}_{h_j}(s)-a_{j+1}\vert\geq\delta,\;\;s\in\R.\]
it follows, provided $\epsilon>0$ is small
\[\vert u^\zeta(\xi_j+{d})-a_{j+1}-(u^\xi(\xi_j+{d})-a_{j+1})\vert
=\vert u^\zeta(\xi_j+{d})-a_{j+1}\vert+\mathrm{O}(e^{-\frac{\alpha\rho}{\epsilon}})
\geq\frac{\delta}{2}\]
and we conclude that \eqref{Linfty-diff} holds also in case $\mathbf{b)}$ with $\bar{m}_1=\frac{\delta}{2}$.

Case $\mathbf{c)}$. In this case we necessarily have $h_j=j$, $1\leq j\leq N$ and in turn
$\bar{u}_{h_j}=\bar{u}_j$. The minimality of $\bar{u}_j$ implies
\begin{equation}
\bar{u}_j(s)\neq\bar{u}_j(s+\tau),\;\;s\in\R,\,\tau\neq 0,
\label{no-loop}
\end{equation}
and \eqref{us-aj} yields $\bar{u}_j(0)\notin\{a_j,a_{j+1}\}$. It follows that, given $\omega>0$, there is   $\bar{m}_1(\omega)>0$ such that
\[\vert\bar{u}_j(s)-\bar{u}_j(0)\vert \geq\bar{m}_1(\omega),\;\;\vert s\vert\geq\omega.\]
This and Lemma \ref{est-uxi} yield
\[\vert\zeta_j-\xi_j\vert\geq\epsilon\omega\;\;\Rightarrow\vert u^\zeta(\xi_j)-u^\xi(\xi_j)\vert\geq \bar{m}_1(\omega)+\mathrm{O}(e^{-\frac{\alpha\rho}{\epsilon}}).\]
From 1, 2, 3 we conclude that
\[\vert\zeta-\xi\vert\geq\epsilon\omega\;\;\Rightarrow\;\;\|u^\zeta-u^\xi\|_{L^\infty}\geq \frac{1}{2}\bar{m}_1(\omega).\]
That establishes \eqref{Linftycase}$_1$ with $m_1(\omega)=\frac{1}{2}\bar{m}_1(\omega)$ and $\omega_1$ such that
$m_1(\omega)\leq\bar{m}_1$ for $\omega\in(0,\omega_1]$.
This and $\vert u_x^\xi(x)\vert\leq\frac{C}{\epsilon}$ that follows from \eqref{uxixi}$_3$ imply
\[\vert\zeta-\xi\vert\geq\epsilon\omega\;\;\Rightarrow\;\;\|u^\zeta-u^\xi\|\geq C{(m_1(\omega))}^\frac{3}{2}\epsilon^\frac{1}{2},\]
that yields \eqref{L2case}$_1$.

4. Proof of \eqref{L2case}$_2$. Fix a number $\omega_0>0$. Then, for $\vert\zeta-\xi\vert\leq\omega_0\epsilon$, the estimates in Lemma \ref{est-uxi} apply both to $u^\xi$ and $u^\zeta$ and we have

\begin{equation}
\begin{split}
&u^\zeta(x)-u^\xi(x)=\int_0^1u_\xi^{\xi+t(\zeta-\xi)}dt\cdot(\zeta-\xi)\\
&=-\frac{1}{\epsilon}\sum_j\int_0^1\bar{u}_j^\prime(\frac{x-\xi_j-t(\zeta_j-\xi_j)}{\epsilon})
(\zeta_j-\xi_j)dt+\sum_j\mathrm{O}(e^{-\frac{\alpha\rho}{\epsilon}})(\zeta_j-\xi_j).
\end{split}
\label{uzeta-uxi0}
\end{equation}
 Note that
\[\bar{u}_j^\prime(\frac{x-\xi_j-t(\zeta_j-\xi_j)}{\epsilon})
=\bar{u}_j^\prime(\frac{x-\xi_j}{\epsilon})
-\frac{t}{\epsilon}\int_0^1\bar{u}_j^{\prime\prime}(\frac{x-\xi_j-\tau t(\zeta_j-\xi_j)}{\epsilon})d\tau(\zeta_j-\xi_j)\]
implies
\begin{equation}\begin{split}
&u^\zeta(x)-u^\xi(x)
=-\frac{1}{\epsilon}\sum_j\bar{u}_j^\prime(\frac{x-\xi_j}{\epsilon})(\zeta_j-\xi_j)\\
&+\int_0^1\frac{t}{\epsilon^2}\sum_j\int_0^1\bar{u}_j^{\prime\prime}(\frac{x-\xi_j-\tau t(\zeta_j-\xi_j)}{\epsilon})dtd\tau(\zeta_j-\xi_j)^2
+\sum_j\mathrm{O}(e^{-\frac{\alpha\rho}{\epsilon}})(\zeta_j-\xi_j)
\end{split}
\label{uzeta-uxi}
\end{equation}
Arguing as in the proof of Lemma \ref{est-uxi} we find
\begin{equation}
\frac{1}{\epsilon^2}\int_0^1\bar{u}_j^\prime(\frac{x-\xi_j}{\epsilon})
\cdot\bar{u}_i^\prime(\frac{x-\xi_i}{\epsilon})dx=\frac{\delta_{ij}}{\epsilon}\bar{q}_j^2
+\mathrm{O}(e^{-\frac{\alpha\rho}{\epsilon}}),
\label{ui.ujprime}
\end{equation}
where $\bar{q}_j^2=\int_\R\vert \bar{u}_j(s)\vert^2ds$.
 From \eqref{ui.ujprime} we obtain
 \begin{equation}
 \begin{split}
 &\|\frac{1}{\epsilon}\sum_j\bar{u}_j^\prime(\frac{x-\xi_j}{\epsilon})(\zeta_j-\xi_j)\|\geq
(\frac{\bar{q}_m}{\epsilon^\frac{1}{2}}+\mathrm{O}(e^{-\frac{\alpha\rho}{\epsilon}}))\vert\zeta-\xi\vert,\\
&\|\frac{1}{\epsilon}\sum_j\bar{u}_j^\prime(\frac{x-\xi_j}{\epsilon})(\zeta_j-\xi_j)\|
\leq(\frac{\bar{q}_M}{\epsilon^\frac{1}{2}}+\mathrm{O}(e^{-\frac{\alpha\rho}{\epsilon}}))\vert\zeta-\xi\vert,
\end{split}
\label{uzeta-uxi11}
\end{equation}
where $\bar{q}_m$ and $\bar{q}_M$ are respectively the minimum and the maximum of the $\bar{q}_j$.

From $\vert x-\xi_j-\tau t(\zeta_j-\xi_j)\vert\geq\vert x-\xi_j\vert-\vert\zeta_j-\xi_j\vert$, \eqref{generic} and the assumption $\vert\zeta-\xi\vert\leq\omega_0\epsilon$ it follows
\begin{equation}
\begin{split}
&\vert\bar{u}_j^{\prime\prime}(\frac{x-\xi_j-\tau t(\zeta_j-\xi_j)}{\epsilon})\vert
\leq Ce^{-\frac{\mu_j}{\epsilon}\vert x-\xi_j\vert}e^{\mu_j\omega_0}
\leq Ce^{-\frac{\mu_j}{\epsilon}\vert x-\xi_j\vert},
\end{split}
\label{D2uh}
\end{equation}

\noindent Therefore if $I(x)=\sum_jI_j(x)$ is the integral term in \eqref{uzeta-uxi} we have
\[\begin{split}
&\vert\zeta-\xi\vert\leq\omega_0\epsilon\quad\Rightarrow\\
&\|I_j\|^2\leq\frac{1}{\epsilon^4}C\int_0^1e^{-\frac{2\mu_j}{\epsilon}\vert x-\xi_j\vert}dx(\zeta_j-\xi_j)^4\\
&\leq\frac{C}{\epsilon^3}(\zeta_j-\xi_j)^4
\leq\frac{C\omega_0^2}{\epsilon}(\zeta_j-\xi_j)^2\\
&\|I\|\leq\sum_j\|I_j\|\leq\frac{C\omega_0}{\epsilon^\frac{1}{2}}\vert\zeta-\xi\vert.
\end{split}\]
This \eqref{uzeta-uxi} and \eqref{uzeta-uxi11} yield
\begin{equation}
\begin{split}
&\vert\zeta-\xi\vert\leq\omega_0\epsilon\quad\Rightarrow\\
&\|u^\zeta-u^\xi\|\geq\frac{1}{\epsilon^\frac{1}{2}}
(\bar{q}_m-C\omega+\mathrm{O}(e^{-\frac{\alpha\rho}{\epsilon}})\vert\zeta-\xi\vert
,\\
&\|u^\zeta-u^\xi\|\leq\frac{1}{\epsilon^\frac{1}{2}}
(\bar{q}_M +C\omega+\mathrm{O}e^{-\frac{\alpha\rho}{\epsilon}})\vert\zeta-\xi\vert
,
\end{split}
\label{caseL2}
\end{equation}
and therefore, for $\omega_0=\frac{\bar{q}_m}{4C}$,
we have
\[\begin{split}
&\vert\zeta-\xi\vert\leq\omega_0\epsilon\quad\Rightarrow\\
&\frac{\bar{q}_m}{2\epsilon^\frac{1}{2}}\vert\zeta-\xi\vert\leq\|u^\zeta-u^\xi\|
\leq\frac{2\bar{q}_M}{\epsilon^\frac{1}{2}}\vert\zeta-\xi\vert.
\end{split}\]
 This concludes the proof of \eqref{L2case}$_2$.
The proof of \eqref{uxix-est} is similar.

5. Proof of \eqref{Linftycase}$_2$. The right inequality follows immediately from \eqref{uzeta-uxi0} and the boundedness of $\bar{u}_j^\prime$, $j=1,\ldots,N$. To prove the other inequality fix a number $s_a>0$ that satisfies
\[\vert\bar{u}_i(s_a)-\bar{u}_i(-s_a)\vert\geq\frac{1}{2}\min_j\vert a_{j+1}-a_j\vert,\;\;i=1,\ldots,N.\]
Then there is $s_i\in[-s_a,s_a]$ such that
 \[\vert\bar{u}_i^\prime(s_i)\vert\geq k_a=
\frac{1}{4s_a}\min_j\vert a_{j+1}-a_j\vert.\]
Assume that $h$ satisfies $\vert\zeta_h-\xi_h\vert=\max_j\vert\zeta_j-\xi_j\vert$ and note that, for $x=\xi_h+\epsilon s_h$ we have
\begin{equation}
\begin{split}
&\vert\sum_j\bar{u}_j^\prime(\frac{\xi_h-\xi_j+\epsilon s_h}{\epsilon})(\zeta_j-\xi_j)\vert\\
&\geq (k_a
-\sum_{j\neq h}\vert\bar{u}_j^\prime(\frac{\xi_h-\xi_j+\epsilon s_h}{\epsilon})\vert)\vert\zeta_h-\xi_h\vert
\geq\frac{k_a}{2N^\frac{1}{2}}\vert\zeta-\xi\vert,
\end{split}
\label{uzeta-uxi1}
\end{equation}
where we have observed that $\vert\bar{u}_j^\prime(\frac{\xi_h-\xi_j+\epsilon s_h}{\epsilon})\vert
=\mathrm{O}(e^{-\frac{\alpha\rho}{\epsilon}})$, if $j\neq h$.
On the other hand the boundedness of $\bar{u}_j^{\prime\prime}(s)$  implies
\[\begin{split}
&\vert I_j(\xi_h+\epsilon s_h)\vert\leq \frac{C}{\epsilon^2}\vert\zeta_j-\xi_j\vert^2\quad\Rightarrow\\
&\vert I(\xi_h+\epsilon s_h)\vert\leq \frac{C}{\epsilon^2}\vert\zeta-\xi\vert^2\leq\frac{k_a}{4\epsilon N^\frac{1}{2}}\vert\zeta-\xi\vert,
\end{split}\]
where we have used the hypothesis $\vert\zeta-\xi\vert\leq\epsilon\omega_1$ and assumed
$\omega_1\leq\frac{k_0}{4N^\frac{1}{2}}$. This ,\eqref{uzeta-uxi}, \eqref{uzeta-uxi1} and a proper definition of $C_\pm^1$ conclude the proof of the estimate \eqref{Linftycase}$_2$.
\end{proof}


\section{On the asymptotic behavior of $\bar{u}$.}\label{B}

\begin{lemma}
\label{w}
Assume that $W:\R^m\rightarrow\R$ satisfies ${\mathbf{H}}_2$ and $\bar{u}:(0,+\infty)\rightarrow\R^m$ is a solution of \eqref{newton} that converges to some $a\in A$:  $\lim_{s\rightarrow+\infty}\bar{u}(s)=a$ . Then there exists
\[\lim_{s\rightarrow+\infty}\frac{\bar{u}(s)-a}{\vert\bar{u}(s)-a\vert}=z,\]
and $z$ is an eigenvector of the matrix $W_{uu}(a)$. Moreover 
there are positive constants $\bar{K}, K^+$ and $\mu^+>\mu$  ($\mu^2$ the eigenvalue of  $W_{uu}(a)$ associated to $z$)  and a smooth map $w:(0,+\infty)\rightarrow\R^m$ such that
\[\begin{split}
&\bar{u}(s)-a=\bar{K}ze^{-\mu s}+w(s),\;\;s>0,\\
&\vert w(s)\vert,\,\vert D^kw(s)\vert\leq K^+e^{-\mu^+ s},\;\;s>0,\;\;k=1,2,3
\end{split}\]
A similar statements applies when $\bar{u}$ converges to $a$ as $s\rightarrow-\infty$.
\end{lemma}
\begin{proof}
1. From the assumptions on $\bar{u}$ it follows that $\bar{u}$ lies on the stable manifold of $a$. This implies that, modulus a linear change of coordinates, we can assume that $a=0$ and $\bar{u}$ is a solution of
\begin{equation}
\dot{u}=-D(\mu)u+f(u),
\label{change}
\end{equation}
where $f$ is a smooth maps that satisfies $\vert f(u)\vert\leq C_f\vert u\vert^2$ for some $C_f>0$ and $D(\mu)$ is the diagonal matrix whose elements $\mu_1,\ldots,\mu_m$ are the square roots of the eigenvalues $\mu_1^2,\ldots,\mu_m^2$ of $W_{uu}(a)$ and $u=(u_1,\ldots,u_m)$ is the vector of the components of $u$ on the orthonormal basis composed by the eigenvectors $z_1,\ldots,z_m$ associated to $\mu_1^2,\ldots,\mu_m^2$.
With the change of variable $u=\vert u\vert\nu$, $r=\vert u\vert$ \eqref{change} transforms in the system
\begin{equation}
\begin{split}
&\dot{\nu}=-D(\mu)\nu+(\nu\cdot D(\mu)\nu) \nu+\frac{f(r\nu)}{r}-(\nu\cdot\frac{f(r\nu)}{r})\nu,\\
&\dot{r}=-(\nu\cdot D(\mu)\nu) r+f(r\nu)\cdot\nu.
\end{split}
\label{change1}
\end{equation}

2. There exists $s_0>0$ such that

\begin{equation}
r(s)\leq r(s_0)e^{-\frac{\mu_m}{2}(s-s_0)},\;\;s\geq s_0,\;(\mu_m=\min_h\mu_h).
\label{r-dec}
\end{equation}
 Since $\nu\cdot D(\mu)\nu\geq\mu_m$, \eqref{change1}$_2$ implies $\dot{r}\leq-\mu_mr+C_fr^2$. This and $\lim_{s\rightarrow+\infty}r(s)=0$ yield \eqref{r-dec}.

 3. For $\delta>0$ small set $\mathscr{K}_\delta=\{\nu\in\SF^{m-1}:\nu\not\in\cup_hB_\delta(z_h)\}$ and define
 $V:\SF^{m-1}\rightarrow\R$, $V(\nu)=\frac{1}{2}\nu\cdot D(\mu)\nu$. Then there are $c_\delta>0$ and $r_\delta>0$ such that
 \begin{equation}
\dot{V}(\nu)\leq-c_\delta,\;\;\nu\in\mathscr{K}_\delta,\;r\leq r_\delta.
\label{is neg}
\end{equation}
 We have
 \begin{equation}
 \dot{V}(\nu)=\nu\cdot D(\mu)\dot{\nu}=-\vert D(\mu)\nu\vert^2+(\nu\cdot D(\mu)\nu)^2+F(r,\nu),
 \label{lapunof}
 \end{equation}
 where $\vert F(r,\nu)\vert\leq C_f^\prime r$, for some $C_f^\prime>0$. We have
 $(\nu\cdot D(\mu)\nu)^2\leq\vert D(\mu)\nu\vert^2$ with the sign of strict inequality unless $\nu\in\{z_1,\ldots,z_m\}$. Therefore $\nu\in\mathscr{K}_\delta$ implies
 $-\vert D(\mu)\nu\vert^2+(\nu\cdot D(\mu)\nu)^2\leq-\hat{c}_\delta$ for some $\hat{c}_\delta>0$ and \eqref{lapunof}
  yields
 \[\dot{V}(\nu)\leq-\hat{c}_\delta+C_f^\prime r\leq-\frac{\hat{c}_\delta}{2}=-c_\delta,\;\;r\leq r_\delta=\frac{\hat{c}_\delta}{2C_f^\prime}.\]

4. The existence of a sequence $s_n\rightarrow+\infty$ such that
 \begin{equation}
\lim_{n\rightarrow+\infty}\nu(s_n)=\bar{\nu}\not\in\{z_1,\ldots,z_m\},
\label{-barnu}
\end{equation}
leads to a contradiction.

Fix $\delta\in(0,\bar{\delta}=\min_h\vert\bar{\nu}-z_h\vert)$ so small that
\begin{equation}
\begin{split}
& V(\bar{\nu})>V_\delta^-=\max\{V(\nu):\nu\in\SF^{m-1}\cap(\cup_{V(z_h)<V(\bar{\nu})}\bar{B}_\delta(z_h))
 \},\\
& V(\bar{\nu})<V_\delta^+=\min\{V(\nu):\nu\in\SF^{m-1}\cap(\cup_{V(z_h)>V(\bar{\nu})}\bar{B}_\delta(z_h))
\}.
 \end{split}
\label{barV<}
\end{equation}

this and \eqref{is neg} imply
\begin{equation}
\dot{V}(\nu(s))\leq-c_\delta,\;\;\text{if}\;V(\nu(s))\in[V_\delta^-,V_\delta^+],\;r(s)\leq r_\delta.
\label{forb}
\end{equation}
This contradicts 
 \eqref{-barnu} and \eqref{barV<} that imply

\begin{equation}
V(\nu(s_n))\in[V_\delta^-,V_\delta^+],\;\;n\geq\bar{n}.
\label{barV1}
\end{equation}

for some $\bar{n}$. 
  This concludes the proof of the first part of the lemma.
Indeed the compactness of $\SF^{m-1}$ and 4. implies that for each sequence $s_n\rightarrow+\infty$ there is a subsequence (still denoted $\{s_n\}$) such that 
\[\lim_{n\rightarrow+\infty}\nu(s_n)=z\in\{\pm z_1,\ldots,\pm z_m\},\]
and \eqref{forb} implies that the limit $z$ is independent from the sequence.
\vskip.2cm
 From the existence of the limit
\[\lim_{s\rightarrow+\infty}\frac{\bar{u}(s)-a}{\vert\bar{u}(s)-a\vert}=z,\]
and the smoothness of $W$ it follows that, for $s\geq s_0$, for some $s_0>0$, it results
\begin{equation}\bar{u}(s)-a=r(s)z+v(r(s)),
\label{centerM}
\end{equation}

where $v:(0,r_0]\rightarrow\R^m$, $r_0>0$ a small number, is a smooth map that satisfies
\[\begin{split}
&v(r)\cdot z=0,\;r\in[0,r_0],\\
&\vert v(r)\vert\leq Cr^2,\;r\in[0,r_0].
\end{split}\]

From \eqref{centerM} we have
\[\vert\dot{\bar{u}}\vert^2=\vert\dot{r}\vert^2(1+\vert v^\prime(r)\vert^2).\]
From this, conservation of energy: $\vert\dot{\bar{u}}\vert^2=2W(\bar{u})$, and \eqref{centerM} we get
\begin{equation}
\vert\dot{r}\vert^2(1+\vert v^\prime(r)\vert^2)=2W(a+r z+v(r)).
\label{EQ1}
\end{equation}
Since $W(a)=W_u(a)=0$ it follows that, in a neighborhood of $a$, we have $W(u)=\frac{1}{2}W_{uu}(a)(u-a)\cdot(u-a)+U(u-a)$ where $U$ is a smooth map defined in
a neighborhood of $0\in\R^m$ and $U(q)=\mathrm{O}(\vert q\vert^3)$.
On the basis of these observations, and possibly after a renumbering of the coodinates, we can write
\begin{equation}
2W(a+r z+v(r))=\mu^2r^2+\sum_{h>1}^m\mu_h^2v_h^2(r)+2U(r z+ v(r)),
\label{PotExpression}
\end{equation}
where $\mu_h^2$, $1\leq h\leq m$ are the eigenvalues of $W_{uu}(a_{j+1})$, $\mu=\mu_1$, $z=z_1$ and $v_h$ is the projection of $v$ on the eigenvector $z_h$. Using \eqref{PotExpression} in \eqref{EQ1} after taking the square root we see that
$r(s)$ solves the equation
\begin{equation}
\dot{r}=-\mu r+\Big(\mu r-(\frac{\mu^2r^2+\sum_{h>1}^m\mu_h^2v_h^2(r)+2U(r z+ v(r))}{1+\vert v^\prime(r)\vert^2})^\frac{1}{2}\Big)
=-\mu r+g(r)
\label{EQ2}
\end{equation}
and from the properties of $v$ and $U$ we see that $g$ is a smooth function that satisfies
\begin{equation}
\vert g(r)\vert\leq Cr^2,\;r\leq r_0.
\label{g-est}
\end{equation}
Fix $\bar{s}\in(s_0,+\infty)$ and set $\bar{r}=r(\bar{s})$ and $s=\tau+\bar{s}$. We look at \eqref{EQ2} as a weakly nonlinear equation and try to determine $\rho(\tau)=r(\tau+\bar{s})$ for $s\geq\bar{s}$ as a solution of
\begin{equation}
\rho(\tau)=\bar{r}e^{-\mu\tau}+e^{-\mu\tau}\int_0^\tau e^{\mu t}g(\rho(t))dt,
\label{EQ3}
\end{equation}
Set
\[R_{\bar{r}}=\{\rho:[0,+\infty)\rightarrow\R: \rho(\tau)\leq 2\bar{r}e^{-\mu\tau}\}.\]
$R_{\bar{r}}$ with the distance $d(\rho,\hat{\rho})=\sup_{\tau\in[0,+\infty)}\vert\rho(\tau)-\hat{\rho}(\tau)\vert e^{\mu\tau}$ is a complete metric space. We now show that for $\bar{r}$ small, that is for each $\bar{s}$ sufficiently large, the right hand side of \eqref{EQ3}
 defines a contraction map $\rho\stackrel{\mathcal{K}_{\bar{r}}}{\rightarrow}\tilde{\rho}$ on $R_{\bar{r}}$ and therefore it results
 \[r(\tau+\bar{s})=\rho^*(\tau),\;\tau\in[0,+\infty),\]
 where $\rho^*\in R_{\bar{r}}$ is the fixed point of $\mathcal{K}_{\bar{r}}$. From the definition of $\mathcal{K}_{\bar{r}}$ we have
 \[\vert\tilde{\rho}(\tau)\vert\leq\bar{r}e^{-\mu\tau}+e^{-\mu\tau}C\bar{r}^2\int_0^\infty e^{\mu t}e^{-2\mu t}dt=(\bar{r}+\frac{C}{\mu }\bar{r}^2)e^{-\mu\tau}\]
 and therefore we have that $\tilde{\rho}\in R_{\bar{r}}$ provided
 \begin{equation}
 \bar{r}\leq\frac{\mu}{C}.
 \label{cond1}
 \end{equation}
 We also have
 \[\vert\tilde{\rho}_1(\tau)-\tilde{\rho}_2(\tau)\vert\leq e^{-\mu\tau}\int_0^\tau e^{\mu t}
 \vert g(\rho_1(t))-g(\rho_2(t))\vert dt,\]
 and, using also $g^\prime(r)\leq C_1r$,
 \[\begin{split}
 &\vert g(\rho_1(t))-g(\rho_2(t))\vert\\
 &=\vert\int_0^1g^\prime(\rho_2(t)+\theta(\rho_1(t)-\rho_2(t)))d\theta
 \vert\vert\rho_1(t)-\rho_2(t))\vert\\
 &\leq 2C_1\bar{r}e^{-\mu t}\vert\rho_1(t)-\rho_2(t))\vert.
 \end{split}\]
 It follows
  \[\begin{split}
  &\vert\tilde{\rho}_1(\tau)-\tilde{\rho}_2(\tau)\vert\leq e^{-\mu\tau}2C_1\bar{r}\int_0^\tau e^{\mu t}e^{-\mu t}\vert\rho_1(t)-\rho_2(t))\vert dt\\
  &\leq e^{-\mu\tau}3C_1\bar{r}\int_0^\infty e^{-\mu t}\sup_{\sigma\geq 0}e^{\mu\sigma}\vert\rho_1(\sigma)-\rho_2(\sigma))\vert dt\\
  &= e^{-\mu\tau}2C_1\bar{r} d(\rho_1,\rho_2)\int_0^\infty e^{-\mu t} dt
  = e^{-\mu\tau}2\frac{C_1}{\mu}\bar{r} d(\rho_1,\rho_2),
  \end{split}\]
  and
  \[d(\tilde{\rho}_1,\tilde{\rho}_2)\leq 2\frac{C_1}{\mu}\bar{r}d(\rho_1,\rho_2)\]
  and, using also \eqref{cond1}, we conclude that $\mathcal{K}_{\bar{r}}$ is a contraction on $R_{\bar{r}}$ if
  \[\bar{r}\leq\mu\min\{\frac{1}{C},\frac{1}{2C_1}\}.\]
  From \eqref{EQ3} with $\rho=\rho^*$ we have
  \begin{equation}
  \frac{\rho^*(\tau)}{e^{-\mu\tau}}=\bar{r}+\int_0^\tau e^{\mu t}g(\rho^*(t))dt.
  \label{Lim1}
  \end{equation}
The integral on the right hand side is absolutely convergent. Therefore if we set $\tilde{r}=\int_0^\infty e^{\mu t}g(\rho^*(t))dt$ we can write
 \begin{equation}
  \frac{\rho^*(\tau)}{e^{-\mu\tau}}=\bar{r}+\tilde{r}-\int_\tau^\infty e^{\mu t}g(\rho^*(t))dt.
  \label{Lim2}
  \end{equation}
  From $\vert\rho^*(\tau)\vert\leq2\bar{r}e^{-\mu\tau}$ and \eqref{g-est} it follows
  \[\vert\int_\tau^\infty e^{\mu t}g(\rho^*(t))dt\vert\leq\frac{4C\bar{r}^2}{\mu}e^{-\mu\tau},\]
  and we have
  \begin{equation}
  r(s)=\rho^*(s-\bar{s})=(\bar{r}+\tilde{r})e^{-\mu(s-\bar{s})}+\omega(s)=\bar{K}^+e^{-\mu s}+\omega(s),
  \label{erre}
  \end{equation}
  where $\bar{K}^+=(\bar{r}+\tilde{r})e^{\mu\bar{s}}$ and
  \[\begin{split}
  &\omega(s)=e^{-\mu(s-\bar{s})}\int_{(s-\bar{s})}^\infty e^{\mu t}g(\rho^*(t))dt,\\
  &\vert\omega(s)\vert\leq\frac{4C\bar{r}^2}{\mu}e^{2\mu\bar{s}}e^{-2\mu s},\;\;s\geq\bar{s}.
  \end{split}\]
  If we insert \eqref{erre} into \eqref{centerM} we get
 \begin{equation}
 \begin{split}
 &\bar{u}(s)-a=\bar{K}^+e^{-\mu s}z+w(s),\;\;s\geq\bar{s},\\
 & w(s)=\omega(s)z+v(r(s)).
 \end{split}
 \label{Decay-Est}
 \end{equation}
  From the previous estimate for $\omega$ and $v(r)=\mathrm{O}(r^2)$ it follows
 \begin{equation}
 \vert w(s)\vert\leq K^+e^{-2\mu s},\;\;s\geq\bar{s},
 \label{w-Est}
 \end{equation}
 for some $K^+>0$.
 Actually, by trivially extending the definition of $w$ to the whole $[0,+\infty)$ by setting
 \[w(s)=\bar{u}(s)-a-\bar{K}^+e^{-\mu s}z,\;\;s\in[0,\bar{s})\]
 we can assume that the expression of $\bar{u}(s)-a$ in \eqref{Decay-Est} is valid for all $s\geq 0$.
 Since $w$ is a bounded map, by increasing the constant $K^+$ if necessary, we can also assume that the estimate \eqref{w-Est} is valid for all $s>0$. Similar arguments based on the fact that $D^k{\rho}$ solves the $k$-derivative of \eqref{EQ3} complete the proof.
\end{proof}

\bibliographystyle{plain}

 \end{document}